\documentclass[oneside,10pt,leqno]{amsart}
\usepackage{amssymb,amsmath,latexsym,amscd}

\usepackage{longtable}

\def\gg{\mathfrak{g}}
\def\gh{\mathfrak{h}}
\def\gi{\mathfrak{i}}

\def\gk{\mathfrak{k}}
\def\gl{\mathfrak{l}}
\def\gm{\mathfrak{m}}
\def\gn{\mathfrak{n}}
\def\go{\mathfrak{o}}
\def\gp{\mathfrak{p}}

\def\gs{\mathfrak{s}}
\def\gt{\mathfrak{t}}
\def\gu{\mathfrak{u}}

\def\gz{\mathfrak{z}}

\def\C{\mathbb{C}}

\def\F{\mathbb{F}}

\def\H{\mathbb{H}}

\def\R{\mathbb{R}}

\def\Z{\mathbb{Z}}

\def\cS{\mathcal{S}}

\def\Ad{{\rm Ad}}
\def\ad{{\rm ad}\,}

\renewcommand{\thesection}{\arabic{section}}

\renewcommand{\thetable}{{\large \thesection.\arabic{equation}}}

\newtheorem{lemma}[equation]{Lemma}

\newtheorem{definition}[equation]{Definition}

\def\sideremark#1{\ifvmode\leavevmode\fi\vadjust{\vbox to0pt{\vss
 \hbox to 0pt{\hskip\hsize\hskip1em
\vbox{\hsize2cm\tiny\raggedright\pretolerance10000 
 \noindent #1\hfill}\hss}\vbox to8pt{\vfil}\vss}}} 

\begin{document}

\title{Semisimple Weakly Symmetric Pseudo--Riemannian Manifolds}

\author{Zhiqi Chen}
\address{School of Mathematical Sciences and LPMC \\ Nankai University \\ Tianjin 300071, P.R. China} \email{chenzhiqi@nankai.edu.cn}

\author{Joseph A. Wolf}
\address{Department of Mathematics \\ University of California, Berkeley \\ CA 94720--3840, U.S.A.} \email{jawolf@math.berkeley.edu}

\subjclass[2000]{}


\begin{abstract}
We develop the classification of weakly symmetric pseudo--riemannian manifolds
$G/H$ where $G$ is a semisimple Lie group and $H$ is a reductive subgroup.
We derive the classification from the cases where $G$ is compact, and then
we discuss the (isotropy) representation of $H$ on the tangent
space of $G/H$ and the signature of the invariant pseudo--riemannian metric.
As a consequence we obtain the classification of semisimple weakly
symmetric manifolds of Lorentz signature $(n-1,1)$ and trans--lorentzian
signature $(n-2,2)$.
\end{abstract}

\maketitle

\section{Introduction}\label{sec1}
\setcounter{equation}{0}

There have been a number of important extensions of the theory of
riemannian symmetric spaces.  Weakly symmetric spaces, introduced
by A. Selberg \cite{S1956}, play key roles in number theory, riemannian geometry
and harmonic analysis.  See \cite{W2007}.  Pseudo--riemannian
symmetric spaces,
including semisimple symmetric spaces, play central but complementary roles in
number theory, differential geometry and relativity, Lie group
representation theory and harmonic analysis.  Much of the activity there has
been on the Lorentz cases, which are of particular interest in physics.  Here
we work out the classification of weakly symmetric pseudo--riemannian manifolds
$G/H$ where $G$ is a semisimple Lie group and $H$ is a reductive subgroup.
We do this in a way that allows us to derive the signatures of all
invariant pseudo--riemannian metrics.  (All such metrics are necessarily 
weakly symmetric.)
In particular we obtain explicit listings for invariant pseudo--riemannian 
metrics of riemannian
(Table 5.1), lorentzian (Table 5.2) and trans--lorentzian (Table 5.3) signature.

This treatment of weakly symmetric pseudo--riemannian manifolds is a major
extension of the classical paper of M. Berger \cite{B1957}.  Even in the 
riemannian case it adds new information: the signatures of invariant 
metrics that may be non--riemannian.  The lorentzian case is of course
of physical interest.  And the trans--lorentzian case is related to
conformal and other parabolic structures as described in \cite{CS2009}.

Our analysis in the weakly symmetric setting uses the classifications
of Kr\" amer \cite{K1979}, Brion \cite{B1987}, Mikityuk \cite{M1986}
and Yakimova \cite{Y2005, Y2006} for the weakly symmetric riemannian
manifolds.  We pass from these weakly symmetric riemannian cases
to our weakly symmetric pseudo--riemannian classification by a combination
of semisimple Lie group methods and ideas extending those of
Gray and Wolf \cite{WG1968a, WG1968b}.

To start, we show how a weakly symmetric pseudo--riemannian manifold
$(M,ds^2)$, $M = G/H$ with $G$ semisimple and $H$ reductive in $G$, belongs
to a family of such spaces associated to a compact weakly symmetric
riemannian manifold $M_u = G_u/H_u$\,.  There $G_u$ and $H_u$ are compact
real forms of the complex Lie groups $G_\C$ and $H_\C$\,. More
generally, whenever $G_u$ is a compact connected semisimple Lie group and
$H_u$ is a closed connected subgroup, we have the complexification
$(G_u)_\C/(H_u)_\C$ of $G_u/H_u$\,.
\begin{definition}\label{family-def}
{\rm The  {\bf real form family} of $G_u/H_u$ consists of
$(G_u)_\C/(H_u)_\C$ and all $G_0/H_0$ with the same complexification
$(G_u)_\C/(H_u)_\C$.
} \hfill $\diamondsuit$
\end{definition}

If $G_0/H_0$ is in the real form family of
$G_u/H_u$\,, we have a Cartan involution $\theta$ of $G_0$ that
preserves $H_0$ and $(G_u,H_u)$ is the corresponding compact real
form of $(G_0,H_0)$.  But the point here is that this is reversible:

\begin{lemma}\label{family-desc}
Let $G_u$ be a compact connected semisimple Lie group and
$H_u$ a closed connected subgroup.  Let $\sigma$ be an involutive
automorphism of $G_u$ that preserves $H_u$.  Then there is a unique
$G_0/H_0$ in the real form family of $G_u/H_u$ such that $G_0$ is
simply connected, $H_0$ is connected, and $\sigma = \theta|_{G_u}$
where $\theta$ is the holomorphic extension to $(G_u)_\C$ of a 
Cartan involution of $G_0$ that preserves $H_0$\,.
Up to covering, every space $G_0/H_0$ in the real form family of $G_u/H_u$
is obtained in this way.
\end{lemma}

In Section \ref{sec2} we recall Kr\" amer's classification \cite{K1979}
of the spaces $M_u = G_u/H_u$ for the cases where
$M_u$ is not symmetric but is weakly symmetric with $G_u$ simple.
See (\ref{kraemer-classn}).
Note that in all but two cases there is an ``intermediate'' subgroup
$K_u$\,, where $H_u \subsetneqq K_u \subsetneqq G_u$ with both
$G_u/K_u$ and $K_u/H_u$ symmetric.  In the cases where an intermediate
group $K_u$ is present we work out the real form families in steps,
from $H_u$ to $K_u$ to $G_u$\,, using commuting involution methods of
Cartan, Berger, and Wolf and Grey.  When no intermediate group $K_u$
is available we manage the calculation with some basic information on $G_2$\,,
$Spin(7)$ and $Spin(8)$.

In Section \ref{sec3} we calculate the $H$--irreducible subspaces of the
real tangent space of spaces $M = G/H$ found in Section \ref{sec2}, and in
each of the twelve cases there we work out the possible signatures of the 
$G$--invariant pseudo--riemannian metrics.  The results are gathered in
Table 3.6.

In Section \ref{sec4} we recall the Brion--Mikityuk classification
\cite{B1987, M1986} as formulated by Yakimova \cite{Y2005, Y2006}.
See (\ref{brion-compact}) below.  The exposition is taken from \cite{W2007}.
Those are the cases where $M_u$ is weakly symmetric and irreducible,
$G_u$ is semisimple but not simple, and $G_u/H_u$ is principal.
In this context, $G_u$ semisimple,
``principal'' just means that the center $Z_{H_\C}$
of $H_\C$ is the product of its intersection with the complexifications of
the centers of the simple factors of $G_\C$.  For the first eight of the nine 
cases of (\ref{brion-compact}) we work out the resulting spaces $M = G/H$ 
of the real form family, the $H$--irreducible subspaces of the
real tangent space, and the resulting contributions to the signatures 
of the $G$--invariant pseudo--riemannian metrics.  The results are gathered in
Table 4.12.  The ninth case of (\ref{brion-compact}) is a pattern rather
than a formula; there we obtain the signature information by applying our 
notion of ``riemannian unfolding'' to the information contained in Tables 
3.6 and 4.12.

Finally, in Section \ref{sec5} we extract some signature information
from Berger's
\cite[\S50, Table II on page 157]{B1957}, and combine it with certain
cases from our Tables 3.6 and 4.2, to classify the semisimple 
pseudo--riemannian weakly symmetric spaces of riemannian signature
$(n,0)$, lorentzian signature $(n-1,1)$ and trans--lorentzian signature
$(n-2,2)$.  It is interesting to note the prevalence of riemannian 
signature here.  The examples of signature $(n-2,2)$ are also quite
interesting: they are related to conformal and other parabolic geometries
(\cite{CS2009}). This data is collected in Tables 5.1, 5.2 and 5.3.

Some of the methods here extend classifications of Gray and Wolf 
\cite{WG1968a, WG1968b}, concerning the isotropy representation of $H_0$ on
$\gg_0/\gh_0$ where $\gh_0$ is the fixed point set of a semisimple
automorphism of a semisimple algebra $\gg_0$.  Those papers, however,
are only peripherally concerned with signatures of invariant metrics.  
There also is a small overlap with the papers
\cite{KKPS2016, KKPS2017} of Knop, Kr\"otz, Pecher and Schlichtkrull 
on reductive real spherical pairs, which are oriented toward algebraic
geometry and not concerned with signatures of invariant metrics; we learned 
of those papers when most of this paper was completed.  

\section{Real Form Families for $G_u$ Simple.}\label{sec2}
\setcounter{equation}{0}

For the cases where $M_u$ is a riemannian symmetric space we have the
classification of \' Elie Cartan and its extension by Marcel Berger
\cite{B1957}, which we need not repeat here.
\vfill\pagebreak

For the cases where $M_u$ is not symmetric but is weakly symmetric with
$G_u$ simple, the Kr\" amer classification is given by
\begin{equation}\label{kraemer-classn}
\begin{tabular}{|r|l|l|l|l|} \hline
\multicolumn{5}{| c |}{Weakly Symmetric Coset Spaces of a Compact
        Connected Simple Lie Group} \\
\hline
\multicolumn{1}{|l}{} &
\multicolumn{3}{ c |}{$M_u = G_u/H_u$ weakly symmetric} & \multicolumn{1}{ c |}
{$G_u/K_u$ symmetric} \\
\hline \hline
\multicolumn{1}{|l}{} &
\multicolumn{1}{c}{$G_u$} & \multicolumn{1}{|c}{$H_u$}  &
        \multicolumn{1}{|c}{conditions} &
        \multicolumn{1}{|c|}{$K_u$ with $H_u \subset K_u \subset G_u$} \\ \hline \hline
\multicolumn{1}{|l}{} &
\multicolumn{3}{ c |}{riemannian symmetric spaces with symmetry $s$}
  & \multicolumn{1}{ c |}{($H_u = K_u$)} \\ \hline\hline
\multicolumn{1}{|l}{} &
\multicolumn{4}{ c|}{circle bundles over hermitian symmetric spaces
dual to a non--tube domain:}
  \\ \hline
$(1)$ & $SU(m+n)$ & $SU(m) \times SU(n)$
        & $m > n \geqq 1$ & $S[U(m) \times U(n)]$ \\
$(2)$ & $SO(2n)$  & $SU(n)$ & $n$ odd, $n \geqq 5$
        & $U(n)$ \\
$(3)$ &$E_6$& $Spin(10)$ & & $Spin(10)\times Spin(2)$  \\ \hline \hline
$(4)$ & $SU(2n+1)$ & $Sp(n)$   & $n \geqq 2$
        & $U(2n) = S[U(2n)\times U(1)]$ \\
$(5)$ & $SU(2n+1)$ & $Sp(n) \times U(1)$ & $n \geqq 2$
        & $U(2n) = S[U(2n)\times U(1)]$ \\ \hline \hline
\multicolumn{1}{|l}{} &
\multicolumn{3}{ c }{constant positive curvature spheres:}
  & \multicolumn{1}{ c |}{} \\ \hline
$(6)$ & $Spin(7)$ & $G_2$ & & (there is none) \\
$(7)$ & $G_2$ & $SU(3)$ & & (there is none) \\ \hline \hline
\multicolumn{1}{|l}{} &
\multicolumn{3}{ c }{weakly symmetric spaces of Cayley type:}
  & \multicolumn{1}{ c |}{} \\ \hline
$(8)$ & $SO(10)$ & $Spin(7) \times SO(2)$ & & $SO(8)\times SO(2)$ \\
$(9)$ & $SO(9)$ & $Spin(7)$ & & $SO(8)$ \\
$(10)$ & $Spin(8)$ & $G_2$ & & $Spin(7)$ \\ \hline \hline
$(11)$ & $SO(2n+1)$ & $U(n)$ & $n \geqq 2$ & $SO(2n)$ \\
$(12)$ & $Sp(n)$ & $Sp(n-1) \times U(1)$ & $n \geqq 3$
        & $Sp(n-1) \times Sp(1)$ \\ \hline
\end{tabular}
\end{equation}

In order to deal with entries other than $(6)$ and $(7)$ we rely on

\begin{lemma}\label{kr-symm}
Let $M_u = G_u/H_u$ be one of the entries in {\rm (\ref{kraemer-classn})}
excluding entries $(6)$ and $(7)$, so we have the corresponding symmetric space
$G_u/K_u$ where $H_u \subset K_u \subset G_u$\,.  Let $\sigma$ be an
automorphism of $\gh_u$ that extends to $\gg_u$\,.  Then
$\sigma(\gk_u) = \gk_u$\,.  Further, in the riemannian metric on $M_u$
defined by the negative of the Killing form of $\gg_u$\,, $K_u/H_u$
is a totally geodesic submanifold of $M_u$ and itself is a riemannian
symmetric space.
\end{lemma}

\begin{proof} For entries (1), (2) and (3) of (\ref{kraemer-classn}),
$\gk_u = \gh_u + \gz_{\gg_u}(\gh_u)$, so it is preserved by $\sigma$.
For the other entries (4), (5), (8), (9), (10), (11) and (12), with $\gg_u$
acting as usual on a
real vector space $V$, we proceed as follows: $\dim V = 4n+2, 4+2, 10, 9, 8,
2n+1 \text{ or }4n$, respectively, for entries (4), (5), (8), (9), (10),
(11) and (12).
Let $W$ be the subspace of $V$ on which $[\gh_u,\gh_u]$ acts trivially.  The
action of $H_u$ on $W^\perp$ is $(\R^2,\{1\})$, $(\R^2.U(1))$, $(\R^2,SO(2))$,
$(\R,\{1\})$, $(\R,\{1\})$, $(\R,\{1\})$ or $(\R^4,T)$, respectively, where
$T$ is a circle subgroup of $Sp(1)$.   $W^\perp$ is $H_u$--invariant and
$K_u$ is its $G_u$--stabilizer.  Thus $\sigma(\gk_u) = \gk_u$\,.

For the last statement note that $K_u/H_u$ is a circle $S^1$ for entries
(1), (2) and (3); $S^1 \times SU(2n)/Sp(n)$ for entry (4); $SU(2n)/Sp(n)$
for entry (5); the sphere $S^7$ for entries (8), (9) and (10); $SO(2n)/U(n)$
for entry (11); and the sphere $S^2$ for entry (12).
\end{proof}

We'll run through the cases of (\ref{kraemer-classn}).  When there is an
``intermediate'' group $K_u$\,, we make use of Berger's work \cite{B1957}.
In the other two cases the situation is less complicated and we can work
directly.  Afterwards we will collect the classification of real form
families as the first column in Table 3.6  below.

{\bf Case (1):} $M_u = SU(m+n)/[SU(m)\times SU(n)]$, $m> n \geqq 1$.
Then $\widetilde{M}_u =
SU(m+n)/S[U(m)\times U(n)]$ is a Grassmann manifold.
We start with Berger's classification \cite[\S 50]{B1957}
(Table 2 on page 157).  There
we need only consider the cases $\widetilde{M} = G/K$  where
either (1) $G = SL(m+n;\C)$ and $K = S[GL(m;\C)\times GL(n;\C)]$ or
(2) $G$ is a real form of $SL(m+n;\C)$, $K$ is a real form of
$S[GL(m;\C)\times GL(n;\C)]$, and $K \subset G$.   In these cases
$K$ is not semisimple.  The possibilities are
\begin{equation}\label{b1}
\begin{aligned}
{\rm (i)\,\,} &\widetilde{M} = SL(m+n;\C)/S[GL(m;\C)\times GL(n;\C)]
        \text{ and } M = SL(m+n;\C)/[SL(m;\C)\times SL(n;\C)]\\
{\rm (ii)\,\,} &\widetilde{M} = SL(m+n;\R)/S[GL(m;\R)\times GL(n;\R)]
        \text{ and } M = SL(m+n;\R)/[SL(m;\R)\times SL(n;\R)]\\
{\rm (iii)\,\,} &\widetilde{M} = SL(m'+n';\H)/S[GL(m';\H)\times GL(n';\H)]
        \text{ where } m = 2m' \text{ and } n = 2n'; \text{ and } \\
        &\phantom{XXXXXXXXXXXXXXXXXXXXXXX}
        M = SL(m'+n';\H)/[SL(m';\H)\times SL(n';\H)]\\
{\rm (iv)\,\,} &\widetilde{M} = SU(m-k+\ell,n-\ell+k)/S[U(m-k,k)\times U(n-\ell,\ell)]
        \text{ for } k \leqq m \text{ and } \ell \leqq n; \text{ and }\\
        &\phantom{XXXXXXXXXXXXXXX}
        M = SU(m-k+\ell,n-\ell+k)/[SU(m-k,k)\times SU(\ell,n-\ell)]\\
\end{aligned}
\end{equation}
where $GL(k;\H) := SL(k;\H)\times \R^+$.

{\bf Case (2):} $M_u = SO(2n)/SU(n),\, n \text{ odd, } n \geqq 5$.
Then $\widetilde{M_u} = SO(2n)/U(n)$.
In Berger's classification \cite[\S 50]{B1957}
(Table 2 on page 157) we need only consider the cases
$\widetilde{M} = G/K$ where either (1) $G = SO(2n;\C)$ and $K = GL(n;\C)$ or
(2) $G$ is a real form of $SO(2n;\C)$, $K$ is a real form of $GL(n;\C)$,
and $K \subset G$.  As $K$ is not semisimple the possibilities are
\begin{equation}\label{b2}
\begin{aligned}
{\rm (i)\,\,} &\widetilde{M} = SO(2n;\C)/GL(n;\C) \text{ and }
        M = SO(2n;\C)/SL(n;\C) \\
{\rm (ii)\,\,} &\widetilde{M} = SO^*(2n)/U(k,\ell) \text{ where } 
	k+\ell = n \text{ and } M = SO^*(2n)/SU(k,\ell) \\
{\rm (iii)\,\,} &\widetilde{M} = SO(2k,2\ell)/U(k,\ell)\text{ where } 
	k+\ell = n \text{ and }
        M = SO(2k,2\ell)/SU(k,\ell) \\
{\rm (iv)\,\,} &\widetilde{M} = SO(n,n)/GL(n;\R) \text{ and }
        M = SO(n,n)/SL(n;\R)
\end{aligned}
\end{equation}

{\bf Case (3):} $M_u = E_6/Spin(10)$.
Then $\widetilde{M_u} =E_6/[Spin(10)\times Spin(2)]$.
Again, in \cite[\S 50]{B1957} we need only consider the cases
$\widetilde{M} = G/K$ where either (1) $G = E_{6,\C}$ and
$K = Spin(10;\C)\times Spin(2;\C)$ or (2) $G$ is a real form of $E_{6,\C}$\,,
$K$ is a real form of $Spin(10;\C)\times Spin(2;\C)$, and $K \subset G$.
Berger writes
$
E_6^1 \text{ for } E_{6,C_4} = E_{6(6)}\,,
	E_6^2 \text{ for } E_{6,A_5A_1} = E_{6(2)}\,,
	E_6^3 \text{ for } E_{6,D_5T_1} = E_{6(14)}\text{ and }
	E_6^4 \text{ for } E_{6,F_4} = E_{6(-26)}\,.
$
The possibilities are

\begin{equation}\label{b3}
\begin{aligned}
{\rm (i)\,\,} &\widetilde{M} = E_{6,\C}/[Spin(10;\C)\times Spin(2;\C)]
        \text{ and } M = E_{6,\C}/Spin(10;\C) \\
{\rm (ii)\,\,} &\widetilde{M} = E_6/[Spin(10)\times Spin(2)] \text{ and }
        M = E_6/Spin(10) \\
{\rm (iii)\,\,} &\widetilde{M} = E_{6,C_4}/[Spin(5,5)\times Spin(1,1)] \text{ and }
        M = E_{6,C_4}/Spin(5,5) \\
{\rm (iv)\,\,} &\widetilde{M} = E_{6,A_5A_1}/[SO^*(10)\times SO(2)] \text{ and }
        M = E_{6,A_5A_1}/SO^*(10) \\
{\rm (v)\,\,} &\widetilde{M} = E_{6,A_5A_1}/[Spin(4,6)\times Spin(2)] \text{ and }
        M = E_{6,A_5A_1}/Spin(4,6) \\
{\rm (vi)\,\,} &\widetilde{M} = E_{6,D_5T_1}/[Spin(10)\times Spin(2)] \text{ and }
        M = E_{6,D_5T_1}/Spin(10) \\
{\rm (vii)\,\,} &\widetilde{M} = E_{6,D_5T_1}/[Spin(2,8)\times Spin(2)] \text{ and }
        M = E_{6,D_5T_1}/Spin(2,8) \\
{\rm (viii)\,\,} &\widetilde{M} = E_{6,D_5T_1}/[SO^*(10)\times SO(2)]
        \text{ and } M = E_{6,D_5T_1}/SO^*(10) \\
{\rm (ix)\,\,} &\widetilde{M} = E_{6,F_4}/[Spin(1,9)\times Spin(1,1)] \text{ and }
        M = E_{6,F_4}/Spin(1,9) \\
\end{aligned}
\end{equation}

{\bf Case (4):} $M_u = SU(2n+1)/Sp(n)$.  Then
$\widetilde{M}_u$ is the complex projective space
$SU(2n+1)/S[U(2n)\times U(1)]$, and $K_u/H_u = U(2n)/Sp(n)$.
In \cite[\S 50]{B1957} we need only consider the cases
$\widetilde{M} = G/K$ where either (1) $G = SL(2n+1;\C)$ and
$K = GL(2n;\C)$, or (2) $G$ is a real form of $SL(2n+1;\C)$\,,
$K$ is a real form of $GL(2n;\C)$, and $K \subset G$; and the cases
(3) $K = GL(2n;\C)$ and $H = Sp(n;\C)$, or (4)
$K$ is a real form of $GL(2n;\C)$, $H$ is a real form of $Sp(n;\C)$,
and $H \subset K$.  The possibilities for $\widetilde{M}$ are
$SL(2n+1;\C)/S[GL(2n;\C)\times GL(1;\C)]$,
$SL(2n+1;\R)/S[GL(2n;\R)\times GL(1;\R)]$, and
$SU(2n+1-k,k)/S[U(2n-k,k)\times U(1)]$.
The possibilities for $K/H$ are
$GL(2n;\C)/Sp(n;\C) = [SL(2n;\C)/Sp(n;\C)]\times \C^*,\,\,\,
[SU^*(2n)/Sp(k,\ell)]\times U(1)\,\, (k+\ell = n)$,
$GL(2n;\R)/Sp(n;\R),\,\,\, U(n,n)/Sp(n;\R),\,\,\, \text{ and }\,\,\,
U(2k,2\ell)/Sp(k,\ell)\,\, (k+\ell = n)$.
Fitting these together, the real form family of $M_u = SU(2n+1)/Sp(n)$
consists of
\begin{equation}\label{b4}
\begin{aligned}
{\rm (i)\,\,} &M = SL(2n+1;\C)/Sp(n;\C) \\
{\rm (ii)\,\,} &M = SL(2n+1;\R)/Sp(n;\R) \\
{\rm (iii)\,\,} &M = SU(n+1,n)/Sp(n;\R) \\
{\rm (iv)\,\,} &M = SU(2n+1-2\ell,2\ell)/Sp(n-\ell,\ell)
\end{aligned}
\end{equation}

{\bf Case (5):} $M_u = SU(2n+1)/[Sp(n) \times U(1)]$.  Then $\widetilde{M}_u
=SU(2n+1)/S[U(2n)\times U(1)]$, complex projective space, and
$K_u/H_u = SU(2n)/Sp(n)$.  As before the cases of $\widetilde{M}$ are
$$
\begin{aligned}
&SL(2n+1;\C)/S[GL(2n;\C)\times GL(1;\C)], \quad SL(2n+1;\R)/S[GL(2n;\R)\times GL(1;\R)]  \\
&SU(2n+1-k,k)/S[U(2n-k,k)\times U(1)]
\end{aligned}
$$
The possibilities for $K/H$ are
$$
GL(2n;\C)/[Sp(n;\C)\times \C^*],\quad
GL(2n+1;\R)/[Sp(n;\R)\times \R^*],\quad
U(2n+1-2\ell,2\ell)/[Sp(n-\ell,\ell)\times U(1)]
$$
Fitting these together, the real form family of
$M_u = SU(2n+1)/[Sp(n) \times U(1)]$ consists of
\begin{equation}\label{b5}
\begin{aligned}
{\rm (i)\,\,} &M = SL(2n+1;\C)/[Sp(n;\C) \times \C^*] \\
{\rm (ii)\,\,} &M = SL(2n+1;\R)/[Sp(n;\R) \times \R^+] \\
{\rm (iii)\,\,} &M = SU(n+1,n)/[Sp(n;\R)\times \R^+] \\
{\rm (iv)\,\,} &M = SU(2n+1-2\ell,2\ell)/[Sp(n-\ell,\ell)] \times U(1)]
\end{aligned}
\end{equation}

{\bf Case (6):} $M_u = Spin(7)/G_2$.
Neither $G_2$ nor $Spin(7)$ has an outer
automorphism.  Further, $G_2$ is a non--symmetric maximal subgroup of
$Spin(7)$, so any involutive automorphism of $Spin(7)$
that is the identity on $G_2$ is itself the identity.  Thus the
involutive automorphisms of $Spin(7)$ that preserve $G_2$ have form
$\Ad(s)$ with $s \in G_2$\,.  Now the real
form family of $M_u = Spin(7)/G_2$ consists of
\begin{equation}\label{b6}
\begin{aligned}
{\rm (i)\,\,} &M = Spin(7;\C)/G_{2,\C} \\
{\rm (ii)\,\,} &M = Spin(7)/G_2 \\
{\rm (iii)\,\,} &M = Spin(3,4)/G_{2,A_1A_1}
\end{aligned}
\end{equation}

{\bf  Case (7):} $M_u = G_2/SU(3)$.
$SU(3)$ is a non--symmetric maximal subgroup of
$G_2$\,, so any involutive automorphism of $G_2$ that is the identity on
$SU(3)$ is itself the identity.  Thus the involutive automorphisms of $G_2$
that preserve $SU(3)$ either have form $\Ad(s)$ with $s \in SU(3)$ or
act by $z \mapsto z^{-1}$ on the center $Z_{SU(3)} (\cong \Z_3)$.
Further, $G_{2,A_1A_1}$ is the only noncompact real form of $G_{2,\C}$\,.
Now the real form family of $M_u = G_2/SU(3)$ consists of
\begin{equation}\label{b7}
\begin{aligned}
{\rm (i)\,\,} &M = G_{2,\C}/SL(3;\C) \\
{\rm (ii)\,\,} &M = G_2/SU(3)\\
{\rm (iii)\,\,} &M = G_{2,A_1A_1}/SU(1,2) \\
{\rm (iv)\,\,} &M = G_{2,A_1A_1}/SL(3;\R)
\end{aligned}
\end{equation}

{\bf Case (8):} $M_u = SO(10)/[Spin(7)\times SO(2)]$.
Here $\widetilde{M}_u$ is
the Grassmann manifold $SO(10)/[SO(8)\times SO(2)]$, and $K_u/H_u =
[SO(8)\times SO(2)]/[Spin(7)\times SO(2)]$.  The possibilities
for $\widetilde{M}$, as described by Berger \cite[\S 50]{B1957}
(Table 2 on page 157) are
$$
\begin{aligned}
&SO(10;\C)/[SO(8;\C)\times SO(2;\C)] \\
&SO(9-a,a+1)/[SO(8-a,a)\times SO(1,1)],\quad
	SO(8-a,a+2)/[SO(8-a,a)\times SO(0,2)] \\
&SO(10-a,a)/[SO(8-a,a)\times SO(2,0)], \quad \text{ and }
	SO^*(10)/[SO^*(8)\times SO(2)].
\end{aligned}
$$

To see the possibilities for $K/H$ we must first look carefully at
$SO(8)/Spin(7)$.  Label the Dynkin diagram and simple roots of $Spin(8)$
by
\setlength{\unitlength}{.3 mm}
\begin{picture}(50,12)
\put(4,1){\circle{2}}
\put(3,5){$\psi_1$}
\put(6,1){\line(1,0){13}}
\put(20,1){\circle{2}}
\put(20,5){$\varphi$}
\put(21,0.7){\line(2,-1){13}}
\put(35,-6){\circle{2}}
\put(37,-6){$\psi_3$}
\put(21,0.7){\line(2,1){13}}
\put(35,8){\circle{2}}
\put(37,8){$\psi_2$}
\end{picture}.
Let $\gt$ be the Cartan subalgebra of $\gs\gp\gi\gn(8)$ implicit in that
diagram, and define three $3$--dimensional subalgebras
$$
\gt_1: \psi_2 = \psi_3\,,\,\, \gt_2: \psi_3 = \psi_1\,,\,\,
	\gt_3: \psi_1 = \psi_2\,.
$$
They are the respective Cartan subalgebras of three $\gs\gp\gi\gn(7)$
subalgebras
$$
\gs_1:= \gs\gp\gi\gn(7)_1\,, \gs_2:= \gs\gp\gi\gn(7)_2 \text{ and }
	\gs_3:= \gs\gp\gi\gn(7)_3\,.
$$
$Spin(8)$ has center $Z_{Spin(8)} = \{1,a_1,a_2,a_3\} \cong \Z_2 \times \Z_2$\,,
numbered so that the analytic subgroups $S_i$ for the $\gs_i$ have centers
$Z_{S_i} = \{1,a_i\} \cong \Z_2$\,.  In terms of the Clifford algebra
construction of the spin groups and an orthonormal basis $\{e_j\}$ of
$\R^8$ we may take $a_1 = -1$, $a_2 = e_1e_2\dots e_8$ and
$a_3 = a_1a_2 = -e_1e_2\dots e_8$\,.  Thus $Z_{S_1}$ is the kernel of the
universal covering group projection $\pi: Spin(8) \to SO(8)$.  Note that
$$
\pi(S_1) = SO(7) \text{ and } \pi: S_i \to SO(8) \text{ is an isomorphism
	onto a } Spin(7)\text{--subgroup } \pi S_i \text{ for } i = 2,3.
$$

The outer automorphism group of $Spin(8)$ is given by the permutations of
$\{\psi_1,\psi_2,\psi_3\}$.  It is generated by the triality automorphism
$\tau: \psi_1 \to \psi_2 \to \psi_3 \to \psi_1$\,, equivalently
$\tau: S_1 \to S_2 \to S_3 \to S_1$\,,  equivalently
$\tau: a_1 \to a_2 \to a_3 \to a_1$\,.  It follows that the outer automorphism
group of $SO(8)$ is given by
\setlength{\unitlength}{.3 mm}
\begin{picture}(50,12)
\put(4,1){\circle{2}}
\put(6,1){\line(1,0){13}}
\put(20,1){\circle{2}}
\put(21,0.7){\line(2,-1){13}}
\put(35,-6){\circle{2}}
\put(21,0.7){\line(2,1){13}}
\put(35,8){\circle{2}}
\put(36,-1){$\updownarrow$}
\end{picture}, and the $SO(8)$--conjugacy classes of $Spin(7)$--subgroups of
$SO(8)$ are represented by $\pi S_2$ and $\pi S_3$\,.  It follows that
no $Spin(7)$--subgroup of $SO(8)$ can be invariant under an outer automorphism
of $SO(8)$.  See \cite{V2001} for a detailed exposition.

Let $\sigma$ be an involutive automorphism of $SO(8)$ that preserves the
$Spin(7)$--subgroup $\pi S_2$\,.  As noted just above, $\sigma$ is inner
on $SO(8)$.  $\sigma$ is nontrivial on $\pi S_2$ because $\pi S_2$ is a
non--symmetric maximal connected subgroup.  As $\pi S_2$ is simply
connected it follows that
the fixed point set of $\sigma|_{\pi S_2}$ is connected.  Express
$\sigma = \Ad(s)$.  Then $s^2 = \pm I$, and $s \in \pi S_2$
because $\pi S_2$ is its own normalizer in $SO(8)$.

We may assume $s \in T$ where $T$ is the maximal torus of $SO(8)$ with Lie
algebra $\gt$.  Let $t \in T$ with $\det t = -1$.  Then $\Ad(t)$ is an outer
automorphism of $SO(8)$ so $\pi S'_3 := \Ad(t)(\pi S_2)$ is conjugate of
$\pi S_3$\,.  Compute $\sigma(\pi S'_3)$ = $\Ad(st)(\pi S_2)$ =
$\Ad(ts)(\pi S_2)$ = $\Ad(t)(\pi S_2)$ = $\pi S'_3$\,, so $s \in \pi S'_3$
as above.  According to \cite[Theorem 4]{V2001}
$(\pi S_2 \cap \pi S'_3) = \{\pm I\}G_2$\,, so now $s \in \{\pm I\}G_2$\,.
As $-I \notin G_2$ we conclude $s^2 = I$.

We can replace $s$ by $-s$ if necessary and assume that $s \in G_2$\,.
The group $G_2$ has only one conjugacy class of nontrivial automorphisms.
If $\sigma_{G_2}$ is the identity then $\sigma_{\pi S_2}$ is the identity,
because $G_2$ is a non--symmetric maximal connected subgroup of
$\pi S_2$\,.  But then $\sigma$ is the identity because $\pi S_2$ is a
non--symmetric maximal connected subgroup of $SO(8)$.

Now suppose that $\sigma|_{G_2}$ is not the identity.  Then $\sigma$
leads to real forms $G_{2,A_1A_1}$ of $G_{2,\C}$ and $Spin(3,4)$ of
$Spin(7;\C)$.  Thus we may assume that $s = \bigl ( \begin{smallmatrix}
+I_4 & 0 \\ 0 & -I_4 \end{smallmatrix}\bigr ) \in T$.
In Clifford algebra terms, a unit vector $e$ acts on $\R^8$ by
reflection in the hyperplane $e^\perp$. Thus the  $\pi^{-1}$--image of $s$
is $\{\pm e_5e_6e_7e_8\}$, and $\sigma$ leads to the real form $SO(4,4)$ of
$SO(8;\C)$.

Now we look at the possibilities for $K/H$.  Recall $K_u/H_u =
[SO(8)\times SO(2)]/[Spin(7)\times SO(2)]$.  So $K/H$ must be one of
$$
\begin{aligned}
&[SO(8;\C)\times SO(2;\C)]/[Spin(7;\C)\times SO(2;\C)], \\
&[SO(8)\times SO(2)]/[Spin(7)\times SO(2)], \quad
	[SO(8)\times SO(1,1)]/[Spin(7)\times SO(1,1)], \\
&[SO(4,4)\times SO(2)]/[Spin(3,4)\times SO(2)], \quad
	[SO(4,4)\times SO(1,1)]/[Spin(3,4)\times SO(1,1)].
\end{aligned}
$$
We conclude that the real form family of $SO(10)/[Spin(7)\times SO(2)]$
consists of
\begin{equation}\label{b8}
\begin{aligned}
{\rm (i)\,\,} & M = SO(10;\C)/[Spin(7;\C)\times SO(2;\C)] \\
{\rm (ii)\,\,} & M = SO(10)/[Spin(7)\times SO(2)] \\
{\rm (iii)\,\,} & M = SO(9,1)/[Spin(7,0) \times SO(1,1)] \\
{\rm (iv)\,\,} & M = SO(8,2)/[Spin(7,0)\times SO(0,2)] \\
{\rm (v)\,\,} & M = SO(6,4)/[Spin(4,3) \times SO(2,0)] \\
{\rm (vi)\,\,} & M = SO(5,5)/[Spin(3,4) \times SO(1,1)]
\end{aligned}
\end{equation}

{\bf Case (9):} $M_u = SO(9)/Spin(7)$.
Then $\widetilde{M}_u$ is the sphere
$SO(9)/SO(8)$ and $K_u/H_u = SO(8)/Spin(7)$.  From the considerations
of the case $M_u = SO(10)/[Spin(7)\times SO(2)]$ we see that here,
$\widetilde{M}$ must be one of
$$
SO(9;\C)/SO(8;\C),\quad SO(8-a,a+1)/SO(8-a,a),\quad \text{ or } SO(9-a,a)/SO(8-a,a)
$$
while $K/H$ must be one of
$$
SO(8;\C)/Spin(7;\C), \quad SO(8)/Spin(7),\quad \text{ or }\quad  SO(4,4)/Spin(3,4).
$$
Thus the real form family of $M_u = SO(9)/Spin(7)$ consists of
\begin{equation}\label{b9}
\begin{aligned}
{\rm (i)\,\,} & M = SO(9;\C)/Spin(7;\C) \\
{\rm (ii)\,\,} & M = SO(9)/Spin(7) \\
{\rm (iii)\,\,} & M = SO(8,1)/Spin(7) \\
{\rm (iv)\,\,} & M = SO(5,4)/Spin(3,4) \\
\end{aligned}
\end{equation}

{\bf Case (10):} $M_u = Spin(8)/G_2$\,.
Topologically, $M_u = S^7 \times S^7$, and
$\widetilde{M}_u = Spin(8)/Spin(7) = SO(8)/SO(7) = S^7$ and $K_u/H_u =
SO(7)/G_2 = S^7$.  The possibilities for $\widetilde{M}$ are
$Spin(8;\C)/Spin(7;\C)$, $Spin(8-a,a)/Spin(7-a,a)$ for $0 \leqq a \leqq 7$
and $Spin(8-a,a)/Spin(8-a,a-1)$ for $1 \leqq a \leqq 8$, and
for $K/H$ are $Spin(7;\C)/G_{2,\C}$\,, $SO(7)/G_2$
and $Spin(3,4)/G_{2,A_1A_1}$\,.  Now the real form family of $M_u$ consists of
\begin{equation}\label{b10}
\begin{aligned}
{\rm (i) \,\,} &M = Spin(8;\C)/G_{2,\C} \\
{\rm (ii) \,\,} &M = Spin(8)/G_2 \\
{\rm (iii) \,\,}&M = Spin(7,1)/G_2 \\
{\rm (iv) \,\,} &M = Spin(4,4)/G_{2,A_1A_1} \\
{\rm (v) \,\,} &M = Spin(3,5)/G_{2,A_1A_1}
\end{aligned}
\end{equation}

{\bf Case (11):} $M_u = SO(2n+1)/U(n)$.
Then $\widetilde{M}_u = SO(2n+1)/SO(2n)$ and
$K_u/H_u = SO(2n)/U(n)$.  The possibilities for $\widetilde{M}$ are
$$
\begin{aligned}
&SO(2n+1;\C)/SO(2n;\C), \quad SO(n,n+1)/SO^*(2n)\\
&SO(2n+1;\C)/SO(2n-k,k) \text{ for } 0 \leqq k \leqq 2n,\quad
	SO(2n+1-k,k)/SO(2n-k,k) \text{ for } 0 \leqq k \leqq 2n \\
&SO(2n-k,k+1)/SO(2n-k,k) \text{ for } 0 \leqq k \leqq 2n \\
\end{aligned}
$$
The possibilities for $K/H$ are
$$
\begin{aligned}
&SO(2n;\C)/GL(n;\C),\quad
	SO(2n-2k,2k)/U(n-k,k) \text{ for } 0 \leqq k \leqq n \\
&SO^*(2n)/U(n),\quad
	SO^*(2n)/GL(n/2;\H) \text{ for $n$ even},\quad
	SO(n,n)/GL(n;\R)
\end{aligned}
$$
Putting these together, the real form family of $M_u$ consists of
\begin{equation}\label{b11}
\begin{aligned}
{\rm (i) \,\,} &M = SO(2n+1;\C)/GL(n;\C) \\
{\rm (ii) \,\,} &M = SO(2n+1-2k,2k)/U(n-k,k) \text{ for } 0 \leqq k \leqq n \\
{\rm (iii) \,\,} &M = SO(2n-2k,2k+1)/U(n-k,k) \text{ for } 0 \leqq k \leqq n \\
{\rm (iv) \,\,} &M = SO(n,n+1)/GL(n;\R)
\end{aligned}
\end{equation}

{\bf Case (12):} $M_u = Sp(n)/[Sp(n-1)\times U(1)]$.
Here $\widetilde{M}_u$ is the quaternionic projective space
$Sp(n)/[Sp(n-1)\times Sp(1)]$ and $K_u/H_u$ is
$[Sp(n-1)\times Sp(1)]/[Sp(n-1)\times U(1)] = S^2$.
The possibilities for $\widetilde{M}$ are
$$
\begin{aligned}
&Sp(n;\C)/[Sp(n-1;\C)\times Sp(1;\C)],\quad Sp(n;\R)/[Sp(n-1;\R)\times Sp(1;\R)]\\
&Sp(n-k,k)/[Sp(n-1-k,k)\times Sp(1,0)] \text{ for } 0 \leqq k \leqq n-1 \\
&Sp(n-k,k)/[Sp(n-k,k-1)\times Sp(0,1)] \text{ for } 1 \leqq k \leqq n \\
\end{aligned}
$$
The possibilities for $K/H$ are
$$
\begin{aligned}
&[Sp(n-1;\C)\times Sp(1;\C)]/[Sp(n-1;\C)\times GL(1;\C)] \\
&[Sp(n-1-k,k)\times Sp(1,0)]/[Sp(n-1-k,k)\times U(1,0)]
	\text{ for } 0 \leqq k \leqq n-1 \\
&[Sp(n-k,k-1)\times Sp(0,1)]/[Sp(n-k,k-1)\times U(0,1)]
	\text{ for } 1 \leqq k \leqq n \\
&[Sp(n-1;\R)\times Sp(1;\R)]/[Sp(n-1;\R)\times GL(1;\R)] \\
&[Sp(n-1;\R)\times Sp(1;\R)]/[Sp(n-1;\R)\times U(1)]
\end{aligned}
$$
Now the real form family of $M_u$ consists of
\begin{equation}\label{b12}
\begin{aligned}
{\rm (i) \,\,} &M = Sp(n;\C)/[Sp(n-1;\C)\times GL(1;\C)] \\
{\rm (ii) \,\,} &M = Sp(n-k,k)/[Sp(n-1-k,k)\times U(1,0)]
	\text{ for } 0 \leqq k \leqq n-1 \\
{\rm (iii) \,\,} &M = Sp(n-k,k)/[Sp(n-k,k-1)\times U(0,1)]
	\text{ for } 1 \leqq k \leqq n \\
{\rm (iv) \,\,} &M = Sp(n;\R)/[Sp(n-1;\R)\times GL(1;\R)] \\
{\rm (v) \,\,} &M = Sp(n;\R)/[Sp(n-1;\R)\times U(1)]
\end{aligned}
\end{equation}

As mentioned earlier, all the real form family classification results
of Section \ref{sec2} are tabulated as the first column in Table 3.6 below.

\section{Isotropy Representations and Signature}\label{sec3}
\setcounter{equation}{0}
We will describe the isotropy representations for the weakly
symmetric spaces $M = G/H$ of Section \ref{sec2} using the Bourbaki order for
the simple root system $\Psi = \Psi_G = \{\psi_1, \dots , \psi_\ell\}$ of
$G$.  The result will appear in the twelve sub-headers on Table 3.6, the
consequence for the decomposition of the tangent space will appear in
the second column of Table 3.6, and the resulting possible signatures of
$G$--invariant riemannian metric will be in the third column.  The Bourbaki
order of the simple roots is

{\small
\begin{equation}\label{bourbaki}
\begin{tabular}{|l|l|}\hline
\setlength{\unitlength}{.5 mm}
\begin{picture}(140,15)
\put(10,10){\circle{2}}
\put(8,5){$\psi_1$}
\put(11,10){\line(1,0){13}}
\put(25,10){\circle{2}}
\put(23,5){$\psi_2$}
\put(26,10){\line(1,0){13}}
\put(42,10){\circle*{1}}
\put(45,10){\circle*{1}}
\put(48,10){\circle*{1}}
\put(51,10){\line(1,0){13}}
\put(65,10){\circle{2}}
\put(63,5){$\psi_\ell$}
\put(110,7){($A_\ell$\,, $\ell \geqq 1$)}
\end{picture}
&
\setlength{\unitlength}{.5 mm}
\begin{picture}(140,15)
\put(10,10){\circle{2}}
\put(8,5){$\psi_1$}
\put(11,10){\line(1,0){13}}
\put(25,10){\circle{2}}
\put(23,5){$\psi_2$}
\put(26,10){\line(1,0){13}}
\put(42,10){\circle*{1}}
\put(45,10){\circle*{1}}
\put(48,10){\circle*{1}}
\put(51,10){\line(1,0){13}}
\put(65,10){\circle{2}}
\put(63,5){$\psi_{\ell - 1}$}
\put(66,10.5){\line(1,0){13}}
\put(66,9.5){\line(1,0){13}}
\put(80,10){\circle*{2}}
\put(79,5){$\psi_\ell$}
\put(110,7){($B_\ell$\,, $\ell \geqq 2$)}
\end{picture}
\\
\hline
\setlength{\unitlength}{.5 mm}
\begin{picture}(140,15)
\put(10,10){\circle*{2}}
\put(8,5){$\psi_1$}
\put(11,10){\line(1,0){13}}
\put(25,10){\circle*{2}}
\put(23,5){$\psi_2$}
\put(26,10){\line(1,0){13}}
\put(42,10){\circle*{1}}
\put(45,10){\circle*{1}}
\put(48,10){\circle*{1}}
\put(51,10){\line(1,0){13}}
\put(65,10){\circle*{2}}
\put(63,5){$\psi_{\ell - 1}$}
\put(66,10.5){\line(1,0){13}}
\put(66,9.5){\line(1,0){13}}
\put(80,10){\circle{2}}
\put(79,5){$\psi_\ell$}
\put(110,10){($C_\ell$\,, $\ell \geqq 3$)}
\end{picture}
&
\setlength{\unitlength}{.5 mm}
\begin{picture}(140,20)
\put(10,10){\circle{2}}
\put(8,5){$\psi_1$}
\put(11,10){\line(1,0){13}}
\put(25,10){\circle{2}}
\put(23,5){$\psi_2$}
\put(26,10){\line(1,0){13}}
\put(42,10){\circle*{1}}
\put(45,10){\circle*{1}}
\put(48,10){\circle*{1}}
\put(51,10){\line(1,0){13}}
\put(65,10){\circle{2}}
\put(63,5){$\psi_{\ell - 2}$}
\put(66,9.5){\line(2,-1){13}}
\put(80,3){\circle{2}}
\put(83,16){$\psi_{\ell-1}$}
\put(66,10.5){\line(2,1){13}}
\put(80,17){\circle{2}}
\put(83,2){$\psi_\ell$}
\put(110,7){($D_\ell$\,, $\ell \geqq 4$)}
\end{picture}
\\
\hline
\setlength{\unitlength}{.5 mm}
\begin{picture}(140,13)
\put(10,6){\circle*{2}}
\put(8,1){$\psi_1$}
\put(11,5){\line(1,0){13}}
\put(11,6){\line(1,0){13}}
\put(11,7){\line(1,0){13}}
\put(25,6){\circle{2}}
\put(23,1){$\psi_2$}
\put(110,5){($G_2$)}
\end{picture}
&
\setlength{\unitlength}{.5 mm}
\begin{picture}(140,13)
\put(10,6){\circle{2}}
\put(8,1){$\psi_1$}
\put(11,6){\line(1,0){13}}
\put(25,6){\circle{2}}
\put(23,1){$\psi_2$}
\put(26,5.5){\line(1,0){13}}
\put(26,6.5){\line(1,0){13}}
\put(40,6){\circle*{2}}
\put(38,1){$\psi_3$}
\put(41,6){\line(1,0){13}}
\put(55,6){\circle*{2}}
\put(53,1){$\psi_4$}
\put(110,5){($F_4$)}
\end{picture}
\\
\hline
\setlength{\unitlength}{.5 mm}
\begin{picture}(140,23)
\put(10,15){\circle{2}}
\put(8,18){$\psi_1$}
\put(11,15){\line(1,0){13}}
\put(25,15){\circle{2}}
\put(23,18){$\psi_3$}
\put(26,15){\line(1,0){13}}
\put(40,15){\circle{2}}
\put(38,18){$\psi_4$}
\put(41,15){\line(1,0){13}}
\put(55,15){\circle{2}}
\put(53,18){$\psi_5$}
\put(56,15){\line(1,0){13}}
\put(70,15){\circle{2}}
\put(68,18){$\psi_6$}
\put(40,14){\line(0,-1){13}}
\put(40,0){\circle{2}}
\put(43,0){$\psi_2$}
\put(110,7){($E_6$)}
\end{picture}
&
\setlength{\unitlength}{.5 mm}
\begin{picture}(140,23)
\put(10,15){\circle{2}}
\put(8,18){$\psi_1$}
\put(11,15){\line(1,0){13}}
\put(25,15){\circle{2}}
\put(23,18){$\psi_3$}
\put(26,15){\line(1,0){13}}
\put(40,15){\circle{2}}
\put(38,18){$\psi_4$}
\put(41,15){\line(1,0){13}}
\put(55,15){\circle{2}}
\put(53,18){$\psi_5$}
\put(56,15){\line(1,0){13}}
\put(70,15){\circle{2}}
\put(68,18){$\psi_6$}
\put(71,15){\line(1,0){13}}
\put(85,15){\circle{2}}
\put(83,18){$\psi_7$}
\put(40,14){\line(0,-1){13}}
\put(40,0){\circle{2}}
\put(43,0){$\psi_2$}
\put(110,7){($E_7$)}
\end{picture}
\\
\hline
\setlength{\unitlength}{.5 mm}
\begin{picture}(140,23)
\put(10,15){\circle{2}}
\put(8,18){$\psi_1$}
\put(11,15){\line(1,0){13}}
\put(25,15){\circle{2}}
\put(23,18){$\psi_3$}
\put(26,15){\line(1,0){13}}
\put(40,15){\circle{2}}
\put(38,18){$\psi_4$}
\put(41,15){\line(1,0){13}}
\put(55,15){\circle{2}}
\put(53,18){$\psi_5$}
\put(56,15){\line(1,0){13}}
\put(70,15){\circle{2}}
\put(68,18){$\psi_6$}
\put(71,15){\line(1,0){13}}
\put(85,15){\circle{2}}
\put(83,18){$\psi_7$}
\put(86,15){\line(1,0){13}}
\put(100,15){\circle{2}}
\put(98,18){$\psi_8$}
\put(40,14){\line(0,-1){13}}
\put(40,0){\circle{2}}
\put(43,0){$\psi_2$}
\put(110,7){($E_8$)}
\end{picture}
&
\\
\hline
\end{tabular}
\end{equation}
}

\noindent
where, if there are two root lengths, the black dots indicate the short roots.
We will use the notation
\begin{equation}\label{notation}
\begin{aligned}
&\xi_i: \text{ fundamental highest weight, }\tfrac{2\langle \xi_i,\psi_j\rangle}
	{\langle \psi_j,\psi_j\rangle} = \delta_{i,j} \\
&\pi_\lambda: \text{ irreducible representation of } \gg \text{ of highest
	weight } \lambda \\
&\nu_\lambda: \text{ irreducible representation of } \gk \text{ of highest
	weight } \lambda \\
&\tau_\lambda: \text{ irreducible representation of } \gh \text{ of highest
	weight } \lambda  \\
& \pi_{\lambda,\R}, \nu_{\lambda,\R}, \tau_{\lambda,\R}: \text{ corresponding
	real representations}
\end{aligned}
\end{equation}
Here, if $\pi_\lambda(\gg)$ preserves a real form of the representation
space of $\pi_\lambda$ then $\pi_{\lambda,\R}$ is the representation on
that real form.  Otherwise, $(\pi_\lambda \oplus \pi_{\lambda^*})_\R$
is the representation on the invariant real form of the representation space
of $\pi_\lambda \oplus \pi_{\lambda^*}$\,, where $\pi_{\lambda^*}$ is the
complex conjugate of $\pi_\lambda$\,;  $\nu_{\lambda,\R}$ and
$\tau_{\lambda,\R}$, etc., are defined similarly.  Thus, for example, the
isotropy (tangent space) representations $\nu_{G/K}$ of the compact
irreducible symmetric spaces $G/K$
that correspond to non--tube bounded symmetric domains are

{\small
\begin{equation}\label{isoreps-symm}
\begin{aligned}
&\begin{tabular}{|c|c|c|c|}\hline
$G/K$ & conditions & weights &
	$\nu_{G/K} = (\nu_\lambda\oplus \overline{\nu_\lambda})_\R$ \\
\hline
$\frac{SU(m+n)}{S(U(m)\times U(n))}$ & $m > n \geqq 1$ &
	$\lambda = \xi_{m-1}+\xi_m+\xi_{m+1}$ &
$\left(
\setlength{\unitlength}{0.5mm}
\begin{picture}(54,5)(0,0)
\multiput(1,0)(6,0){4}{\circle{2}}
\multiput(2,0)(10,0){1}{\line(1,0){4}}
\multiput(8,0)(1,0){4}{\line(1,0){0.5}}
\multiput(14,0)(10,0){1}{\line(1,0){4}}
\multiput(35,0)(6,0){4}{\circle{2}}
\multiput(36,0)(10,0){1}{\line(1,0){4}}
\multiput(42,0)(1,0){4}{\line(1,0){0.5}}
\multiput(48,0)(10,0){1}{\line(1,0){4}}
\put(23,0){\makebox(0,0){$\scriptstyle \otimes$}}
\put(27,0){\makebox(0,0){$\times$}}
\put(31,0){\makebox(0,0){$\scriptstyle \otimes$}}
\put(19,3){\makebox(0,0){$\scriptstyle 1$}}
\put(27,3){\makebox(0,0){$\scriptstyle 1$}}
\put(35,3){\makebox(0,0){$\scriptstyle 1$}}
\end{picture}
\right)
\oplus
\left(
\setlength{\unitlength}{0.5mm}
\begin{picture}(54,5)(0,0)
\multiput(1,0)(6,0){4}{\circle{2}}
\multiput(2,0)(10,0){1}{\line(1,0){4}}
\multiput(8,0)(1,0){4}{\line(1,0){0.5}}
\multiput(14,0)(10,0){1}{\line(1,0){4}}
\multiput(35,0)(6,0){4}{\circle{2}}
\multiput(36,0)(10,0){1}{\line(1,0){4}}
\multiput(42,0)(1,0){4}{\line(1,0){0.5}}
\multiput(48,0)(10,0){1}{\line(1,0){4}}
\put(23,0){\makebox(0,0){$\scriptstyle \otimes$}}
\put(27,0){\makebox(0,0){$\times$}}
\put(31,0){\makebox(0,0){$\scriptstyle \otimes$}}
\put(1,3){\makebox(0,0){$\scriptstyle 1$}}
\put(27,3){\makebox(0,0){$\scriptstyle -1$}}
\put(53,3){\makebox(0,0){$\scriptstyle 1$}}
\end{picture}
\right)$ \\
\hline
$\frac{SO(2n)}{U(n)}$ & $n$ odd, $n \geqq 3$ &
	$\lambda = \xi_2 + \xi_n$ &
$ \left(
\setlength{\unitlength}{0.5mm}
\begin{picture}(34,5)(0,0)
\multiput(1,0)(6,0){5}{\circle{2}}
\multiput(2,0)(6,0){2}{\line(1,0){4}}
\multiput(14,0)(1,0){4}{\line(1,0){0.5}}
\multiput(20,0)(10,0){1}{\line(1,0){4}}
\put(29,0){\makebox(0,0){$\scriptstyle \otimes$}}
\put(33,0){\makebox(0,0){$\times$}}
\put(33,3){\makebox(0,0){$\scriptstyle 1$}}
\put(7,3){\makebox(0,0){$\scriptstyle 1$}}
\end{picture}
\right)
\oplus
\left(
\setlength{\unitlength}{0.7mm}
\begin{picture}(34,5)(0,0)
\multiput(1,0)(6,0){5}{\circle{2}}
\multiput(2,0)(10,0){1}{\line(1,0){4}}
\multiput(8,0)(1,0){4}{\line(1,0){0.5}}
\multiput(14,0)(6,0){2}{\line(1,0){4}}
\put(29,0){\makebox(0,0){$\scriptstyle \otimes$}}
\put(33,0){\makebox(0,0){$\times$}}
\put(19,3){\makebox(0,0){$\scriptstyle 1$}}
\put(32,3){\makebox(0,0){$\scriptstyle -1$}}
\end{picture}
\right)
$
\\ \hline
$\frac{E_6}{Spin(10)\times Spin(2)}$ & &
	$\lambda = \xi_5 + \xi_6$ &
$\left(
\setlength{\unitlength}{0.7mm}
\begin{picture}(28,10)(0,3)
\multiput(1,6)(6,0){4}{\circle{2}}
\multiput(2,6)(6,0){3}{\line(1,0){4}}
\put(13,5){\line(0,-1){4}}
\put(13,0){\circle{2}}
\put(23,6){\makebox(0,0){$\scriptstyle \otimes$}}
\put(27,6){\makebox(0,0){$\times$}}
\put(27,9){\makebox(0,0){$\scriptstyle 1$}}
\put(19,9){\makebox(0,0){$\scriptstyle 1$}}
\end{picture}
\right)
\oplus
\left(
\setlength{\unitlength}{0.7mm}
\begin{picture}(28,10)(0,3)
\multiput(1,6)(6,0){4}{\circle{2}}
\multiput(2,6)(6,0){3}{\line(1,0){4}}
\put(13,5){\line(0,-1){4}}
\put(13,0){\circle{2}}
\put(23,6){\makebox(0,0){$\scriptstyle \otimes$}}
\put(27,6){\makebox(0,0){$\times$}}
\put(26,9){\makebox(0,0){$\scriptstyle -1$}}
\put(16,0){\makebox(0,0){$\scriptstyle 1$}}
\end{picture}
\right)
$
\\ \hline
\end{tabular}
\end{aligned}
\end{equation}
}
\hskip -3pt
Here the $\times$ corresponds to the ($1$--dimensional) center of $\gk$,
and with $a$ over the $\times$ we have the unitary character $\zeta^a$
which is the $a^{th}$ power of a basic character $\zeta$ on that center.
We note that

\begin{lemma}\label{h-irred}
Let $G_u/H_u$ be a circle bundle over an irreducible hermitian symmetric
space $G_u/K_u$ dual to a non--tube domain, in other words one of the spaces
{\rm (1)}, {\rm (2)} or {\rm (3)} of {\rm (\ref{kraemer-classn})}.  Let
$\nu_{G/K}$ denote the representation of $K_u$ on the real tangent space
$\gg_u/\gk_u$\,, from {\rm (\ref{isoreps-symm})}.  Then
$\nu_{G/K}|_{H_u}$ is irreducible.
\end{lemma}
\begin{proof} In view of the conditions from (\ref{isoreps-symm}),
$\nu_{\lambda}|_{H_u} \ne \overline{\nu_{\lambda}|_{H_u}}$ with the one
exception of $SU(3)/SU(2)$.  It follows, with that exception, that
$\nu_{G/K}|_{H_u}$ is irreducible and $\tau_{G/H} =
\nu_{G/K}|_{H_u} \oplus \tau_{0,\R}$.
In the case of $SU(3)/SU(2)$, $\dim \gg/\gk = 4$ while $\tau_\lambda$
cannot have a trivial summand in $\gg/\gk$.  If $\tau_\lambda$ reduces on
$\gg/\gk$ it is the sum of two $2$--dimensional
real representations.  But $SU(2)$ does not have a nontrivial
$2$--dimensional real representation: the $2$--dimensional complex
representation of $SU(2)$ is quaternionic, not real. Thus, in the case of
$SU(3)/SU(2)$, again $\nu_{G/K}|_{H_u}$ is irreducible and
$\tau_{G/H} = \nu_{G/K}|_{H_u} \oplus \tau_{0,\R}$\,.
\end{proof}

Now, in the cases of (\ref{isoreps-symm}) and Lemma \ref{h-irred},
the isotropy representations of the corresponding weakly symmetric
spaces involve suppressing the $\times$ and adding a trivial representation,
as follows.

{\small
\begin{equation}\label{isoreps-wksymm}
\begin{aligned}
&\begin{tabular}{|c|c|c|}\hline
$G/H$ & weights &
	$\tau_{G/H} = (\tau_\lambda\oplus \overline{\tau_\lambda})_\R
		\oplus \tau_{0,\R}$ \\
\hline
$\frac{SU(m+n)}{SU(m)\times SU(n))}$ & $\lambda = \xi_{m-1}+\xi_{m+1}$ &
$\left(
\setlength{\unitlength}{0.5mm}
\begin{picture}(54,5)(0,0)
\multiput(1,0)(6,0){4}{\circle{2}}
\multiput(2,0)(10,0){1}{\line(1,0){4}}
\multiput(8,0)(1,0){4}{\line(1,0){0.5}}
\multiput(14,0)(10,0){1}{\line(1,0){4}}
\multiput(35,0)(6,0){4}{\circle{2}}
\multiput(36,0)(10,0){1}{\line(1,0){4}}
\multiput(42,0)(1,0){4}{\line(1,0){0.5}}
\multiput(48,0)(10,0){1}{\line(1,0){4}}
\put(27,0){\makebox(0,0){$\scriptstyle \otimes$}}
\put(19,3){\makebox(0,0){$\scriptstyle 1$}}
\put(35,3){\makebox(0,0){$\scriptstyle 1$}}
\end{picture}
\right)
\oplus
\left(
\setlength{\unitlength}{0.5mm}
\begin{picture}(54,5)(0,0)
\multiput(1,0)(6,0){4}{\circle{2}}
\multiput(2,0)(10,0){1}{\line(1,0){4}}
\multiput(8,0)(1,0){4}{\line(1,0){0.5}}
\multiput(14,0)(10,0){1}{\line(1,0){4}}
\multiput(35,0)(6,0){4}{\circle{2}}
\multiput(36,0)(10,0){1}{\line(1,0){4}}
\multiput(42,0)(1,0){4}{\line(1,0){0.5}}
\multiput(48,0)(10,0){1}{\line(1,0){4}}
\put(27,0){\makebox(0,0){$\scriptstyle \otimes$}}
\put(1,3){\makebox(0,0){$\scriptstyle 1$}}
\put(53,3){\makebox(0,0){$\scriptstyle 1$}}
\end{picture}
\right)
\oplus
\left(
\setlength{\unitlength}{0.5mm}
\begin{picture}(54,5)(0,0)
\multiput(1,0)(6,0){4}{\circle{2}}
\multiput(2,0)(10,0){1}{\line(1,0){4}}
\multiput(8,0)(1,0){4}{\line(1,0){0.5}}
\multiput(14,0)(10,0){1}{\line(1,0){4}}
\multiput(35,0)(6,0){4}{\circle{2}}
\multiput(36,0)(10,0){1}{\line(1,0){4}}
\multiput(42,0)(1,0){4}{\line(1,0){0.5}}
\multiput(48,0)(10,0){1}{\line(1,0){4}}
\put(27,0){\makebox(0,0){$\scriptstyle \otimes$}}
\end{picture}
\right)$
\\
\hline
$\frac{SO(2n)}{SU(n)}$ & $\lambda = \xi_2$ &
$ \left(
\setlength{\unitlength}{0.5mm}
\begin{picture}(27,5)(0,0)
\multiput(1,0)(6,0){5}{\circle{2}}
\multiput(2,0)(6,0){2}{\line(1,0){4}}
\multiput(14,0)(1,0){4}{\line(1,0){0.5}}
\multiput(20,0)(10,0){1}{\line(1,0){4}}
\put(7,3){\makebox(0,0){$\scriptstyle 1$}}
\end{picture}
\right)
\oplus
\left(
\setlength{\unitlength}{0.7mm}
\begin{picture}(27,5)(0,0)
\multiput(1,0)(6,0){5}{\circle{2}}
\multiput(2,0)(10,0){1}{\line(1,0){4}}
\multiput(8,0)(1,0){4}{\line(1,0){0.5}}
\multiput(14,0)(6,0){2}{\line(1,0){4}}
\put(19,3){\makebox(0,0){$\scriptstyle 1$}}
\end{picture}
\right)
\oplus
\left(
\setlength{\unitlength}{0.7mm}
\begin{picture}(27,5)(0,0)
\multiput(1,0)(6,0){5}{\circle{2}}
\multiput(2,0)(10,0){1}{\line(1,0){4}}
\multiput(8,0)(1,0){4}{\line(1,0){0.5}}
\multiput(14,0)(6,0){2}{\line(1,0){4}}
\end{picture}
\right)
$
\\ \hline
$\frac{E_6}{Spin(10)}$ & $\lambda = \xi_5$ &
$\left (
\setlength{\unitlength}{0.7mm}
\begin{picture}(22,10)(0,3)
\multiput(1,6)(6,0){4}{\circle{2}}
\multiput(2,6)(6,0){3}{\line(1,0){4}}
\put(13,5){\line(0,-1){4}}
\put(13,0){\circle{2}}
\put(19,9){\makebox(0,0){$\scriptstyle 1$}}
\end{picture}
\right )
\oplus
\left (
\setlength{\unitlength}{0.7mm}
\begin{picture}(22,10)(0,3)
\multiput(1,6)(6,0){4}{\circle{2}}
\multiput(2,6)(6,0){3}{\line(1,0){4}}
\put(13,5){\line(0,-1){4}}
\put(13,0){\circle{2}}
\put(16,0){\makebox(0,0){$\scriptstyle 1$}}
\end{picture}
\right )
\oplus
\left (
\setlength{\unitlength}{0.7mm}
\begin{picture}(22,10)(0,3)
\multiput(1,6)(6,0){4}{\circle{2}}
\multiput(2,6)(6,0){3}{\line(1,0){4}}
\put(13,5){\line(0,-1){4}}
\put(13,0){\circle{2}}
\end{picture}
\right )
$
\\ \hline
\end{tabular}
\end{aligned}
\end{equation}
}

Now we run through the cases of Section \ref{sec2}.

{\bf Case (1):} $M_u = SU(m+n)/[SU(m)\times SU(n)]$, $m > n \geqq 1$.
Consider the spaces listed in (\ref{b1}).  The first three have form
$SL(m+n;\F)/[SL(m;\F) \times SL(n;\F)]$.  In block form matrices
$\left ( \begin{smallmatrix} a & b \\ c & d \end{smallmatrix} \right )$,
the real tangent space of $G/H$ is given by $b$, $c$ and the (real, complex,
real) scalar matrices in the places of $a$ and $d$.
This says that the irreducible summands of the real isotropy representation
have dimensions $mn$, $mn$ and $1$ for $\F = \R$; $2mn$, $2mn$, $1$ and $1$
for $\F =\C$; and $4mn$, $4mn$ and $1$ for $\F = \H$.
Each $\Ad(H)$--invariant
space $\left ( \begin{smallmatrix} 0 & b \\ 0 & 0 \end{smallmatrix} \right )$
and $\left ( \begin{smallmatrix} 0 & 0 \\ c & 0 \end{smallmatrix} \right )$
is null for the Killing form of $\gg$, but they are paired, so together they
contribute signature $(mnd,mnd)$ to any invariant pseudo--riemannian metric
on $G/H$.  Thus the possibilities for signature of invariant pseudo--riemannian
metrics here are
$$
\begin{aligned}
& SL(m+n;\R)/[SL(m;\R) \times SL(n;\R)]:\,\, (mn+1,mn)\,, (mn,mn+1) \\
& SL(m+n;\C)/[SL(m;\C) \times SL(n;\C)]:\,\, (2mn+1,2mn+1)\,,(2mn,2mn+2)\,,(2mn+2,2mn) \\
& SL(m+n;\H)/[SL(m;\H) \times SL(n;\H)]: \,\, (4mn+1,4mn)\,, (4mn,4mn+1).
\end{aligned}
$$

Now consider the fourth space,
$G/H = SU(m-k+\ell,n-\ell+k)/[SU(m-k,k)\times SU(\ell,n-\ell)]$.
In the notation of (\ref{notation}) and (\ref{isoreps-wksymm}),
the complex tangent space of $G/K$ is the sum of $\ad(\gh)$--invariant
subspaces $\gs_+$ and $\gs_-$, the holomorphic and antiholomorphic tangent
spaces of $G/K$, where $\gh$ acts irreducibly
on $\gs_+$ by $\tau_{\xi_1} \otimes \tau_{\xi_{m+n-1}}$ and on
$\gs_-$ by $\tau_{\xi_{m-1}} \otimes \tau_{\xi_{m+1}}$\,.  The Killing form
$\kappa_\C$ of $\gg_\C$ is null on $\gs_+$ and on $\gs_-$ but pairs them, and
the Killing form $\kappa$ of $\gg$ is the real part of $\kappa_\C$\,.
Now the irreducible summands of the isotropy representation of $H$ on the
real tangent space of $G/H$ have dimensions $2mn$ and $1$.  Note the
signs of certain inner products: $\C^{x,y}\otimes\C^{z,w} = \C^{xz+yw,xw+yz}$.
Thus summand of
dimension $2mn$ contributes $(2(m-k)(n-\ell)+2k\ell, 2(m-k)\ell + 2(n-\ell)k)$
or $(2(m-k)\ell + 2(n-\ell)k, 2(m-k)(n-\ell)+2k\ell)$ to the signature of
any invariant pseudo--riemannian metric on $G/H$. Now the possibilities for
the signature of invariant pseudo--riemannian metrics here are
$$
\begin{aligned}
&(2(mn-m\ell-nk+2k\ell)+1,2(m\ell+nk-2k\ell))\,,\, (2(mn-m\ell-nk+2k\ell),2(m\ell+nk-2k\ell)+1)\,,\\
&(2(m\ell+nk-2k\ell)+1,2(mn-m\ell-nk+2k\ell))\,,\, (2(m\ell+nk-2k\ell),2(mn-m\ell-nk+2k\ell)+1).
\end{aligned}
$$

{\bf Case (2):} $M_u = SO(2n)/SU(n)$.
We now consider the spaces listed in (\ref{b2}).  In the  first and
fourth cases $H$ has form $SL(n;\F)$.  For (i), $\gg$ consists of all
$\left ( \begin{smallmatrix} a & b \\ c & d \end{smallmatrix} \right )$ with
$a + a^t = 0 = d+ d^t$ and $c = b^t$.  The symmetry of $\gg$ over $\gk$ is
$\Ad(J)$ where $J = \left ( \begin{smallmatrix} 0 & I \\ -I & 0
\end{smallmatrix} \right )$ as above, and the $(-1)$--eigenspace $\gs$
of $\Ad(J)$ on $\gg$ consists of all
$\left ( \begin{smallmatrix} a & b \\ b & -a \end{smallmatrix} \right )$ with
$a$ antisymmetric and $b$ symmetric.  Thus the contribution of $\gs$ to the
Killing form of $\gg$ has signature $(n(n-1),n(n-1))$ from real $a$ and pure
imaginary $a$.  For (iv), the matrices are real, so the contribution of
$\gs$ to the Killing form of $\gg$ has signature
$\bigl ( \frac{n(n-1)}{2},\frac{n(n-1)}{2}\bigr )$.
Thus the possibilities for
the signature of invariant pseudo--riemannian metrics here are
$$
\begin{aligned}
& SO(2n;\C)/SL(n;\C):\,\, (n(n-1)+2,n(n-1))\,,\, (n(n-1)+1,n(n-1)+1)
	\,,\, (n(n-1),n(n-1)+2);\\
& SO(n,n)/SL(n;\R):\,\, \bigl (\tfrac{1}{2}n(n-1)+1\,,\,\tfrac{1}{2}n(n-1)
	\bigr ) \,,\,\bigl ( \tfrac{1}{2}n(n-1),\tfrac{1}{2}n(n-1)+1\bigr ).
\end{aligned}
$$

In cases (ii) and (iii) of (\ref{b2}) we argue as above for the last case of
(\ref{b1}).  Note the signs of the inner products:
$\Lambda^2(\C^{k,\ell}) = \C^{a,b}$ where
$a = \tfrac{1}{2}(k(k-1)+\ell(\ell-1))$ and
$b = k\ell$.
Thus the possibilities for the signature of invariant
pseudo--riemannian metrics here are
$$
\begin{aligned}
& SO^*(2n)/SU(k,\ell):\,\, (k(k-1)+\ell(\ell-1)+1, 2k\ell),
	(k(k-1)+\ell(\ell-1), 2k\ell+1), \\
& \phantom{SO^*(2n)/SU(k,\ell):\,\,i}(2k\ell+1,k(k-1)+\ell(\ell-1)),
	(2k\ell, k(k-1)+\ell(\ell-1)+1); \\
& SO(2k,2\ell)/SU(k,\ell):\,\, (k(k-1)+\ell(\ell-1))+1, 2k\ell),                (k(k-1)+\ell(\ell-1)), 2k\ell+1), \\
& \phantom{SO^*(2m)/SU(m,n):\,\,(m}(2k\ell+1,k(k-1)+\ell(\ell-1)),
        (2k\ell, k(k-1)+\ell(\ell-1)+1).
\end{aligned}
$$

{\bf Case (3):} $M_u = E_6/Spin(10)$.
We now consider the spaces $M = G/H$ listed in (\ref{b3}).  The representation
of $\gh$ on the complexified tangent space of $M$ is the sum of its two
half spin representations, whose spaces $\gs_\pm$  are null under the Killing
form but are paired.  In case (i) this means that the contribution of
$\gs = \gs_+ + \gs_-$ to the (real) Killing form is $(32,32)$.  In the
riemannian cases where $H = Spin(10)$ the contribution is $(32,0)$ or
$(0,32)$.  In the other four cases for which $H$ has form $Spin(a,b)$ the
contribution is $(16,16)$, because those half spin representations of $H$
are real. Finally, in the two cases where $H = SO^*(10)$, each half spin
representation restricts on the maximal compact subgroup $G_u\cap H\cong U(5)$
of $H$ to the sum of three irreducible representations,
$(
\setlength{\unitlength}{0.4mm}
\begin{picture}(27,5)(0,0)
\multiput(1,0)(6,0){4}{\circle{2}}
\put(21,-2){$\times$}
\put(22,2){$\scriptstyle a$}
\multiput(2,0)(6,0){3}{\line(1,0){4}}
\put(7,3){\makebox(0,0){$\scriptstyle 1$}}
\end{picture}
)
\oplus
(
\setlength{\unitlength}{0.4mm}
\begin{picture}(27,5)(0,0)
\multiput(1,0)(6,0){4}{\circle{2}}
\put(21,-2){$\times$}
\put(22,2){$\scriptstyle b$}
\multiput(2,0)(6,0){3}{\line(1,0){4}}
\put(19,3){\makebox(0,0){$\scriptstyle 1$}}
\end{picture}
)
\oplus
(
\setlength{\unitlength}{0.4mm}
\begin{picture}(27,5)(0,0)
\multiput(1,0)(6,0){4}{\circle{2}}
\put(21,-2){$\times$}
\put(22,2){$\scriptstyle c$}
\multiput(2,0)(6,0){3}{\line(1,0){4}}
\end{picture}
)
$
with $a = 3$, $b = 1$ and $c = 5$, or its dual.  The signature of the
Killing form of $\gg$ on the real tangent space of $G/K$ is
$(20,12)$ both for case (iv) and for case (viii).
Thus the possibilities for the signature of invariant
pseudo--riemannian metrics here are
$$
\begin{aligned}
& E_{6,\C}/Spin(10;\C):\,\, (34,32), (33,33), (32,34);\\
& \text{cases $H = Spin(10)$:}\,\, (17,16)\,,\, (16,17); \\
& \text{cases $H = SO^*(10)$:}\,\, (21,12)\,,\,(20,13)\,,\, (13,20)\,,\,(12,21).
\end{aligned}
$$

{\bf Case (4):} $M_u = SU(2n+1)/Sp(n)$.
Consider the four spaces listed in (\ref{b4}).  The isotropy representation of
$\gh$ on the real tangent space of $G/H$ is the sum of the isotropy
representation $\tau_{K/H}$ on the real tangent space of $K/H$ and the
restriction $\nu_{G/K}|_H$
of the isotropy representation of $K$ on the real tangent space of
$G/K$.  Thus the signatures of the Killing form on the minimal nondegenerate
summands in the real isotropy representation are as the second column
of Table 3.6, and the possible signatures of invariant pseudo--riemannian
metric on $G/H$, are given by the third column there.

{\bf Case (5):} $M_u = SU(2n+1)/[Sp(n) \times U(1)]$.
The four spaces listed in (\ref{b5}) are minor variations on those of
(\ref{b4}).  The commutative factor of $H$ is central in $K$, where it
delivers a trivial factor in the representation of $\gh$ on $\gk/\gh$
and the identity character $\chi$, a nontrivial rotation $\rho$, or a
dilation $\delta$ plus $1/\delta$, in the representation of $\gh$ on $\gg/\gk$.
In this notation, the representation of $\gh$ on
the real tangent space of $G/H$, and the signature of the restriction
of the Killing form of $\gg$ there, are listed in Table 3.6.

{\bf Case (6):} $M_u = Spin(7)/G_2$.
In the three cases of (\ref{b6}), the representation $\tau_{\xi_1}$ of
$\gh$ on the real
tangent space of $G/H$, the signature of the Killing form of $\gg$ on
that tangent space, and the possible signatures of invariant
pseudo--riemannian metric, are as listed in Table 3.6.

{\bf Case (7):} $M_u = G_2/SU(3)$.
In the four cases of (\ref{b7}), the representation of $\gh$ on the real
tangent space of $G/H$, the signature of the Killing form of $\gg$ on
that tangent space, and the possible signatures of invariant
pseudo--riemannian metric, are given by the sum of the vector representation
$\tau_{\xi_1}$ and its dual $\tau_{\xi_2}$\,, and listed in Table 3.6.

{\bf Case (8):} $M_u = SO(10)/[Spin(7)\times SO(2)]$.
We run through the cases of (\ref{b8}).  The representation of $K_u$ on
the real tangent space of the Grassmannian $SO(10)/[SO(8)\times SO(2)]$
remains irreducible on $Spin(7)\times SO(2)$, and the representation of
$H_u$ on the tangent space of $K_u/H_u = [SO(8)\times SO(2)]/[Spin(7)
\times SO(2)] = S^7$ is just the vector representation.  In the cases of
(\ref{b8}), there is no further decomposition when the identity component
of the center of $H$ is a compact. But it splits when that component is
noncompact.  Thus the isotropy representation of $\gh$ on the real tangent
space of $G/H$, the signature of the Killing form on the minimal nondegenerate
summands in the real isotropy representation, and the possible signatures
of invariant pseudo--riemannian metrics on $G/H$, are as listed in Table 3.6.

{\bf Case (9):} $M_u = SO(9)/Spin(7)$.
The cases of (\ref{b9}) are essentially the same as those of (\ref{b8}),
but with the central subgroup of $H$ removed and with $\tau_{\xi_3}$ no
longer tensored with a $2$--dimensional representation.  The isotropy
representation of $\gh$ on the real tangent space of $G/H$, the signature
of the Killing form on the minimal nondegenerate summands in the
real isotropy representation, and the possible signatures of invariant
pseudo--riemannian metric on $G/H$, follow immediately as listed in
Table 3.6.

{\bf Case (10):} $M_u = Spin(8)/G_2$.
In the cases of (\ref{b10}), the representation of $H$ on the tangent space
of $G/H$ is the sum of two copies of the $7$--dimensional representation
$\tau_{\xi_1}$ of $G_2$\,.  The isotropy representation
of $\gh$ on the real tangent space of $G/H$, the signature of the
Killing form on the minimal nondegenerate summands in the
real isotropy representation, and the possible signatures of invariant
pseudo--riemannian metric on $G/H$, are listed in Table 3.6.

{\bf Case (11):} $M_u = SO(2n+1)/U(n)$.
The representation of $H$ on the real tangent space of $G/K = SO(2n+1)/SO(2n)$
is the restriction $(\tau_{\xi_2} \otimes (\zeta \oplus \overline{\zeta})_\R)$
of the vector representation of $SO(2n)$, and on the real tangent space of
$K/H = SO(2n)/U(n)$ is $((\tau_{\xi_2}\otimes \chi) \oplus
(\tau_{\xi_{n-2}} \otimes \overline{\chi}))_\R$ as indicated in the second
line of (\ref{isoreps-symm}).  Now the isotropy representation
of $\gh$ on the real tangent space of $G/H$, the signature of the
Killing form on the minimal nondegenerate summands in the
real isotropy representation, and the possible signatures of invariant
pseudo--riemannian metric on $G/H$, are given as stated below.

{\bf Case (12):} $M_u = Sp(n)/[Sp(n-1)\times U(1)]$.  The representation of
$H$ on the real tangent space of $G/K = Sp(n)/[Sp(n-1)\times Sp(1)]$ is
the restriction $\tau_{\xi_1,\R}\otimes
((\zeta_{+1}\oplus\zeta_{-1})_\R\oplus(\zeta_{-1}\oplus\zeta_{+1})_\R)$
of the representation
$\tau_{\xi_1,\R}\otimes (\tau_{\xi_n} \oplus \tau_{\xi_n})_\R$ of $K$.
The representation of $H$ on the real tangent space of $K/H$ is trivial on
$Sp(n-1)$ and is $(\zeta_{+1}\oplus\zeta_{-1})_\R$ on $U(1)$.  The results are
listed below in Table 3.6.  There ``metric--irreducible'' means minimal
subspace nondegenerate for the Killing form of $\gg$.

\addtocounter{equation}{1}
{\tiny
\begin{longtable}{|l|l|l|}
\caption*{\large \bf Table \thetable \quad Weakly Symmetric Pseudo--Riemannian $G/H$, $G$ simple, $H$ reductive}\\
\hline
{\normalsize $G/H$} & {\normalsize metric--irreducibles} & {\normalsize metric signatures} \\
\hline
\hline
\endfirsthead
\multicolumn{3}{l}{\textit{table continued from previous page $\dots$}} \\
\hline
{\normalsize $G/H$} & {\normalsize metric--irreducibles} & {\normalsize metric signatures} \\
\hline
\hline
\endhead
\hline \multicolumn{3}{r}{\textit{$\dots$ table continued on next page}} \\
\endfoot
\hline
\endlastfoot
\multicolumn{3}{|c|} {\footnotesize \bf (1) Real Form Family of $SU(m+n)/[SU(m)\times SU(n)], m> n \geqq 1$; $\tau = ((\tau_{\xi_1}\otimes\tau_{\xi_{m+n-1}})
       \oplus (\tau_{\xi_{m-1}}\otimes\tau_{\xi_{m+1}}))_\R
       \oplus \tau_{0,\R}$}\\
\hline
{\footnotesize $\tfrac{SL(m+n;\C)}{SL(m;\C)\times SL(n;\C)}$}
	& \begin{tabular}{l} $(2mn,2mn)$, $(1,0)$, $(0,1)$ \end{tabular}
	& \begin{tabular}{l} $(2mn+2,2mn)$, $(2mn+1,2mn+1)$, $(2mn,2mn+2)$ \end{tabular} \\
\hline
{\footnotesize $\tfrac{SL(m+n;\R)}{SL(m;\R)\times SL(n;\R)}$}
	& \begin{tabular}{l} $(mn,mn)$, $(1,0)$ \end{tabular}
	& \begin{tabular}{l} $(mn+1,mn)$, $(mn,mn+1)$  \end{tabular} \\
\hline
{\footnotesize $\tfrac{SL(m+n;\H)}{SL(m;\H)\times SL(n;\H)}$}
	& \begin{tabular}{l} $(4mn,4mn)$, $(1,0)$ \end{tabular}
	& \begin{tabular}{l} $(4mn+1,4mn)$, $(4mn,4mn+1)$ \end{tabular} \\
\hline
{\footnotesize $\tfrac{SU(m-k+\ell,n-\ell+k)}{SU(m-k,k)\times SU(\ell,n-\ell)}$ }
& \begin{tabular}{l}
	$(2m\ell+2nk-4k\ell, 2mn-2m\ell-2nk+4k\ell)$ \\
	$(0,1)$ \end{tabular}
& \begin{tabular}{l}
	$(2m\ell+2nk-4k\ell+1, 2mn-2m\ell-2nk+4k\ell)$ \\
	$(2m\ell+2nk-4k\ell, 2mn-2m\ell-2nk+4k\ell+1)$ \end{tabular} \\
\hline
\hline
\multicolumn{3}{|c|} {\footnotesize \bf (2) Real Form Family of $SO(2n)/SU(n)$;
	$\tau = (\tau_{\xi_2} \oplus \tau_{\xi_{n-2}})_\R
	\oplus \tau_{0,\R}$}\\
\hline
$SO(2n;\C)/SL(n;\C)$ & \begin{tabular}{l} $(n(n-1),n(n-1))$, $(1,0)$, $(0,1)$ \end{tabular}
	& \begin{tabular}{l} $(n(n-1)+2,n(n-1))$ \\
		$(n(n-1)+1,n(n-1)+1)$ \\
		$(n(n-1),n(n-1)+2)$ \end{tabular}\\
\hline
$SO^*(2n)/SU(k,\ell)$
	& \begin{tabular}{l} $(k(k-1)+\ell(\ell-1), 2k\ell)$, $(0,1)$ \end{tabular}
	& \begin{tabular}{l} $(k(k-1)+\ell(\ell-1)+1, 2k\ell)$ \\
          $(k(k-1)+\ell(\ell-1), 2k\ell+1)$ \\
	  $(2k\ell+1,k(k-1)+\ell(\ell-1))$ \\
	  $(2k\ell, k(k-1)+\ell(\ell-1)+1)$
	  \end{tabular}
	\\
\hline
$SO(2k,2\ell)/SU(k,\ell)$
	& \begin{tabular}{l} ($(k(k-1)+\ell(\ell-1),2k\ell)$, $(0,1)$ \end{tabular}
	&\begin{tabular}{l} $(k(k-1)+\ell(\ell-1)+1, 2k\ell)$\\
         $(k(k-1)+\ell(\ell-1), 2k\ell+1)$\\
         $(2k\ell+1,k(k-1)+\ell(\ell-1))$\\
        $(2k\ell, k(k-1)+\ell(\ell-1)+1)$ \end{tabular} \\
\hline
$SO(n,n)/SL(n;\R)$
	& \begin{tabular}{l} $(\frac{1}{2}n(n-1),\frac{1}{2}n(n-1))$, $(1,0)$ \end{tabular}
	& \begin{tabular}{l}  $(\frac{1}{2}n(n-1)+1,\frac{1}{2}n(n-1))$ \\
		$(\frac{1}{2}n(n-1),\frac{1}{2}n(n-1)+1)$  \end{tabular}  \\
\hline
\hline
\multicolumn{3}{|c|} {\footnotesize \bf (3) Real Form Family of $E_6/Spin(10)$;
	$\tau = (\tau_{\xi_4} \oplus \tau_{\xi_5})_\R \oplus \tau_{0,\R}$}\\
\hline
$E_{6,\C}/Spin(10;\C)$ &  \begin{tabular}{l} $(32,32)$,\, $(1,0)$,\, $(0,1)$ \end{tabular}
	&\begin{tabular}{l} $(34,32)$, $(33,33)$, $(32,34)$ \end{tabular} \\
\hline
$E_6/Spin(10)$ & \begin{tabular}{l} $(0,32)$\,, $(0,1)$ \end{tabular}
	& \begin{tabular}{l} $(33,0)$\,, $(32,1)$\,, $(1,32)$, $(0,33)$ \end{tabular} \\
\hline
$E_{6,C_4}/Spin(5,5)$ & \begin{tabular}{l} $(16,16)$\,, $(1,0)$ \end{tabular}
	& \begin{tabular}{l} $(17,16)$\,, $(16,17)$ \end{tabular}\\
\hline
$E_{6,A_5A_1}/SO^*(10)$ & \begin{tabular}{l} $(20,12)$\,, $(0,1)$ \end{tabular}
	& \begin{tabular}{l} $(21,12)$\,, $(20,13)$\,, $(13,20)$\,, $(12,21)$ \end{tabular} \\
\hline
$E_{6,A_5A_1}/Spin(4,6)$ &  \begin{tabular}{l} $(16,16)$\,, $(0,1)$  \end{tabular}
& \begin{tabular}{l} $(17,16)$\,, $(16,17)$ \end{tabular} \\
\hline
$E_{6,D_5T_1}/Spin(10)$ & \begin{tabular}{l} $(32,0)$\,, $(0,1)$ \end{tabular}
	& \begin{tabular}{l} $(33,0)$\,, $(32,1)$\,, $(1,32)$, $(0,33)$ \end{tabular} \\
\hline
$E_{6,D_5T_1}/Spin(2,8)$ & \begin{tabular}{l} $(16,16)$\,, $(0,1)$ \end{tabular}
& \begin{tabular}{l} $(17,16)$\,, $(16,17)$ \end{tabular} \\
\hline
$E_{6,D_5T_1}/SO^*(10)$ & \begin{tabular}{l} $(12,20)$\,, $(0,1)$ \end{tabular}
	& \begin{tabular}{l} $(21,12)$\,, $(20,13)$\,, $(13,20)$\,, $(12,21)$ \end{tabular} \\
\hline
$E_{6,F_4}/Spin(1,9)$ & \begin{tabular}{l} $(16,16)$\,, $(1,0)$  \end{tabular}
 & \begin{tabular}{l} $(17,16)$\,, $(16,17)$ \end{tabular} \\
\hline
\hline
\multicolumn{3}{|c|} {\footnotesize \bf (4) Real Form Family of $SU(2n+1)/Sp(n)$;
	$\tau = (\tau_1 \oplus \tau_1)_\R
	\oplus (\tau_2 \oplus \tau_2)_\R \oplus \tau_{0,\R}$}\\
\hline
$SL(2n+1;\C)/Sp(n;\C)$ & \begin{tabular}{l} $(1,0)$, $(0,1)$ \\ $(4n,4n)$ \\
                $(2n^2-n-1, 2n^2-n-1)$ \end{tabular}
	& \begin{tabular}{l} $(2n^2+3n+1, 2n^2+3n-1)$\\
	 $(2n^2+3n, 2n^2+3n)$\\ $(2n^2+3n-1, 2n^2+3n+1)$  \end{tabular}\\
\hline
$SL(2n+1;\R)/Sp(n;\R)$ & \begin{tabular}{l} $(1,0)$ \\ $(2n,2n)$ \\
         $(n^2-1,n^2-n)$ \end{tabular} &
        \begin{tabular}{l} $(n^2+n+1,n^2+2n-1)$ \\ $(n^2+2n,n^2+n)$ \\
                $(n^2+n,n^2+2n)$ \\ $(n^2+2n-1,n^2+n+1)$ \end{tabular} \\
\hline
$SU(n+1,n)/Sp(n;\R)$ & \begin{tabular}{l} $(0,1)$ \\ $(2n, 2n)$ \\
        $(n^2-n, n^2-1)$  \end{tabular} &
        \begin{tabular}{l} $(n^2+n+1,n^2+2n-1)$ \\ $(n^2+2n,n^2+n)$ \\
                $(n^2+n,n^2+2n)$ \\ $(n^2+2n-1,n^2+n+1)$
        \end{tabular} \\
\hline
{\footnotesize $\tfrac{SU(2n+1-2\ell,2\ell)}{Sp(n-\ell,\ell)}$} &
        \begin{tabular}{l} $(0,1)$ \\ $(4\ell,4n-4\ell)$ \\
                $(4n\ell-4\ell^2,2n^2-4n\ell+4\ell^2-n-1)$ \end{tabular} &
        {\tiny
        \begin{tabular}{l}
                $(4n\ell-4\ell^2+4\ell+1,2n^2-4n\ell+4\ell^2+3n-4\ell-1)$ \\
                $(4n\ell-4\ell^2+4n-4\ell+1,2n^2-4n\ell+4\ell^2+4\ell-n-1)$ \\
                $(2n^2-4n\ell+4\ell^2+3n-4\ell,4n\ell-4\ell^2+4\ell)$ \\
                $(2n^2-4n\ell+4\ell^2+4\ell-n,4n\ell-4\ell^2+4n-4\ell)$ \\
                $(4n\ell-4\ell^2+4\ell,2n^2-4n\ell+4\ell^2+3n-4\ell)$ \\
                $(4n\ell-4\ell^2+4n-4\ell,2n^2-4n\ell+4\ell^2+4\ell-n)$ \\
                $(2n^2-4n\ell+4\ell^2+3n-4\ell-1,4n\ell-4\ell^2+4\ell+1)$ \\
                $(2n^2-4n\ell+4\ell^2+4\ell-n-1,4n\ell-4\ell^2+4n-4\ell+1)$
                \end{tabular} }\\
\hline
\hline
\multicolumn{3}{|c|} {\footnotesize \bf (5) Real Form Family of $SU(2n+1)/[Sp(n)\times  U(1)]$;
  $\tau = ((\tau_1 \otimes \zeta_1)\oplus (\tau_1 \otimes \zeta_{-1}))_\R
  \oplus ((\tau_2 \otimes \zeta_2) \oplus (\tau_2 \otimes \zeta_{-2}))_\R$}\\
\hline
{\footnotesize $\tfrac{SL(2n+1;\C)}{Sp(n;\C)\times\C^*}$} &
          \begin{tabular}{l} $(4n,4n), (2n^2-n-1, 2n^2-n-1)$ \end{tabular}
        &   \begin{tabular}{l} $(2n^2+3n-1, 2n^2+3n-1)$ \end{tabular} \\
\hline
{\footnotesize $\tfrac{SL(2n+1;\R)}{Sp(n;\R)\times \R^*}$} &
        \begin{tabular}{l} $(n^2-1,n^2-n), (2n,2n)$ \end{tabular} &
        \begin{tabular}{l} $(n^2+2n-1,n^2+n), (n^2+n, n^2+2n-1)$ \end{tabular} \\
\hline
{\footnotesize $\tfrac{SU(n+1,n)}{Sp(n;\R)\times U(1)}$} &
	 \begin{tabular}{l} $(n^2-n, n^2-1), (2n, 2n)$ \end{tabular} &
        \begin{tabular}{l}  $(n^2+n,n^2+2n-1), (n^2+2n-1, n^2+n)$ \end{tabular}\\
\hline
{\footnotesize $\frac{SU(2n+1-2\ell,2\ell)}{Sp(n-\ell,\ell)\times U(1)}$} &
        \begin{tabular}{l} $(4\ell,4n-4\ell)$ \\
                $(4n\ell-4\ell^2,2n^2-4n\ell+4\ell^2-n-1)$ \end{tabular} &
        {\tiny \begin{tabular}{l}
        $(4n\ell-4\ell^2+4\ell, 2n^2-4n\ell+4\ell^2+3n-4\ell-1)$ \\
        $(4n\ell-4\ell^2+4n-4\ell, 2n^2-4n\ell+4\ell^2+4\ell-n-1)$ \\
        $(2n^2-4n\ell+4\ell^2+3n-4\ell-1, 4n\ell-4\ell^2+4\ell)$ \\
        $(2n^2-4n\ell+4\ell^2+4\ell-n-1, 4n\ell-4\ell^2+4n-4\ell)$
        \end{tabular}} \\
\hline
\hline
\multicolumn{3}{|c|} {\footnotesize \bf (6) Real Form Family of $Spin(7)/G_2$;
	$\tau = \tau_{\xi_1,\R}$}\\
\hline
$Spin(7;\C)/G_{2,\C}$ & \begin{tabular}{l} (7,7) \end{tabular} & \begin{tabular}{l} (7,7) \end{tabular} \\
\hline
$Spin(7)/G_2$ & \begin{tabular}{l} (0,7) \end{tabular} & \begin{tabular}{l} (7,0) and (0,7) \end{tabular} \\
\hline
$Spin(3,4)/G_{2,A_1A_1}$ & \begin{tabular}{l} (4,3) \end{tabular} & \begin{tabular}{l} (4,3) and (3,4) \end{tabular}\\
\hline
\hline
\multicolumn{3}{|c|} {\footnotesize \bf (7) Real Form Family of $G_2/SU(3)$;
	$\tau = (\tau_{\xi_1} \oplus \tau_{\xi_2})_\R$}\\
\hline
$G_{2,\C}/SL(3;\C)$ & \begin{tabular}{l} (6,6)\end{tabular} & \begin{tabular}{l} (6,6) \end{tabular} \\
\hline
$G_2/SU(3)$ &\begin{tabular}{l} (0,6) \end{tabular} & \begin{tabular}{l} (6,0) and (0,6) \end{tabular} \\
\hline
$G_{2,A_1A_1}/SU(1,2)$ & \begin{tabular}{l} (4,2) \end{tabular} &\begin{tabular}{l}  (4,2) and (2,4) \end{tabular} \\
\hline
$G_{2,A_1A_1}/SL(3;\R)$ & \begin{tabular}{l}(3,3)\end{tabular} & \begin{tabular}{l} (3,3)\end{tabular}\\
\hline
\hline
\multicolumn{3}{|c|} {\footnotesize \bf (8) Real Form Family of $SO(10)/[Spin(7)\times SO(2)]$;
	$\tau = \tau_{\xi_1,\R}\oplus (\tau_{\xi_3,\R}
	\otimes (\zeta_1 \oplus \zeta_{-1}))_\R$}\\
\hline
{\footnotesize $\tfrac{SO(10;\C)}{Spin(7;\C)\times SO(2;\C)}$} & \begin{tabular}{l} (16,16), (7,7)\end{tabular} & \begin{tabular}{l} (23,23)\end{tabular}\\
\hline
{\footnotesize $\tfrac{SO(10)}{Spin(7)\times SO(2)}$} & \begin{tabular}{l} (0,16), (0,7)\end{tabular} & \begin{tabular}{l} (23,0), (16,7), (7,16), (0,23)\end{tabular}\\
\hline
{\footnotesize $\tfrac{SO(9,1)}{Spin(7,0)\times  SO(1,1)}$} & \begin{tabular}{l} (8,8), (0,7)\end{tabular} & \begin{tabular}{l} (15,8), (8,15)\end{tabular} \\
\hline
{\footnotesize $\tfrac{SO(8,2)}{Spin(7,0)\times SO(0,2)}$} & \begin{tabular}{l} (16,0), (0,7)\end{tabular} &
			 \begin{tabular}{l} (23,0), (16,7), (7,16), (0,23)\end{tabular} \\
\hline
{\footnotesize $\tfrac{SO(6,4)}{Spin(4,3)\times  SO(2,0)}$} & \begin{tabular}{l} (8,8), (4,3) \end{tabular}& \begin{tabular}{l} (12,11), (11,12)\end{tabular} \\
\hline
{\footnotesize $\tfrac{SO(5,5)}{Spin(3,4)\times  SO(1,1)}$} & \begin{tabular}{l} (8,8), (4,3)\end{tabular} & \begin{tabular}{l} (12,11), (11,12)\end{tabular} \\
\hline
\hline
\multicolumn{3}{|c|} {\footnotesize \bf (9) Real Form Family of $SO(9)/Spin(7)$;
	$\tau = \tau_{\xi_3,\R}\oplus \tau_{\xi_1,\R}$} \\
\hline
$SO(9;\C)/Spin(7;\C)$ & \begin{tabular}{l}(7,7) and (8,8) \end{tabular}&\begin{tabular}{l} (15,15)\end{tabular} \\
\hline
$SO(9)/Spin(7)$ & \begin{tabular}{l}(0,7) and (0,8)\end{tabular} &\begin{tabular}{l} (15,0), (8,7), (7,8), (0,15)\end{tabular}\\
\hline
$SO(8,1)/Spin(7)$ &\begin{tabular}{l} (0,7) and (8,0) \end{tabular}& \begin{tabular}{l}(15,0), (8,7), (7,8), (0,15)\end{tabular} \\
\hline
$SO(5,4)/Spin(3,4)$ & \begin{tabular}{l}(4,3) and (4,4)\end{tabular} & \begin{tabular}{l}(8,7), (7,8)\end{tabular}\\
\hline
\hline
\multicolumn{3}{|c|} {\footnotesize \bf (10) Real Form Family of $Spin(8)/G_2$;
	$\tau = \tau_{\xi_1,\R} \oplus \tau_{\xi_1,\R}$} \\
\hline
$Spin(8;\C)/G_{2,\C}$ &\begin{tabular}{l}   (7,7) and (7,7)\end{tabular} & \begin{tabular}{l}  (14,14) \end{tabular}\\
\hline
$Spin(8)/G_2$ & \begin{tabular}{l}  (0,7) and (0,7) \end{tabular}& \begin{tabular}{l} (14,0), (7,7), (0,14)\end{tabular} \\
\hline
$Spin(7,1)/G_2$ & \begin{tabular}{l}  (7,0) and (0,7)\end{tabular} & \begin{tabular}{l} (14,0), (7,7), (0,14)\end{tabular} \\
\hline
$Spin(4,4)/G_{2,A_1A_1}$ & \begin{tabular}{l} (4,3) and (4,3)\end{tabular} & \begin{tabular}{l} (8,6), (7,7), (6,8)\end{tabular} \\
\hline
$Spin(3,5)/G_{2,A_1A_1}$ &\begin{tabular}{l}  (4,3) and (3,4) \end{tabular}&\begin{tabular}{l}  (8,6), (7,7), (6,8) \end{tabular}\\
\hline
\hline
\multicolumn{3}{|c|} {\footnotesize \bf (11) Real Form Family of $SO(2n+1)/U(n)$;
	$\tau = (\tau_{\xi_2,\R} \otimes (\zeta \oplus \zeta_{-1}))_\R$} \\
\hline
 {\footnotesize $\tfrac{SO(2n+1;\C)}{GL(n;\C)}$} &
        \begin{tabular}{l} $(2n,2n)$, $\bigl ( n(n-1), n(n-1) \bigr )$
        \end{tabular} &
       \begin{tabular}{l} $\bigl ( n(n+1),n(n+1) \bigr )$ \end{tabular}\\
\hline
{\footnotesize $\tfrac{SO(2n+1-2k,2k)}{U(n-k,k)}$} & \begin{tabular}{l} $(2k,2n-2k)$ \\
				$(2kn-2k^2,n^2-2kn+2k^2-n)$ \end{tabular}
	& {\tiny \begin{tabular}{l} $(n^2-2nk+2k^2+n-2k, 2nk-2k^2+2k)$ \\
        $(n^2-2nk+2k^2-n+2k, 2nk-2k^2+2n-2k)$ \\
                $(2nk-2k^2+2k, n^2-2nk+2k^2+n-2k)$ \\
                $(2nk-2k^2+2n-2k, n^2-2nk+2k^2-n+2k)$
        \end{tabular}} \\
\hline
{\footnotesize $\tfrac{SO(2n-2k,2k+1)}{U(n-k,k)}$} & \begin{tabular}{l} $(2n-2k,2k)$ \\
				$(2kn-2k^2,n^2-2kn+2k^2-n)$ \end{tabular}
	& {\tiny \begin{tabular}{l}
        $(n^2 - 2nk + 2k^2 + n - 2k, 2nk - 2k^2 + 2k)$ \\
        $(n^2 - 2nk + 2k^2 - n + 2k, 2nk - 2k^2 + 2n - 2k)$ \\
        $(2nk - 2k^2 + 2k, n^2 - 2nk + 2k^2 + n - 2k)$ \\
        $(2nk - 2k^2 + 2n - 2k, n^2 - 2nk + 2k^2 - n + 2k)$
        \end{tabular}}  \\
\hline
$SO(n,n+1)/GL(n;\R)$ & \begin{tabular}{l} $(n,n)$ \\
                $\bigl ( \tfrac{1}{2}n(n-1), \tfrac{1}{2}n(n-1) \bigr )$
        \end{tabular} &
       \begin{tabular}{l} $\bigl ( \tfrac{1}{2}n(n+1),\tfrac{1}{2}n(n+1) \bigr )$ \end{tabular} \\
\hline
\hline
\multicolumn{3}{|c|} {\footnotesize \bf (12) Real Form Family of $Sp(n)/[Sp(n-1)\times U(1)]$;
	$\tau = (\tau_{\xi_1} \otimes (\zeta_{+1} \oplus \zeta_{-1})_\R)
	\oplus (\zeta_{+1} \oplus \zeta_{-1})_\R$} \\
\hline
{\footnotesize $\tfrac{Sp(n;\C)}{Sp(n-1;\C)\times GL(1;\C)}$} &
        \begin{tabular}{l} $(2,2)$ \\ $(4n-4,4n-4)$
        \end{tabular} &
       \begin{tabular}{l} $(4n-2,4n-2)$ \end{tabular} \\
\hline
{\footnotesize $\tfrac{Sp(n-k,k)}{Sp(n-1-k,k)\times U(1)}$} &
        \begin{tabular}{l} $(0,2)$ \\ $(4k,4n-4k-4)$ \end{tabular} &
 \begin{tabular}{l}  $(4n-4k-2, 4k)$ \\
        $(4n-4k-4, 4k+2)$ \\
        $(4k+2, 4n-4k-4)$ \\
        $(4k, 4n-4k-2)$
        \end{tabular} \\
\hline
{\footnotesize $\tfrac{Sp(n-k,k)}{Sp(n-k,k-1)\times U(1)}$} &
        \begin{tabular}{l} $(0,2)$ \\ $(4n-4k,4k-4)$ \end{tabular} &
 \begin{tabular}{l}  $(4n-4k, 4k-2)$ \\
        $(4n-4k+2, 4k-4)$ \\
        $(4k-4, 4n-4k+2)$ \\
        $(4k-2, 4n-4k)$
        \end{tabular} \\
\hline
{\footnotesize $\tfrac{Sp(n;\R)}{Sp(n-1;\R)\times GL(1;\R)}$} & \begin{tabular}{l} $(1,1)$, $(2n-2,2n-2)$ \end{tabular} &
 \begin{tabular}{l}
        $(2n-1,2n-1)$
        \end{tabular} \\
\hline
{\footnotesize $\tfrac{Sp(n;\R)}{Sp(n-1;\R)\times U(1)}$} &
        \begin{tabular}{l} $(2,0)$, $(2n-2,2n-2)$ \end{tabular} &
        \begin{tabular}{l}
        $(2n,2n-2), (2n-2,2n)$
        \end{tabular} \\
\hline
\end{longtable}
}

\section{Real Form Families for $G_u$ Semisimple but not Simple.}\label{sec4}
\setcounter{equation}{0}
Table (\ref{brion-compact}) just below is Yakimova's formulation
(\cite{Y2005}, \cite{Y2006}) of the principal case diagram of Mikityuk
\cite{M1986}, with some indices shifted to facilitate descriptions of the
real form families.
See \cite[Section 12.8]{W2007} for the details.
It gives the irreducible compact spherical
pairs and nonsymmetric compact weakly symmetric spaces.  There,
$\gs\gp(n)$ corresponds to the compact group $Sp(n)$.
In each of the nine entries of Table (\ref{brion-compact}), $\gg$ is the sum
of the algebras on the top row and $\gh$ is the sum of the algebras on the
bottom row.  We continue the numbering from (\ref{kraemer-classn}).

The spaces of (\ref{kraemer-classn}), and entries (1) through (8) in
(\ref{brion-compact}), all are principal.  Entry (9) of (\ref{brion-compact})
is a little more complicated; see \cite[Section 12.8]{W2007}.  There the
$\gg_i$ are semisimple but not necessarily simple.

{\footnotesize
\begin{equation} \label{brion-compact}
\begin{tabular}{|l|l|l|}\hline
\multicolumn{3}{| c |}{Compact Irred Nonsymmetric Weakly
Symmetric $(\gg,\gh)$,\
$\gg$ is Semisimple but not Simple}\\
\hline \hline
\setlength{\unitlength}{.45 mm}
\begin{picture}(65,28)
\put(0,20){(13)}
\put(13,20){$\gs\gu(n)$}
\put(38,20){$\gs\gu(n+1)$}
\put(13,5){$\gs\gu(n)$}
\put(45,5){$\gu(1)$}
\put(21,10){\line(0,1){8}}
\put(50,10){\line(0,1){8}}
\put(26,10){\line(2,1){15}}
\end{picture}
&
\setlength{\unitlength}{.45 mm}
\begin{picture}(85,28)
\put(0,20){(16)}
\put(13,20){$\gs\gu(n+2)$}
\put(58,20){$\gs\gp(m+1)$}
\put(11,5){$\gu(n)$}
\put(25,0){$\gs\gu(2) = \gs\gp(1)$}
\put(63,5){$\gs\gp(m)$}
\put(21,10){\line(0,1){8}}
\put(73,10){\line(0,1){8}}
\put(50,5){\line(1,1){13}}
\put(38,5){\line(-1,1){13}}
\end{picture}
&
\setlength{\unitlength}{.45 mm}
\begin{picture}(105,28)
\put(0,20){(19)}
\put(13,20){$\gs\gp(n+1)$}
\put(45,20){$\gs\gp(\ell+1)$}
\put(75,20){$\gs\gp(m+1)$}
\put(13,5){$\gs\gp(n)$}
\put(35,5){$\gs\gp(1)$}
\put(53,5){$\gs\gp(\ell)$}
\put(75,5){$\gs\gp(m)$}
\put(20,10){\line(0,1){8}}
\put(35,10){\line(-1,1){8}}
\put(41,10){\line(1,1){8}}
\put(45,10){\line(3,1){27}}
\put(60,10){\line(-1,1){8}}
\put(83,10){\line(0,1){8}}
\end{picture} \\
\hline
\setlength{\unitlength}{.45 mm}
\begin{picture}(65,28)
\put(0,20){(14)}
\put(13,20){$\gs\gp(n+2)$}
\put(48,20){$\gs\gp(2)$}
\put(13,5){$\gs\gp(n)$}
\put(48,5){$\gs\gp(2)$}
\put(21,10){\line(0,1){8}}
\put(53,10){\line(0,1){8}}
\put(51,10){\line(-2,1){15}}
\end{picture}
&
\setlength{\unitlength}{.45 mm}
\begin{picture}(85,28)
\put(0,20){(17)}
\put(13,20){$\gs\gu(n+2)$}
\put(58,20){$\gs\gp(m+1)$}
\put(11,5){$\gs\gu(n)$}
\put(25,0){$\gs\gu(2) = \gs\gp(1)$}
\put(63,5){$\gs\gp(m)$}
\put(21,10){\line(0,1){8}}
\put(70,10){\line(0,1){8}}
\put(47,5){\line(1,1){13}}
\put(35,5){\line(-1,1){13}}
\end{picture}
&
\setlength{\unitlength}{.45 mm}
\begin{picture}(105,28)
\put(0,20){(20)}
\put(13,20){$\gs\gp(n+1)$}
\put(45,20){$\gs\gp(2)$}
\put(70,20){$\gs\gp(m+1)$}
\put(13,5){$\gs\gp(n)$}
\put(35,5){$\gs\gp(1)$}
\put(57,5){$\gs\gp(1)$}
\put(75,5){$\gs\gp(m)$}
\put(17,10){\line(0,1){8}}
\put(35,10){\line(-1,1){8}}
\put(41,10){\line(1,1){8}}
\put(60,10){\line(-1,1){8}}
\put(65,10){\line(1,1){8}}
\put(83,10){\line(0,1){8}}
\end{picture} \\
\hline
\setlength{\unitlength}{.45 mm}
\begin{picture}(65,28)
\put(0,20){(15)}
\put(13,20){$\gs\go(n)$}
\put(38,20){$\gs\go(n+1)$}
\put(28,5){$\gs\go(n)$}
\put(33,10){\line(-1,1){8}}
\put(35,10){\line(1,1){8}}
\end{picture}
&
\setlength{\unitlength}{.45 mm}
\begin{picture}(85,28)
\put(0,20){(18)}
\put(13,20){$\gs\gp(n+1)$}
\put(58,20){$\gs\gp(m+1)$}
\put(11,5){$\gs\gp(n)$}
\put(38,5){$\gs\gp(1)$}
\put(63,5){$\gs\gp(m)$}
\put(21,10){\line(0,1){8}}
\put(73,10){\line(0,1){8}}
\put(50,10){\line(2,1){15}}
\put(38,10){\line(-2,1){15}}
\end{picture}
&
\setlength{\unitlength}{.45 mm}
\begin{picture}(105,28)
\put(0,20){(21)}
\put(38,20){$\gg_1 \ \ \ \ \dots \ \ \ \ \gg_n$}
\put(38,5){$\gh'_1 \ \ \ \ \dots \ \ \ \ \gh'_n$}
\put(40,10){\line(0,1){8}}
\put(52,12){$\dots$}
\put(70,10){\line(0,1){8}}
\put(20,5){$\gz_\gh$}
\put(25,10){\line(3,2){10}}
\put(26,6){\line(3,1){40}}
\end{picture} \\
\hline
\end{tabular}
\end{equation}
}

\begin{definition}\label{br-symm}
Let $M_u = G_u/H_u$ be one of the entries in {\rm (\ref{brion-compact})}
excluding entry $(21)$.
For entries $(13)$, $(14)$, $(15)$, $(16)$, $(17)$ and $(18)$ express
$\gg_u = \gg_{1,u} \oplus \gg_{2,u}$ with $\gg_{u,i}$
nonzero and simple.  For entries $(19)$ and $(20)$ express
$\gg_u = \gg_{1,u} \oplus \gg_{2,u} \oplus \gg_{3,u}$ with $\gg_{u,i}$
nonzero and simple.  Let $\gh_{u,i}$ denote the image of $\gh_u$ under
the projection $\gg_u \to \gg_{u,i}$\,.  Then each $(\gg_{u,i},\gh_{u,i})$
corresponds to a compact simply connected irreducible riemannian symmetric
{\rm (or at least weakly symmetric)} space $M_{u,i} = G_{u,i}/H_{u,i}$\,,
and $\widetilde{M_u} = \prod M_{u,i}$ is the {\bf riemannian unfolding} of
$M_u$\,.
\end{definition}

Now we run through the cases of (\ref{brion-compact}).  The results will
be summarized in Table 4.12 below.
\medskip

\centerline{
{\footnotesize
\fbox{\setlength{\unitlength}{.40 mm}
\begin{picture}(220,28)
\put(0,14){(13)}
\put(23,20){$\gs\gu(n)$}
\put(48,20){$\gs\gu(n+1)$}
\put(23,5){$\gs\gu(n)$}
\put(55,5){$\gu(1)$}
\put(31,10){\line(0,1){8}}
\put(60,10){\line(0,1){8}}
\put(36,10){\line(2,1){15}}
\put(90,12){$[SU(n)SU(n+1)]/S[U(n)\times U(1)]$}
\end{picture}
}
}
}

The real form family of $SU(n)/SU(n)$ consists of the $G_1/H_1$ given by
$$
\begin{aligned}
&SL(n;\C)/SL(n;\C),\, SL(n;\R)/SL(n;\R),\, SL(n';\H)/SL(n';\H)
	\text{ if } n = 2n',\\
& \text{and one of the } SU(k,\ell)/SU(k,\ell) \text{ where } k+\ell = n.
\end{aligned}
$$
The real form family of $SU(n+1)/S[U(n)\times U(1)]$ consists of the
$\widetilde{M}_2 = G_2/H_2$ in (\ref{b1}) with $(m,n)$ replaced by $(n,1)$.
That family is
$$
\begin{aligned}
&SL(n+1;\C)/S[GL(n;\C)\times GL(1;\C)],\,
 SL(n+1;\R)/S[GL(n;\R)\times GL(1;\R)],	\\
&SU(n-a+1,a)/S[U(n-a,a)\times U(1,0)],\, SU(n-a,1+a)/S[U(n-a,a)\times U(0,1)].
\end{aligned}
$$
We can fold these together exactly in the cases where $H_1$ is the
semisimple part of $H_2$\,, so the possibilities for $G/H$ are
\begin{equation}\label{C1}
\begin{aligned}
{\rm (i)\,\,}& [SL(n;\C) \times SL(n+1;\C)]/[SL(n;\C) \times GL(1;\C)] \\
{\rm (ii)\,\,}&[SL(n;\R) \times SL(n+1;\R)]/[SL(n;\R) \times GL(1;\R)] \\
{\rm (iii)\,\,}&[SU(k,\ell) \times SU(k+1,\ell)]/[SU(k,\ell) \times U(1)]\\
{\rm (iv)\,\,}&[SU(k,\ell) \times SU(k,\ell+1)]/[SU(k,\ell) \times U(1)]
\end{aligned}
\end{equation}
In case (i), the metric irreducible summands of the tangent space have
signatures $(n^{2}-1,n^{2}-1)$ and $(2n,2n)$. In case (ii), those signatures are $\big(\frac{n(n+1)}{2}-1,\frac{n(n-1)}{2}\big)$ and $(n,n)$. In case (iii),
those signatures are $(2k\ell,k^2+\ell^2-1)$ and $(2\ell,2k)$. In case (iv), those signatures are $(2k\ell,k^2+\ell^2-1)$ and $(2k,2\ell)$.
\medskip

\centerline{
{\footnotesize
\fbox{\setlength{\unitlength}{.40 mm}
\begin{picture}(215,28)
\put(0,14){(14)}
\put(23,20){$\gs\gp(n+2)$}
\put(58,20){$\gs\gp(2)$}
\put(23,5){$\gs\gp(n)$}
\put(58,5){$\gs\gp(2)$}
\put(31,10){\line(0,1){8}}
\put(63,10){\line(0,1){8}}
\put(61,10){\line(-2,1){15}}
\put(90,12){$[Sp(n+2)\times Sp(2)]/[Sp(n)\times Sp(2)]$}
\end{picture}
}
}
}

For $Sp(n+2)/[Sp(n)\times Sp(2)$, we have the following possibilities:
$$
\begin{aligned}
&  Sp(n+2;\C)/[Sp(n;\C)\times Sp(2;\C)], \quad Sp(n+2;\R)/[Sp(n;\C)\times Sp(2;\R)] \\
& Sp(n-a+b,2-b+a)/[Sp(n-a,a)\times Sp(b,2-b)] \text{ for } 0\leqq a\leqq n \text{ and } 0\leqq b\leqq 2 \\
\end{aligned}
$$
Thus we have the following possibilities for this case:
\begin{equation}\label{C2}
\begin{aligned}
{\rm (i)\,\,} & [Sp(n+2;\C)\times Sp(2;\C)]/[Sp(n;\C)\times Sp(2;\C)] \\
{\rm (ii)\,\,} &[Sp(n+2;\R)\times Sp(2;\R)]/[Sp(n;\R)\times Sp(2;\R)] \\
{\rm (iii)\,\,} & [Sp(n-a+b,2-b+a)\times Sp(b,2-b)]/[Sp(n-a,a)\times Sp(b,2-b)] \\
&\phantom{XXXXXXXXXXXXXXXXXXXX} \text{ for } 0\leqq a\leqq n \text{ and } b\in \{0,1,2\} \\
\end{aligned}
\end{equation}
In case (i), the metric irreducible summands of the tangent space have
signatures $(10,10)$ and $(8n,8n)$.
In case (ii), those signatures are $(6,4)$ and $(4n,4n)$.
In case (iii), those signatures are $(8b-4b^2,4b^2-8b+10)$
and $(8n+4(a-n)b+4a(b-2),4(n-a)b+4a(2-b))$.
\medskip

\centerline{
{\footnotesize
\fbox{\setlength{\unitlength}{.45 mm}
\begin{picture}(190,28)
\put(0,15){(15)}
\put(23,20){$\gs\go(n)$}
\put(55,20){$\gs\go(n+1)$}
\put(38,5){$\gs\go(n)$}
\put(43,10){\line(-1,1){8}}
\put(45,10){\line(1,1){8}}
\put(100,12){$[SO(n)\times SO(n+1)]/SO(n)$}
\end{picture}
}
}
}

For $SO(n+1)/SO(n)$, the possibilities are
$$
SO(n+1;\C)/SO(n,\C), \text{ and } SO(n-a+b,1-b+a)/SO(n-a,a) \text{ for }
0\leqq a\leqq n \text{ and } 0\leqq b\leqq 1.
$$
Thus the possibilities for $M$ in this case are:
\begin{equation}\label{C3}
\begin{aligned}
{\rm (i)\,\,} &[SO(n,\C)\times SO(n+1;\C)]/SO(n,\C)  \\
{\rm (ii)\,\,} & [SO(n-a,a)\times SO(n-a,a+1)]/SO(n-a,a) \text{ for }
	0\leqq a\leqq n \\
{\rm (iii)\,\,} & [SO(n-a,a)\times SO(n-a+1,a)]/SO(n-a,a) \text{ for }
	0\leqq a\leqq n
\end{aligned}
\end{equation}
In case (i), the metric irreducible subspaces of the real tangent space
have signatures $\bigl(\frac{n(n-1)}{2},\frac{n(n-1)}{2}\bigr)$ and $(n,n)$.
In case (ii), those signatures are $\bigl((n-a)a,\frac{n(n-1)}{2}-(n-a)a\bigr)$ and $(n-a,a)$. In case (iii), those signatures are $\bigl((n-a)a,\frac{n(n-1)}{2}-(n-a)a\bigr)$ and $(a,n-a)$.
\medskip

\centerline{
{\footnotesize
\fbox{\setlength{\unitlength}{.45 mm}
\begin{picture}(260,28)
\put(0,12){(16)}
\put(23,20){$\gs\gu(n+2)$}
\put(68,20){$\gs\gp(m+1)$}
\put(21,5){$\gu(n)$}
\put(35,0){$\gs\gu(2) = \gs\gp(1)$}
\put(73,5){$\gs\gp(m)$}
\put(31,10){\line(0,1){8}}
\put(83,10){\line(0,1){8}}
\put(60,5){\line(1,1){13}}
\put(48,5){\line(-1,1){13}}
\put(110,12){$[SU(n+2)\times Sp(m+1)]/[U(n)\times SU(2)\times Sp( m)]$}
\end{picture}
}
}
}

Let $G_u/H_u = [SU(n+2)\times Sp(m+1)]/[U(n)\times SU(2)\times Sp( m)]$ as in
entry (14) on Table \ref{brion-compact}.  The real form family of
$M_1 = SU(n+2)/S[U(n)\times U(2)]$ consists of the
$G_1/H_1$ in (\ref{b1}) with $(m,n)$ replaced by $(n,2)$.
That family is
$$
\begin{aligned}
& SL(n+2;\C)/S[GL(n;\C)\times GL(2;\C)], \quad SL(n+2;\R)/S[GL(n;\R)\times GL(2;\R)] \\
&  SL(n'+1;\H)/S[GL(n';\H)\times GL(1;\H)]
        \text{ where } n = 2n' \\
& SU(n-a+b,2-b+a)/S[U(n-a,a)\times U(b,2-b)] \text{ for } a\leqq n \text{ and } b\leqq 2. \\
& SL(4;\R)/GL'(2;\C), \quad SU^*(4)/[GL'(2;\C)], 
	\quad SU(2,2)/[SL(2;\C)\times \R]. \\
\end{aligned}
$$
where $GL'(m;\C) := \{g \in GL(m;\C) \mid |\det(g)| = 1\}$ and
$GL(k;\H) := SL(k;\H)\times \R^+$. The real form family of
$M_2 = Sp(m+1)/[Sp(m)\times Sp(1)]$ consists of the
$$
\begin{aligned}
&  Sp(m+1;\C)/[Sp(m;\C)\times Sp(1;\C)], \quad
Sp(m+1;\R)/[Sp(m;\R)\times Sp(1;\R)],\text{ and }\\
& Sp(m-a+b,1-b+a)/[Sp(m-a,a)\times Sp(b,1-b)]
	\text{ for } 0\leqq a\leqq m \text{ and } 0\leqq b\leqq 1. \\
\end{aligned}
$$
Fitting these together, the real form family of $M_u =
[SU(n+2)\times Sp(m+1)]/[U(n)\times SU(2)\times Sp( m)]$ consists of the
$$
\begin{aligned}
& [SL(n+2;\C)\times Sp(m+1;\C)]/[GL(n;\C)\times SL(2;\C)\times Sp(m;\C)] \\
&[SL(n+2;\R)\times Sp(m+1;\R)]/[GL(n;\R)\times SL(2;\R)\times Sp(m;\R)] \\
& [SU(n-a_1+b_1,2-b_1+a_1)\times Sp(m-a_2+b_2,1-b_2+a_2)] \\
&\phantom{XXXX}
 /[U(n-a_1,a_1)\times SU(2)\times Sp(m-a_2,a_2)]
        \text{ where } 0\leqq a_1\leqq n, 0\leqq a_2\leqq m, b_1=0,2, b_2=0,1 \\
& [SU(n+1-a,a+1)\times Sp(m+1;\R)]/[U(n-1,a)\times SU(1,1)\times Sp(m;\R)] \text{ for } 0\leqq a\leqq n\\
&  [SL(4;\R)\times Sp(m+1;\C)]/[GL'(2;\C)\times Sp(m;\C)], \quad [SU^*(4)\times Sp(m+1;\C)]/[SL(2;\C)\times Sp(m;\C)\times T] \\
&  [SU(2,2)\times Sp(m+1;\C)]/[SL(2;\C)\times Sp(m;\C) \times \R]
\end{aligned}
$$
This list is not convenient for analysis of the metric irreducible
subspaces of the tangent space, so we refine it as follows.

\begin{equation}\label{C4}
\begin{aligned}
{\rm (i)\,\,} & [SL(n+2;\C)\times Sp(m+1;\C)]/[GL(n;\C)\times SL(2;\C)\times Sp(m;\C)] \\
{\rm (ii)\,\,} &[SL(n+2;\R)\times Sp(m+1;\R)]/[GL(n;\R)\times SL(2;\R)\times Sp(m;\R)] \\
{\rm (iii)\,\,} &[SL(n'+1;\H)\times Sp(m-a,1+a)]/[GL(n';\H)\times SU(2)\times Sp(m-a,a)] \text{ for } 0\leqq a\leqq n \\
{\rm (iv)\,\,} &[SL(n'+1;\H)\times Sp(m-a+1,a)]/[GL(n';\H)\times SU(2)\times Sp(m-a,a)] \text{ for } 0\leqq a\leqq n \\
{\rm (v)\,\,} & [SU(n-a_1+b_1,2-b_1+a_1)\times Sp(m-a_2,1+a_2)] \\
&  /[U(n-a_1,a_1)\times SU(2)\times Sp(m-a_2,a_2)], \quad
        \text{ where } 0\leqq a_1\leqq n, 0\leqq a_2\leqq m, b_1\in \{0,2\} \\
 {\rm (vi)\,\,} & [SU(n-a_1+b_1,2-b_1+a_1)\times Sp(m-a_2+1,a_2)] \\
&
 /[U(n-a_1,a_1)\times SU(2)\times Sp(m-a_2,a_2)], \quad
        \text{ where } 0\leqq a_1\leqq n, 0\leqq a_2\leqq m, b_1\in \{0,2\} \\
{\rm (vii)\,\,} & [SU(n+1-a,a+1)\times Sp(m+1;\R)]/[U(n-a,a)\times SU(1,1)\times Sp(m;\R)] \text{ for } 0\leqq a\leqq n\\
{\rm (viii)\,\,} &  [SL(4;\R)\times Sp(m+1;\C)]/[SL(2;\C)\times Sp(m;\C) \times T] \\
{\rm (ix)\,\,} &  [SU^*(4)\times Sp(m+1;\C)]/[SL(2;\C)\times Sp(m;\C)\times T] \\
{\rm (x)\,\,} &  [SU(2,2)\times Sp(m+1;\C)]/[SL(2;\C)\times Sp(m;\C) \times \R]
\end{aligned}
\end{equation}

In case (i), the metric irreducible subspaces of the real tangent space
have signatures $(3,3)$, $(4n,4n)$ and $(4m,4m)$. In case (ii), those signatures are $(2,1)$, $(2n,2n)$ and $(2m,2m)$. In case (iii), those signatures are $(0,3)$, $(4n,4n)$ and $(4m-4a,4a)$. In case (iv), those signatures are $(0,3)$, $(4n,4n)$ and $(4a,4m-4a)$. In case (v), those signatures are $(0,3)$, $(4(n-a_1)-2b_1(n-2a_1),2b_1(n-2a_1)+4a_1)$ and $(4m-4a_2,4a_2)$.
In case (vi), those signatures are $(0,3)$, $(4(n-a_1)-2b_1(n-2a_1),2b_1(n-2a_1)+4a_1)$ and $(4a_2,4m-4a_2)$.
In case (vii), those signatures are $(2,1)$, $(2n,2n)$ and $(2m,2m)$.
In case (viii), those signatures are $(3,3)$, $(6,2)$ and $(4m,4m)$.
In case (ix), those signatures are $(3,3)$, $(2,6)$ and $(4m,4m)$.
In case (x), those signatures are $(3,3)$, $(4,4)$ and $(4m,4m)$.
\medskip

\centerline{
{\footnotesize
\fbox{\setlength{\unitlength}{.45 mm}
\begin{picture}(250,28)
\put(0,12){(17)}
\put(23,20){$\gs\gu(n+2)$}
\put(68,20){$\gs\gp(m+1)$}
\put(21,5){$\gs\gu(n)$}
\put(35,0){$\gs\gu(2) = \gs\gp(1)$}
\put(73,5){$\gs\gp(m)$}
\put(31,10){\line(0,1){8}}
\put(80,10){\line(0,1){8}}
\put(57,5){\line(1,1){13}}
\put(45,5){\line(-1,1){13}}
\put(100,12){$[SU(n+2)\times Sp(m+1)]/[SU(n)\times SU(2)\times Sp(m)]$}
\end{picture}
}
}
}

Using the calculations in \S \ref{sec2}, the real form family for
$SU(n+2)/[SU(n)\times SU(2)]$ is
$$
\begin{aligned}
& SL(n+2;\C)/[SL(n;\C)\times SL(2;\C)], \quad
  SL(n+2;\R)/[SL(n;\R)\times SL(2;\R)], \quad SL(4;\R)/SL(2;\C) \\
&  SU^*(4)/SL(2;\C), \quad SU(2,2)/SL(2;\C), \quad SL(n'+1;\H)/[SL(n';\H)\times SL(1;\H)]
	\text{ where } n = 2n', \text{ and }\\
& SU(n-a+b,2-b+a)/[SU(n-a,a)\times SU(b,2-b)] \text{ for } a\leqq n \text{ and } b\leqq 2.
\end{aligned}
$$
Combining that with the possibilities for $Sp(n+1)/[Sp(n)\times Sp(1)]$, and
refining the result as appropriate for computation of the metric
irreducible subspaces of the real tangent space, this case gives us

\begin{equation}\label{C5}
\begin{aligned}
{\rm (i)\,\,} & [SL(n+2;\C)\times Sp(m+1;\C)]/[SL(n;\C)\times SL(2;\C)\times Sp(m;\C)] \\
{\rm (ii)\,\,} &[SL(n+2;\R)\times Sp(m+1;\R)]/[SL(n;\R)\times SL(2;\R)\times Sp(m;\R)] \\
{\rm (iii)\,\,} &[SL(n'+1;\H)\times Sp(m-a,1+a)]/[SL(n';\H)\times SU(2)\times Sp(m-a,a)] \text{ for } 0\leqq a\leqq n \\
{\rm (iv)\,\,} &[SL(n'+1;\H)\times Sp(m-a+1,a)]/[GL(n';\H)\times SU(2)\times Sp(m-a,a)] \text{ for } 0\leqq a\leqq n \\
{\rm (v)\,\,} & [SU(n-a_1+b_1,2-b_1+a_1)\times Sp(m-a_2,1+a_2)] \\
&  /[SU(n-a_1,a_1)\times SU(2)\times Sp(m-a_2,a_2)], \quad
        \text{ where } 0\leqq a_1\leqq n, 0\leqq a_2\leqq m, b_1\in \{0,2\} \\
 {\rm (vi)\,\,} & [SU(n-a_1+b_1,2-b_1+a_1)\times Sp(m-a_2+1,a_2)] \\
&
 /[SU(n-a_1,a_1)\times SU(2)\times Sp(m-a_2,a_2)], \quad
        \text{ where } 0\leqq a_1\leqq n, 0\leqq a_2\leqq m, b_1\in \{0,2\} \\
{\rm (vii)\,\,} & [SU(n+1-a,a+1)\times Sp(m+1;\R)]/[SU(n-a,a)\times SU(1,1)\times Sp(m;\R)] \text{ for } 0\leqq a\leqq n\\
{\rm (viii)\,\,} &  [SL(4;\R)\times Sp(m+1;\C)]/[SL(2;\C)\times Sp(m;\C)] \\
{\rm (ix)\,\,} &  [SU^*(4)\times Sp(m+1;\C)]/[SL(2;\C)\times Sp(m;\C)] \\
{\rm (x)\,\,} &  [SU(2,2)\times Sp(m+1;\C)]/[SL(2;\C)\times Sp(m;\C)]
\end{aligned}
\end{equation}

In case (i), the metric irreducible subspaces of the real tangent space
have signatures $(3,3)$, $(1,0)$, $(0,1)$, $(4n,4n)$ and $(4m,4m)$.
In case (ii), those signatures are $(2,1)$, $(1,0)$, $(2n,2n)$ and $(2m,2m)$.
In case (iii), those signatures are $(0,3)$, $(1,0)$, $(4n,4n)$ and $(4m-4a,4a)$.
In case (iv), those signatures are $(0,3)$, $(1,0)$, $(4n,4n)$ and $(4a,4m-4a)$.
In case (v), those signatures are $(0,3)$, $(0,1)$, $(4(n-a_1)-2b_1(n-2a_1),2b_1(n-2a_1)+4a_1)$ and $(4m-4a_2,4a_2)$.
In case (vi), those signatures are $(0,3)$, $(0,1)$, $(4(n-a_1)-2b_1(n-2a_1),2b_1(n-2a_1)+4a_1)$ and $(4a_2,4m-4a_2)$.
In case (vii), those signatures are $(2,1)$, $(0,1)$, $(2n,2n)$ and $(2m,2m)$.
In case (viii), those signatures are $(3,3)$, $(0,1)$, $(6,2)$ and $(4m,4m)$.
In case (ix), those signatures are $(3,3)$, $(0,1)$, $(2,6)$ and $(4m,4m)$.
In case (x), those signatures are $(3,3)$, $(1,0)$, $(4,4)$ and $(4m,4m)$.
\medskip

\centerline{
{\footnotesize
\fbox{\setlength{\unitlength}{.45 mm}
\begin{picture}(200,28)
\put(0,12){(18)}
\put(23,20){$\gs\gp(n+1)$}
\put(68,20){$\gs\gp(m+1)$}
\put(21,5){$\gs\gp(n)$}
\put(48,5){$\gs\gp(1)$}
\put(73,5){$\gs\gp(m)$}
\put(31,10){\line(0,1){8}}
\put(83,10){\line(0,1){8}}
\put(60,10){\line(2,1){15}}
\put(48,10){\line(-2,1){15}}
\put(125,12){$\frac{Sp(n+1)\times Sp(m+1)}{Sp(n)\times Sp(m)\times Sp(1)}$}
\end{picture}
}
}
}

Here are the possibilities for real forms:
\begin{equation}\label{C6}
\begin{aligned}
{\rm (i)\,\,} & [Sp(n+1;\C)\times Sp(m+1;\C)]
	/[Sp(n;\C)\times Sp(1;\C) \times Sp(m;\C)] \\
{\rm (ii)\,\,} &[Sp(n+1;\R)\times Sp(m+1;\R)]
	/[Sp(n;\R)\times Sp(1;\R)\times Sp(m;\R)] \\
{\rm (iii)\,\,} & [Sp(n-a_1+b_1,1-b_1+a_1)\times Sp(m-a_2+b_2,1-b_2+a_2)] \\
	&\phantom{X}
 		/[Sp(n-a_1,a_1) \times Sp(1) \times Sp(m-a_2,a_2)]
		\text{ where } 0\leqq a_1\leqq n,
		0\leqq a_2\leqq  m, b_1, b_2\in \{0,1\} \\
{\rm (iv) \,\,} & Sp(n+1;\C)/[Sp(n;\C)\times Sp(1)] \text{ where } m=n \\
{\rm (v) \,\,} & Sp(n+1;\C)/[Sp(n;\C)\times Sp(1;\R)] \text{ where } m=n
\end{aligned}
\end{equation}
The first three correspond to inner automorphisms of $G_u$, preserving
each simple factor, and the last two to an
involutive automorphism $\alpha$ that interchanges the two simple
factors.  Then $\alpha$ is given by the interchange
$(x_1,x_2) \mapsto (x_2,x_1)$ and is the identity on the common $Sp(1)$
factor of $H_u$\,, so the corresponding real form $G/H$ is given by
$G = Sp(n+1;\C)$ of $G_\C$ and $H = Sp(n;\C)\times Sp(1)$ of $H_\C$\,.
In detail we are using
\begin{lemma}\label{switch}
Let $\gm_1$ and $\gm_2$ be Lie algebras, $\gm = \gm_1 \oplus \gm_2$\,, and
$\alpha$ an involutive automorphism of $\gm$ that exchanges the $\gm_i$
(so $\gm_1 \cong \gm_2$).
Write $\gm = \gm_+ + \gm_-$\,, sum the $(\pm 1)$--eigenspaces
of $\alpha$.  Then the corresponding Lie algebra
$\gm_+ + \sqrt{-1}\,\gm_- \cong (\gm_1)_\C$ as a real Lie algebra.
\end{lemma}
\begin{proof}
Identify the $\gm_i$ by means of $\alpha$, so $\gm = \gm_1 \oplus \gm_1$
with $\alpha$ given by $\alpha(x,y) = (y,x)$.  Then
$\gm_+ = \{(x,x)\mid x \in \gm_1\}$ and $\gm_- = \{(y,-y)\mid y \in \gm_1\}$,
and $\gm_+ + \sqrt{-1}\,\gm_- = \{(x,x) + \sqrt{-1}\,(y,-y)\}
= \{(z,\overline{z})\mid z \in (\gm_1)_\C\}$, which is isomorphic to
$(\gm_1)_\C$ as a real Lie algebra.
\end{proof}

In case (i), the metric irreducible subspaces of the real tangent space are
of signatures $(3,3)$, $(4n,4n)$ and $(4m,4m)$.
In case (ii), those signatures are $(2,1)$, $(2n,2n)$ and $(2m,2m)$.
In case (iii), those signatures are
$(0,3)$, $(4(n-a_1)-4b_1(n-2a_1),4a_1+4b_1(n-2a_1))$
and $(4(m-a_2)-4b_2(m-2a_2),4a_2+4b_2(m-2a_2))$.
In case (iv), those signatures are $(3,0)$ and $(4n,4n)$.
In case (v), those signatures are $(1,2)$ and $(4n,4n)$.
\medskip

\centerline{
{\footnotesize
\fbox{\setlength{\unitlength}{.50 mm}
\begin{picture}(200,28)
\put(0,12){(19)}
\put(23,20){$\gs\gp(n+1)$}
\put(55,20){$\gs\gp(\ell+1)$}
\put(85,20){$\gs\gp(m+1)$}
\put(23,5){$\gs\gp(n)$}
\put(45,5){$\gs\gp(1)$}
\put(63,5){$\gs\gp(\ell)$}
\put(85,5){$\gs\gp(m)$}
\put(30,10){\line(0,1){8}}
\put(45,10){\line(-1,1){8}}
\put(51,10){\line(1,1){8}}
\put(55,10){\line(3,1){27}}
\put(70,10){\line(-1,1){8}}
\put(93,10){\line(0,1){8}}
\put(125,12){$\frac{Sp(n+1)\times Sp(\ell+1)\times Sp(m+1)}{Sp(n)\times Sp(\ell)\times Sp(m)\times Sp(1)}$}
\end{picture}
}
}
}

Let $M_u = G_u/H_u = [Sp(n+1)\times Sp(\ell +1)\times Sp(m+1)]
/[Sp(n)\times Sp(\ell)\times Sp(m)\times Sp(1)]$\,, $n \leqq \ell \leqq m$.
Then
$M_{1,u} = Sp(n+1)/[Sp(n)\times Sp(1)]$,
$M_{2,u} = Sp(\ell+1)/[Sp(\ell)\times Sp(1)]$ and
$M_{3,u} = Sp(m+1)/[Sp(n)\times Sp(1)]$.
Let $\alpha$ be an involutive automorphism of $G_u$\,.  It induces a
permutation $\overline{\alpha}$ of $\{M_{1,u}\,, M_{2,u}\,, M_{3,u}\}$.
Up to conjugacy, and using $\alpha^2 = 1$, the possibilities
are (a) $\alpha$ is inner and $\overline{\alpha} = 1$,
and (b) $\alpha$ is outer, $n = \ell$, $\overline{\alpha}$ exchanges
$M_{1,u}$ and $M_{2,u}$\,, and $\overline{\alpha}(M_{3,u}) = M_{3,u}$\,.
In case (b) we argue as in (\ref{C6}).  Now the possibilities for
$M = G/H$ are
\begin{equation}\label{C7}
\begin{aligned}
{\rm (i)\,\,} & [Sp(n+1;\C)\times Sp(\ell+1;\C)\times Sp(m+1;\C)]
	/[Sp(n;\C)\times Sp(\ell;\C)\times Sp(m;\C) \times Sp(1;\C)] \\
{\rm (ii)\,\,} &[Sp(n+1;\R)\times Sp(\ell+1;\R)\times Sp(m+1;\R)]
	/[Sp(n;\R)\times Sp(\ell;\R)\times Sp(m;\R) \times Sp(1;\R)] \\
{\rm (iii)\,\,} & [Sp(n-a_1+b_1,1-b_1+a_1)\times Sp(\ell-a_2+b_2,1-b_2+a_2)
		\times Sp(m-a_3+b_3,1-b_3+a_3)] \\
	&\phantom{XXXXXXX} /[Sp(n-a_1,a_1)\times Sp(\ell-a_2,a_2)
		\times Sp(m-a_3,a_3)\times Sp(1)] \\
	&\phantom{XXXXXXXXX} \text{ where } 0\leqq a_1\leqq n,
	  0\leqq a_2\leqq \ell, 0\leqq a_3\leqq m, b_1, b_2, b_3\in \{0,1\} \\
{\rm (iv)\,\,} & [Sp(n+1;\C) \times Sp(m+1;\R)]
	/[Sp(n;\C)\times Sp(1;\R)\times Sp(m;\R)] \text{ if } n = \ell\\
{\rm (v)\,\,} & [Sp(n+1;\C) \times Sp(m+1-a,a)]
	/[Sp(n;\C)\times Sp(1)\times Sp(m-a,a)],\,\, 0 \leqq a \leqq m,\,\,
	\text{ if } n = \ell \\
{\rm (vi)\,\,} & [Sp(n+1;\C) \times Sp(m-a,a+1)]
	/[Sp(n;\C)\times Sp(1)\times Sp(m-a,a)],\,\, 0 \leqq a \leqq m,\,\,
	\text{ if } n = \ell \\
\end{aligned}
\end{equation}
Here the first three cases correspond to inner automorphisms, case (a),
and the remaining three correspond to outer automorphisms $\alpha$, case (b).
There we apply Lemma \ref{switch} to the interchange $G_{1,u} \leftrightarrow
G_{2,u}$ defined by $\alpha$, $\alpha|_{G_{3,u}}$ is any involutive
automorphism.

In case (i), the signatures of the metric irreducible subspaces of the
real tangent space of $M = G/H$ are $(3,3)$, $(3,3)$, $(4n,4n)$, $(4l,4l)$ and $(4m,4m)$.
In case (ii) those signatures are $(2,1)$, $(2,1)$, $(2n,2n)$, $(2\ell,2\ell)$
and $(2m,2m)$.
In case (iii) those signatures are $(0,3)$, $(0,3)$,
$(4(n-a_1)-4b_1(n-2a_1),4a_1+4b_1(n-2a_1))$,
$(4(\ell-a_2)-4b_2(\ell-2a_2),4a_2+4b_2(\ell-2a_2))$
and $(4(m-a_3)-4b_3(m-2a_3),4a_3+4b_3(m-2a_3))$.
In case (iv) those signatures are $(1,2)$, $(2,1)$, $(4n,4n)$ and $(2m,2m)$.
In case (v) those signatures are $(3,0)$, $(0,3)$, $(4n,4n)$ and $(4a,4m-4a)$.
In case (vi) those signatures are $(3,0)$, $(0,3)$, $(4n,4n)$ and $(4m-4a,4a)$.
\medskip

\centerline{
{\footnotesize
\fbox{\setlength{\unitlength}{.50 mm}
\begin{picture}(200,28)
\put(0,12){(20)}
\put(23,20){$\gs\gp(n+1)$}
\put(55,20){$\gs\gp(2)$}
\put(85,20){$\gs\gp(m+1)$}
\put(23,5){$\gs\gp(n)$}
\put(45,5){$\gs\gp(1)$}
\put(63,5){$\gs\gp(1)$}
\put(85,5){$\gs\gp(m)$}
\put(27,10){\line(0,1){8}}
\put(45,10){\line(-1,1){8}}
\put(51,10){\line(1,1){8}}
\put(70,10){\line(-1,1){8}}
\put(75,10){\line(1,1){8}}
\put(93,10){\line(0,1){8}}
\put(125,12){$\frac{Sp(n+1)\times Sp(2)\times Sp(m+1)}{Sp(n)\times Sp(1)\times Sp(1)\times Sp(m)}$}
\end{picture}
}
}
}

The real form family members defined by involutive inner automorphisms of $G_u$
are straightforward now.  If $m = n$ we also have the automorphism $\alpha$
that is the interchange $Sp(n+1) \leftrightarrow Sp(m+1)$ and preserves
$Sp(2)$.  Then $Sp(2)$ goes to a real form of $Sp(2;\C)$ that
contains $Sp(1;\C)$ as a symmetric subgroup. Again making use of Lemma \ref{switch}, the result is
\begin{equation}\label{C8}
\begin{aligned}
{\rm (i)\,\,} & [Sp(n+1;\C)\times Sp(2;\C)\times Sp(m+1;\C)]
	/[Sp(n;\C)\times Sp(1;\C)\times Sp(1;\C)\times Sp(m;\C) ] \\
{\rm (ii)\,\,} &[Sp(n+1;\R)\times Sp(2;\R)\times Sp(m+1;\R)]
	/[Sp(n;\R)\times Sp(1;\R)\times Sp(1;\R) \times Sp(m;\R)] \\
{\rm (iii)\,\,} & [Sp(n-a_1+b_1,1-b_1+a_1)\times Sp(1,1)\times
		Sp(m-a_2+b_2,1-b_2+a_2)] \\
	& \phantom{XXX}/[Sp(n-a_1,a_1)\times Sp(1) \times Sp(1)
		\times Sp(m-a_2,a_2)]\\
        & \phantom{XXXXXXXXXX}
	\text{ where }0\leqq a_1\leqq n,\, 0\leqq a_2\leqq m,
		\text{ and } b_1,\,b_2\in \{0,1\} \\
{\rm (iv)\,\,} & [Sp(n-a_1+b_1,1-b_1+a_1)\times Sp(2)\times
		Sp(m-a_2+b_2,1-b_2+a_2)] \\
	& \phantom{XXX}/[Sp(n-a_1,a_1)\times Sp(1) \times Sp(1)
		\times Sp(m-a_2,a_2)]\\
        & \phantom{XXXXXXXXXX}
	\text{ where } 0\leqq a_1\leqq n,\, 0\leqq a_2\leqq m,\,
		\text{ and } b_1,\, b_2\in \{0,1\} \\
{\rm (v)\,\,} & [Sp(n+1;\C)\times Sp(2;\R)]/[Sp(n;\C)\times Sp(1;\C)]
	\text{ where } m = n \\
{\rm (vi)\,\,} & [Sp(n+1;\C)\times Sp(1,1)]/[Sp(n;\C)\times Sp(1;\C)]
	\text{ where } m = n \\
\end{aligned}
\end{equation}

In case (i), the metric irreducible subspaces of the real tangent space have
signatures $(3,3)$, $(3,3)$, $(4n,4n)$, $(4,4)$ and $(4m,4m)$.
In case (ii), those signatures are $(2,1)$, $(2,1)$,
$(2n,2n)$, $(2,2)$ and $(2m,2m)$.
In case (iii), those signatures are $(0,3)$, $(0,3)$,
$(4(n-a_1)-4b_1(n-2a_1),4a_1+4b_1(n-2a_1))$, $(4,0)$, and
$(4(m-a_2)-4b_2(m-2a_2),4a_2+4b_2(m-2a_2))$.
In case (iv), those signatures are $(0,3)$, $(0,3)$,
$(4(n-a_1)-4b_1(n-2a_1),4a_1+4b_1(n-2a_1))$, $(0,4)$, and
$(4(m-a_2)-4b_2(m-2a_2),4a_2+4b_2(m-2a_2))$.
In case (v), those signatures are $(3,3)$, $(4n,4n)$ and $(3,1)$.
In case (vi), those signatures are $(3,3)$, $(4n,4n)$ and $(1,3)$.

We summarize the computations for $G$ semisimple but not simple,
except for item (21), in the following table.  After the table we
will discuss item (21).

\addtocounter{equation}{1}
{\tiny
\begin{longtable}{|l|l|}
\caption*{\large \bf Table \thetable \quad Weakly Symmetric Pseudo-Riemannian
Homogeneous Spaces $G/H$, \\ \phantom{XXXXXXXX}$G/H$ Not Symmetric, 
$G$ Semisimple But Not Simple, $H$ Reductive}\\
\hline
{\normalsize $G/H$} & {\normalsize metric--irreducibles}  \\
\hline
\hline
\endfirsthead
\multicolumn{2}{l}{\textit{table continued from previous page $\dots$}} \\
\hline
{\normalsize $G/H$} & {\normalsize metric--irreducibles} \\
\hline
\hline
\endhead
\hline \multicolumn{2}{r}{\textit{$\dots$ table continued on next page}} \\
\endfoot
\hline
\endlastfoot
\multicolumn{2}{|c|} {\footnotesize \bf (13) Real Form Family of
$[SU(n)\times SU(n+1)]/S[U(n)\times U(1)]$}\\
\hline
{$[SL(n;\C) \times SL(n+1;\C)]/[SL(n;\C) \times GL(1;\C)]$}
	&  \begin{tabular}{l} $(n^{2}-1,n^{2}-1)$, $(2n,2n)$ \end{tabular}
	 \\
\hline
{$[SL(n;\R) \times SL(n+1;\R)]/[SL(n;\R) \times GL(1;\R)]$}
	& \begin{tabular}{l} $\bigl(\frac{n(n+1)}{2}-1,\frac{n(n-1)}{2}\bigr)$, $(n,n)$ \end{tabular}
	\\
\hline
{$[SU(k,\ell) \times SU(k+1,\ell)]/[SU(k,\ell) \times U(1)]$}
	& \begin{tabular}{l} $(2kl,k^2+l^2-1)$, $(2l,2k)$ \end{tabular}
	\\
\hline
{$[SU(k,\ell) \times SU(k,\ell+1)]/[SU(k,\ell) \times U(1)]$ }
& \begin{tabular}{l} $(2kl,k^2+l^2-1)$, $(2k,2l)$ \end{tabular}
 \\
\hline
\hline
\multicolumn{2}{|c|} {\footnotesize \bf (14) Real Form Family of $[Sp(n+2)\times Sp(2)]/[Sp(n)\times Sp(2)]$}\\
\hline
{$[Sp(n+2;\C)\times Sp(2;\C)]/[Sp(n;\C)\times Sp(2;\C)]$} &
	\begin{tabular}{l} $(10,10)$, $(8n,8n)$ \end{tabular} \\
\hline
{$[Sp(n+2;\R)\times Sp(2;\R)]/[Sp(n;\R)\times Sp(2;\R)]$}
	& \begin{tabular}{l} $(6,4)$, $(4n,4n)$ \end{tabular}
		\\
\hline
{$[Sp(n-a+b,2-b+a)\times Sp(b,2-b)]/[Sp(n-a,a)\times Sp(b,2-b)]$}
	& \begin{tabular}{l} $(8b-4b^2,4b^2-8b+10)$
\\
	$(8n+4(a-n)b+4a(b-2),4(n-a)b+4a(2-b))$ \end{tabular}
	 \\
\hline
\hline

\multicolumn{2}{|c|} {\footnotesize \bf (15) Real Form Family of $[SO(n)\times SO(n+1)]/SO(n)$}\\
\hline
{$[SO(n;\C)\times SO(n+1;\C)]/SO(n,\C)$} & \begin{tabular}{l} $\bigl(\frac{n(n-1)}{2},\frac{n(n-1)}{2}\bigr)$, $(n,n)$ \end{tabular}
	\\
\hline
{$[SO(n-a,a)\times SO(n-a,a+1)]/SO(n-a,a)$}
	& \begin{tabular}{l} $\bigl((n-a)a,\frac{n(n-1)}{2}-(n-a)a\bigr)$, $(n-a,a)$ \end{tabular}\\
\hline
{$[SO(n-a,a)\times SO(n-a+1,a)]/SO(n-a,a)$}
	& \begin{tabular}{l} $\bigl((n-a)a,\frac{n(n-1)}{2}-(n-a)a\bigr)$, $(a,n-a)$ \end{tabular}
	 \\
\hline
\hline
\multicolumn{2}{|c|} {\footnotesize \bf (16) Real Form Family of $[SU(n+2)\times Sp(m+1)]/[U(n)\times SU(2)\times Sp( m)]$}\\
\hline
{$[SL(n+2;\C)\times Sp(m+1;\C)]/[GL(n;\C)\times SL(2;\C)\times Sp(m;\C)]$} & \begin{tabular}{l} $(3,3)$, $(4n,4n)$, $(4m,4m)$ \end{tabular}
	\\
\hline
{$[SL(n+2;\R)\times Sp(m+1;\R)]/[GL(n;\R)\times SL(2;\R)\times Sp(m;\R)]$}
	&\begin{tabular}{l} $(2,1)$, $(2n,2n)$, $(2m,2m)$ \end{tabular}\\
\hline
{$[SL(n'+1;\H)\times Sp(m-a,1+a)]/[GL(n';\H)\times SU(2)\times Sp(m-a,a)]$}
	& \begin{tabular}{l} $(0,3)$, $(2n,2n)$, $(4m-4a,4a)$ \end{tabular}
	 \\
\hline
{$[SL(n'+1;\H)\times Sp(m-a+1,a)]/[GL(n';\H)\times SU(2)\times Sp(m-a,a)]$} &
	\begin{tabular}{l} $(0,3)$, $(2n,2n)$, $(4a,4m-4a)$ \end{tabular}\\
\hline
{\footnotesize $\tfrac{SU(n-a_1+b_1,2-b_1+a_1)\times Sp(m-a_2,1+a_2)}{U(n-a_1,a_1)\times SU(2)\times Sp(m-a_2,a_2)}$}
	& \begin{tabular}{l}
$(0,3)$ \\
$(4(n-a_1)-2b_1(n-2a_1),2b_1(n-2a_1)+4a_1)$
\\
	$(4m-4a_2,4a_2)$ \end{tabular}\\
\hline
{\footnotesize $\tfrac{SU(n-a_1+b_1,2-b_1+a_1)\times Sp(m-a_2+1,a_2)}{U(n-a_1,a_1)\times SU(2)\times Sp(m-a_2,a_2)}$}
	& \begin{tabular}{l}
$(0,3)$ \\
$(4(n-a_1)-2b_1(n-2a_1),2b_1(n-2a_1)+4a_1)$
\\
	$(4a_2,4m-4a_2)$ \end{tabular}
	 \\
\hline
{$[SU(n+1-a,a+1)\times Sp(m+1;\R)]/[U(n-a,a)\times SU(1,1)\times Sp(m;\R)]$}
	 & \begin{tabular}{l} $(2,1)$, $(2n,2n)$, $(2m,2m)$ \end{tabular} \\
\hline
{$[SL(4;\R)\times Sp(m+1;\C)]/[SL(2;\C)\times Sp(m;\C) \times T]$}
	&\begin{tabular}{l} $(3,3)$, $(6,2)$, $(4m,4m)$ \end{tabular}\\
\hline
{$[SU^*(4)\times Sp(m+1;\C)]/[SL(2;\C)\times Sp(m;\C)\times T]$}
	& \begin{tabular}{l} $(3,3)$, $(2,6)$, $(4m,4m)$ \end{tabular}
	 \\
\hline
{$[SU(2,2)\times Sp(m+1;\C)]/[SL(2;\C)\times Sp(m;\C) \times \R]$} & \begin{tabular}{l} $(3,3)$, $(4,4)$, $(4m,4m)$ \end{tabular}
	\\
\hline
\hline
\multicolumn{2}{|c|} {\footnotesize \bf (17) Real Form Family of $[SU(n+2)\times Sp(m+1)]/[SU(n)\times SU(2)\times Sp( m)]$}\\
\hline
{$[SL(n+2;\C)\times Sp(m+1;\C)]/[SL(n;\C)\times SL(2;\C)\times Sp(m;\C)]$} & \begin{tabular}{l} $(3,3)$, $(1,0)$, $(0,1)$, $(4n,4n)$, $(4m,4m)$ \end{tabular}
	\\
\hline
{$[SL(n+2;\R)\times Sp(m+1;\R)]/[SL(n;\R)\times SL(2;\R)\times Sp(m;\R)]$}
	&\begin{tabular}{l} $(2,1)$, $(1,0)$, $(2n,2n)$, $(2m,2m)$ \end{tabular}\\
\hline
{$[SL(n'+1;\H)\times Sp(m-a,1+a)]/[SL(n';\H)\times SU(2)\times Sp(m-a,a)]$}
	& \begin{tabular}{l} $(0,3)$, $(1,0)$, $(2n,2n)$, $(4m-4a,4a)$ \end{tabular}
	 \\
\hline
{$[SL(n'+1;\H)\times Sp(m-a+1,a)]/[SL(n';\H)\times SU(2)\times Sp(m-a,a)]$} &
	\begin{tabular}{l} $(0,3)$, $(1,0)$, $(2n,2n)$, $(4a,4m-4a)$ \end{tabular}\\
\hline
{\footnotesize $\tfrac{SU(n-a_1+b_1,2-b_1+a_1)\times Sp(m-a_2,1+a_2)}{SU(n-a_1,a_1)\times SU(2)\times Sp(m-a_2,a_2)}$}
	& \begin{tabular}{l}
$(0,3)$, $(0,1)$ \\
$(4(n-a_1)-2b_1(n-2a_1),2b_1(n-2a_1)+4a_1)$
\\
	$(4m-4a_2,4a_2)$ \end{tabular}\\
\hline
{\footnotesize $\tfrac{SU(n-a_1+b_1,2-b_1+a_1)\times Sp(m-a_2+1,a_2)}{SU(n-a_1,a_1)\times SU(2)\times Sp(m-a_2,a_2)}$}
	& \begin{tabular}{l}
$(0,3)$, $(0,1)$ \\
$(4(n-a_1)-2b_1(n-2a_1),2b_1(n-2a_1)+4a_1)$
\\
	$(4a_2,4m-4a_2)$ \end{tabular}
	 \\
\hline
{$[SU(n+1-a,a+1)\times Sp(m+1;\R)]/[SU(n-a,a)\times SU(1,1)\times Sp(m;\R)]$}
	 & \begin{tabular}{l} $(2,1)$, $(0,1)$, $(2n,2n)$, $(2m,2m)$ \end{tabular} \\
\hline
{$[SL(4;\R)\times Sp(m+1;\C)]/[SL(2;\C)\times Sp(m;\C)]$}
	&\begin{tabular}{l} $(3,3)$, $(0,1)$, $(6,2)$, $(4m,4m)$ \end{tabular}\\
\hline
{$[SU^*(4)\times Sp(m+1;\C)]/[SL(2;\C)\times Sp(m;\C)]$}
	& \begin{tabular}{l} $(3,3)$, $(0,1)$ $(2,6)$, $(4m,4m)$ \end{tabular}
	 \\
\hline
{$[SU(2,2)\times Sp(m+1;\C)]/[SL(2;\C)\times Sp(m;\C)]$} & \begin{tabular}{l} $(3,3)$, $(1,0)$, $(4,4)$, $(4m,4m)$ \end{tabular}
	\\
\hline
\hline
\multicolumn{2}{|c|} {\footnotesize \bf (18) Real Form Family of $[Sp(n+1)\times Sp(m+1)]/[Sp(n)\times Sp(m)\times Sp(1)]$}\\
\hline
{$[Sp(n+1;\C)\times Sp(m+1;\C)]/[Sp(n;\C)\times Sp(m;\C) \times Sp(1;\C)]$} & \begin{tabular}{l} $(3,3)$, $(4n,4n)$, $(4m,4m)$ \end{tabular}
	\\
\hline
{$[Sp(n+1;\R)\times Sp(m+1;\R)]/[Sp(n;\R)\times Sp(m;\R) \times Sp(1;\R)]$}
	&  \begin{tabular}{l} $(2,1)$, $(2n,2n)$, $(2m,2m)$ \end{tabular}
		\\
\hline
{\footnotesize $\tfrac{Sp(n-a_1+b_1,1-b_1+a_1)\times Sp(m-a_2+b_2,1-b_2+a_2)}{\{Sp(n-a_1,a_1) \times Sp(m-a_2,a_2)\times Sp(1)}$}
	& \begin{tabular}{l} $(0,3)$\\
$(4(n-a_1)-4b_1(n-2a_1),4a_1+4b_1(n-2a_1))$
\\
	$(4(m-a_2)-4b_2(m-2a_2),4a_2+4b_2(m-2a_2))$ \end{tabular}
	 \\
\hline
{$Sp(n+1;\C)/[Sp(n;\C)\times Sp(1)]$ where $m = n$} & \begin{tabular}{l} $(3,0)$, $(4n,4n)$ \end{tabular}\\
\hline
{$Sp(n+1;\C)/[Sp(n;\C)\times Sp(1;\R)]$ where $m = n$} & \begin{tabular}{l} $(1,2)$, $(4n,4n)$ \end{tabular}\\
\hline
\hline
\multicolumn{2}{|c|} {\footnotesize \bf (19) Real Form Family of $[Sp(n+1)\times Sp(\ell+1)\times Sp(m+1)]/[Sp(n)\times Sp(\ell)\times Sp(m)\times Sp(1)]$}\\
\hline
{\footnotesize $\tfrac{Sp(n+1;\C)\times Sp(l+1;\C)\times Sp(m+1;\C)}{Sp(n;\C)\times Sp(l;\C)\times Sp(m;\C) \times Sp(1;\C)}$} &
 \begin{tabular}{l} $(3,3)$, $(3,3)$, $(4n,4n)$, $(4l,4l)$, $(4m,4m)$ \end{tabular}
	\\
\hline
{\footnotesize $\tfrac{Sp(n+1;\R)\times Sp(l+1;\R)\times Sp(m+1;\R)}{Sp(n;\R)\times Sp(l;\R)\times Sp(m;\R) \times Sp(1;\R)}$}
	& \begin{tabular}{l} $(2,1)$, $(2,1)$, $(2n,2n)$, $(2l,2l)$, $(2m,2m)$ \end{tabular}
		\\
\hline
{\footnotesize $\tfrac{Sp(n-a_1+b_1,1-b_1+a_1)\times Sp(l-a_2+b_2,1-b_2+a_2)\times Sp(m-a_3+b_3,1-b_3+a_3)}{Sp(n-a_1,a_1)\times Sp(l-a_2,a_2) \times Sp(m-a_3,a_3)\times Sp(1)}$}
	& \begin{tabular}{l} $(0,3)$, $(0,3)$, \\
$(4(n-a_1)-4b_1(n-2a_1),4a_1+4b_1(n-2a_1))$
\\
	$(4(l-a_2)-4b_2(l-2a_2),4a_2+4b_2(l-2a_2))$ \\
$(4(m-a_3)-4b_3(m-2a_3),4a_3+4b_3(m-2a_3))$
\end{tabular}
	 \\
\hline
{$[Sp(n+1;\C) \times Sp(m+1;\R)]
	/[Sp(n;\C)\times Sp(1;\R)\times Sp(m;\R)]$ where $n = \ell$} &
    \begin{tabular}{l} $(1,2)$, $(2,1)$, $(4n,4n)$, $(2m,2m)$ \end{tabular} \\
\hline
{$[Sp(n+1;\C) \times Sp(m+1-a,a)]/
	[Sp(n;\C)\times Sp(1)\times Sp(m-a,a)]\,, n=\ell$} &  
	\begin{tabular}{l} $(3,0)$, $(0,3)$, $(4n,4n)$, $(4a,4m-4a)$ 
	\end{tabular} \\
\hline
{$[Sp(n+1;\C) \times Sp(m-a,a+1)]/
	[Sp(n;\C)\times Sp(1)\times Sp(m-a,a)]\,, n=\ell$} & 
	\begin{tabular}{l} $(3,0)$, $(0,3)$, $(4n,4n)$, $(4m-4a,4a)$ 
	\end{tabular} \\
\hline
\hline
\multicolumn{2}{|c|} {\footnotesize \bf (20) Real Form Family of $[Sp(n+1)\times Sp(2)\times Sp(m+1)]/[Sp(n)\times Sp(1)\times Sp(1)\times Sp(m)]$}\\
\hline
{\footnotesize $\tfrac{Sp(n+1;\C)\times Sp(2;\C)\times Sp(m+1;\C)}{Sp(n;\C)\times Sp(l;\C)\times Sp(1;\C) \times Sp(m;\C)}$} &
\begin{tabular}{l}  $(3,3)$, $(3,3)$, $(4n,4n)$, $(4,4)$, $(4m,4m)$ \end{tabular}
	\\
\hline
{\footnotesize $\tfrac{Sp(n+1;\R)\times Sp(2;\R)\times Sp(m+1;\R)}{Sp(n;\R)\times Sp(l;\R)\times Sp(1;\R) \times Sp(m;\R)}$}
	&\begin{tabular}{l}  $(2,1)$, $(2,1)$, $(2n,2n)$, $(2,2)$, $(2m,2m)$ \end{tabular}
		\\
\hline
{\footnotesize $\tfrac{Sp(n-a_1+b_1,1-b_1+a_1)\times Sp(1,1)\times Sp(m-a_2+b_2,1-b_2+a_2)}{Sp(n-a_1,a_1)\times Sp(1) \times Sp(1) \times Sp(m-a_2,a_2)}$}
	& \begin{tabular}{l} $(0,3)$, $(0,3)$, $(4,0)$ \\
$(4(n-a_1)-4b_1(n-2a_1),4a_1+4b_1(n-2a_1))$
\\
	$(4(m-a_2)-4b_2(m-2a_2),4a_2+4b_2(m-2a_2))$
\end{tabular}
	 \\
\hline
{\footnotesize $\tfrac{Sp(n-a_1+b_1,1-b_1+a_1)\times Sp(2)\times Sp(m-a_2+b_2,1-b_2+a_2)}{Sp(n-a_1,a_1)\times Sp(1) \times Sp(1) \times Sp(m-a_2,a_2)}$}
	& \begin{tabular}{l} $(0,3)$, $(0,3)$, $(0,4)$ \\
$(4(n-a_1)-4b_1(n-2a_1),4a_1+4b_1(n-2a_1))$
\\
	$(4(m-a_2)-4b_2(m-2a_2),4a_2+4b_a2(m-2a_2))$
\end{tabular}
	 \\
\hline
{$[Sp(n+1;\C)\times Sp(2;\R)]/[Sp(n;\C)\times Sp(1;\C)]
	\text{ where } m = n$} & \begin{tabular}{l} $(3,3)$, $(4n,4n)$, $(3,1)$ \end{tabular}
\\
\hline
{$[Sp(n+1;\C)\times Sp(1,1)]/[Sp(n;\C)\times Sp(1;\C)]
	\text{ where } m = n$} & \begin{tabular}{l} $(3,3)$, $(4n,4n)$, $(1,3)$ \end{tabular} \\
\hline
\end{longtable}
}

\centerline{
{\footnotesize
\fbox{\setlength{\unitlength}{.50 mm}
\begin{picture}(240,28)
\put(0,12){(21)}
\put(38,20){$\gg_1 \ \ \ \ \dots \ \ \ \ \gg_n$}
\put(38,5){$\gh'_1 \ \ \ \ \dots \ \ \ \ \gh'_n$}
\put(40,10){\line(0,1){8}}
\put(52,12){$\dots$}
\put(70,10){\line(0,1){8}}
\put(20,5){$\gz_\gh$}
\put(25,10){\line(3,2){10}}
\put(26,6){\line(3,1){40}}
\put(100,10){\{$(G_{u,1}/H'_{u,1})\times \dots \times
	(G_{u,n}/H'_{u,n})\}/diag(Z_{H_u})$}
\end{picture} \\
}
}
}
\smallskip

Case (21) requires some discussion.
To pass to the group level we assume that the semisimple groups
$G_{u,i}$ are compact and simply connected, so $G_u = \prod G_{u,i}$
also is compact and simply connected, that the
central subgroup $Z_\gh$ of $H$ is connected, and that $H_u$ is
connected.  Thus $M_u = G_u/H_u$ is simply connected
Let $H_{u,i}$ be the projection of $H_u$ to $G_{u,i}$, say
$H_{u,i} = H'_{u,i}Z_{u,i}$ where $Z_{u,i} = p_i(Z_\gh)$ is the
projection of $Z_\gh$ to $H_u$\,.  Then $M_{u,i} = G_{u,i}/H_{u,i}$ is
weakly symmetric with non--semisimple isotropy $H_{u,i}$\,.  Further,
each $M_{u,i}$ is symmetric, or is the complexification of $M_{u,i}$, or is 
one of the spaces of Cases (1) 
through (20) of Tables 3.6 and 4.12.  Combining these requirements, each
$M_{u,i} = G_{u,i}/H_{u,i}$ is one of the following:
\begin{itemize}
   \item a compact irreducible hermitian symmetric space, or
   \item one of the spaces of cases (5), (8), (11) or (12) in Table 3.6, or
   \item one of the spaces of cases (13) or (16) in Table 4.12.
\end{itemize}
Thus either $M_i = G_i/H_i$ is on Berger's list of pseudo--riemannian symmetric 
spaces, or it is listed under Case (5), (8), (11) or (12) in Table 3.6, or
it is listed under Case (13) or (16) in Table 4.12.

Let $M = G/H$ be a pseudo--riemannian weakly symmetric space with the same
complexification as $M_u = G_u/H_u$\,.  Then $M$ corresponds to an involutive
automorphism $\sigma$ of $\gg_u$ that preserves $\gh_u$\,.  It necessarily
preserves $\gz_\gh$ as well.  Now permute the simple factors 
$\gg_{u,i}$ of $\gg_u$ so that $\sigma$ exchanges $\gg_{u,2i-1}$ and
$\gg_{u,2i}$ for $2i \leqq s$ and preserves each $\gg_{u,i}$ for 
$s < i \leqq s+t$.  For $j = 2i \leqq s$ we then have $(\gg_{j,\C},\gh_{j,\C})$
corresponding to indices $(2i-1,i)$, and for $i > s$ we have $(\gg_i,\gh_i)$
where $\gg_i$ (resp. $\gh_i$, resp. $\gz_i)$ is a real form of $\gg_{u,i,\C}$
(resp $\gh_{u,i,\C}$, resp. $\gz_{u,i,\C})$.  It is implicit here that
$\sigma$ preserves the center $Z_{H_u}$ of $H_u$ so that $Z_{H_u}$ is a
subgroup of the center $\widetilde{Z_{H_u}} = \prod Z_{H_{u,i}}$ of
$\widetilde{H_u} = \prod H_{u,i}$.  Thus $H_u \subset \widetilde{H_u}$ 
and we have 
\begin{equation}\label{fibrations1}
\varphi_u: M_u = G_u/H_u \to G_u/\widetilde{H}_u = \widetilde{M}_u 
	\text{ by } gH_u \mapsto g\widetilde{H}_u
	\text{ where } \widetilde{M}_u = \prod M_{u,i} \text{ and } 
	\widetilde{H}_u = \prod H_{u,i}\,.
\end{equation}
Since everything is $\sigma$--stable here, $H \subset \widetilde{H}$
where $\widetilde{H} = \prod H_i$ and we have a well defined projection
\begin{equation}\label{fibrations2}
\varphi: M = G/H \to G/\widetilde{H} = \widetilde{M} 
        \text{ by } gH \mapsto g\widetilde{H}
        \text{ where } \widetilde{M} = \prod M_i \text{ and } 
        \widetilde{H} = \prod H_i\,.
\end{equation}

Conversely, let $M_i = G_i/H_i$ be irreducible weakly symmetric 
pseudo--riemannian
manifolds, not all symmetric, where each $G_i$ is semisimple but each $H_i$ has
center $Z_i$ of dimension $1$.  Thus each $Z_i^0$ is a circle group or the
multiplicative group of positive reals or (if $G_i$ is complex) the 
multiplicative group $\C^*$, and each $H_i = H'_iZ_i^0$ with $H'_i$ 
semisimple.  Suppose that $G = G_1 \times\dots \times G_\ell$
has a Cartan involution $\theta$ that preserves each $G_i$\,, each $H_i$\,
and thus each $Z_i^0$\,. Then we have the compact real forms
$$
G_u = \prod G_{u,i}\,\,\,\,, \widetilde{H_u} =  \prod H_{u,i}\,\,\,\,,
\widetilde{Z_u^0}=\prod Z_{u,i}^0\,\,\,\,, \text{ and } \,\,\,
	\widetilde{M}_u = \prod M_{u,i}\,
$$
where $M_{u,i} = G_{u,i}/H_{u,i}$ and $Z_{u,i}$ is the center of $H_{u,i}$\,.
Consider the set $\cS$ of all closed connected $\theta$--invariant subgroups 
$S_u \subset \widetilde{Z_u}$
such that the projections $S_u \to Z_{u,i}^0$ all are surjective.  The set
$\cS$ is nonempty -- it contains $\widetilde{Z_u^0}$ -- so it has 
elements of minimal dimension.  Let $Z_u$ denote one
of them and define $H = (\prod H'_i)Z_u$\,.  Then $M = G/H$ belongs to the
real form family of Case (21).  This constructs every element in that 
real form family.
\smallskip

Following \cite[Proposition 12.8.4]{W2007} and Tables 3.6 and 4.12, 
the metric signature of the weakly symmetric pseudo--riemannian 
manifold $G/H$ in the real form family of Case (21) is given as
follows.  First, we have the metric irreducible subspaces $S_{i,j}$ of
the real tangent space of $G_i/H_i$, and their signatures $(a_{i,j},b_{i,j})$.
That gives us the metric irreducible subspaces, with their signatures, for
$G/\widetilde{H}$.  To this collection we add the metric irreducible subspaces
of the fiber $\widetilde{\gh}/\gh$
of $\widetilde{\gh} \to \gh$ implicit in (\ref{fibrations2}).

\section{Special Signatures: Riemannian, Lorentz, and Trans--Lorentz.}
\label{sec5}
\setcounter{equation}{0}

We go through Berger's classification \cite{B1957} and our
Tables 3.6 and 4.12 to pick out the cases where $M = G/K$
can have an invariant weakly symmetric pseudo--riemannian metric of 
signature $(n,0)$, $(n-1,1)$ or $(n-2,2)$.  Of course this gives the 
classification of the weakly symmetric pseudo--riemannian manifolds
of those signatures with $G$ semisimple and $H$ reductive in $G$; they
are certain products $G/H = \prod G_i/H_i$ 
from Berger \cite{B1957} for the pseudo--riemannian symmetric cases 
and from Tables 3.6 and 4.12 for the 
nonsymmetric pseudo--riemannian weakly symmetric cases.

We will refer to $(n,0)$, $(n-1,1)$ and $(n-2,2)$ as {\em special signatures}.
Now we run through the cases of Table 3.6, then the cases of Table 4.12,
and finally the symmetric cases from \cite{B1957}.
\smallskip

\centerline{\bf From Table 3.6.}

{\bf Case (1):} Since $m>n\geqq 1$ we know $mn\geqq 2$. Then, of the first 
three cases, only $SL(3;\R)/SL(2;\R)$ can have special signature; 
it is $(3,2)$.  

For the fourth case of Case (1), 
$\tfrac{SU(m-k+\ell,n-\ell+k)}{SU(m-k,k)\times SU(\ell,n-\ell)}$, both
$2m\ell+2nk-4k\ell=2(m-k)\ell+2(n-\ell)k$ and 
$2mn-2m\ell-2nk+4k\ell=2(m-k)(n-\ell)+2k\ell$ are even, so it is enough to 
see when one of them is $0$ or $2$.
If $2(m-k)\ell+2(n-\ell)k=0$, then $2(m-k)\ell=0$ and $2(n-\ell)k=0$, so 
$\ell = 0$ or $k = m$, and $k = 0$ or $\ell = n$.  
If $k = \ell = 0$, or if $k = m$ and $\ell = n$, then the metric irreducibles 
have signatures $(2mn,0)$ and $(0,1)$; the other two cases of $(k,\ell)$
trivialize $M$.  That leaves us with $\tfrac{SU(m,n)}{SU(m)\times SU(n)}$,
which has invariant metrics of signatures $(2mn+1,0)$ and $(2mn,1)$.

If $2(m-k)\ell+2(n-\ell)k=2$, then $2(m-k)\ell=0$ and 
$2(n-\ell)k=2$, or $2(m-k)\ell=2$ and $2(n-\ell)k=0$. 
If $2(m-k)\ell = 2$ then $(m-k)\ell = 1$, and either $n = \ell$ or $k = 0$;
if $k = 0$ then $m = \ell = 1$ and we have 
$\tfrac{SU(2,n-1)}{SU(1)\times SU(1,n-1)}$
if $n = \ell$ then $(m-k) = \ell = 1$ and we have
$\tfrac{SU(2,k)}{SU(1,k)\times SU(1)}$\,.

Since $m\geqq 2$, we 
then have $k=n=1$ and $\ell=0$, or $k=m-1$ and $\ell=n=1$; then
$SU(m-1,2)/SU(m-1,1)$ has invariant metric
of signature $(2m-1,2)$. If $2(m-k)(n-\ell)+2k\ell=0$, then $k=0$ and 
$\ell=n$, or $\ell=0$ and $k=m$. As expected this shows 
that $SU(m+n)/SU(m)\times SU(n)$ 
has metrics of signatures $(2mn+1,0)$ and $(2mn,1)$. 
If $2(m-k)(n-\ell)+2k\ell=2$, then $k=\ell=n=1$, or $\ell=0,\, k=m-1,\,n=1$. 
Then $SU(m,1)/SU(m-1,1)$ has a metric of signature $(2m-1,2)$.  Summarizing,
$$
\begin{aligned}
&SL(3;\R)/SL(2;\R):\,\, (3,2) \\
&SU(m+n)/[SU(m)\times SU(n)]:\,\, (2mn+1,0),\,\, (2mn,1) \\
&SU(m,n)/[SU(m)\times SU(n)]:\,\, (2mn+1,0),\,\, (2mn,1) \\
&SU(n-1,2)/SU(n-1,1):\,\, (2n-1,2) \\
&SU(n,1)/SU(n-1,1):\,\, (2n-1,2)
\end{aligned}
$$

{\bf Case (2):}  Here $n$ is odd and $\geqq 5$ by (\ref{kraemer-classn}).  
The first and fourth cases are excluded because 
$\tfrac{1}{2}n(n-1) \leqq 2$ would 
give $n < 3$, so we only need to discuss the second and third cases.  There 
$(k(k-1)+\ell(\ell-1),2k\ell)$ is the signature.  Since
$k(k-1)+\ell(\ell-1) \geqq \tfrac{1}{2}(k+\ell)^2 - n = \tfrac{n^2}{2}-n > 2$
we are reduced to considering $2k\ell \leqq 2$.
if $k = 0$ then $\ell = n$, $G/H$ is $SO^*(2n)/SU(n)$ or
$SO(2n)/SU(n)$, $n$ odd,  and the possible signatures are $(n(n-1)+1,0)$
and $(n(n-1),1)$.  It is the same for $\ell = 0$.  Now we may suppose 
$k\ell > 0$; so $k = \ell = 1$ because $2k\ell \leqq 2$.  So $n = k+\ell = 2$.  
But $n$ is odd.  Summarizing, we have
$$
SO^*(2n)/SU(n),\,\, SO(2n)/SU(n):\,\, (n(n-1)+1,0),\,\, (n(n-1),1) 
$$

{\bf Case (3):} From Table 3.6, the spaces $E_6/Spin(10) \text{ and } 
E_{6,D_5T_1}/Spin(10)$ have invariant metrics of special signatures
only for signatures $(33,0)$ and $(32,1)$.

{\bf Case (4):} We may assume $n \geqq 2$, so the  first three cases of Case (4)
are excluded.  For the fourth, 
$\tfrac{SU(2n+1-2\ell,2\ell)}{Sp(n-\ell,\ell)}$, we need $\ell = n$ or
$\ell = 0$, leading to signatures $(2n^2+3n-1,1)$ and $(2n^2+3n,0)$.
Summarizing,
$$
\tfrac{SU(2n+1)}{Sp(n)},\,\, \tfrac{SU(2n,1)}{Sp(n)}:\,\,
	(2n^2+3n-1,1),\,\, (2n^2+3n,0).
$$

{\bf Case (5):}  As above, $n\geqq 2$, and that excludes the first three cases
of Case (5).  For the fourth, 
$\tfrac{SU(2n+1-2\ell,2\ell)}{Sp(n-\ell,\ell)\times U(1)}$, we need
$\ell = n$ or $\ell = 0$, leading to special signature $(2n^2+3n-1,0)$.  
Summarizing,
$$
SU(2n+1)/[Sp(n)\times U(1)], \,\, 
SU(2n,1)/[Sp(n)\times U(1)] :\,\, (2n^2+3n-1,0).
$$

{\bf Case (6):} The space $Spin(7)/G_2$ has an invariant  metric of 
special signature $(7,0)$.

{\bf Case (7):} The space $G_2/SU(3)$ has an invariant  metric of 
special signature $(6,0)$, and the space $G_{2,A_1A_1}/SU(1,2)$ has an 
invariant metric of special signature $(4,2)$.

{\bf Case (8):} The spaces $SO(10)/[Spin(7)\times SO(2)]$ and 
$SO(8,2)/[Spin(7)\times SO(2)]$ each has an invariant metric of special 
signature $(23,0)$.

{\bf Case (9):} The spaces $SO(9)/Spin(7)$ and $SO(8,1)/Spin(7)$ have 
invariant metrics of signatures $(15,0)$.

{\bf Case (10):} The spaces $Spin(8)/G_2$ and $Spin(7,1)/G_2$ have 
invariant metrics of signatures $(14,0)$.


{\bf Case (11):} Here $n\geqq 2$.  That excludes the first and fourth
cases of Case (11).  For the second case, if $k=0$ or $k=n$, the signatures 
of the real tangent space irreducibles are $(2n,0)$ and $(n^2-n,0)$, so the
spaces $SO(2n+1)/U(n)\text{ and } SO(2n,1)/U(n)$ have invariant metrics
of special signature $(n^2+n,0)$; and $SO(5)/U(2)\text{ and } SO(4,1)/U(2)$ 
have invariant metrics of special
signature $(4,2)$. If $k=1$ or $k=n-1$, the signatures of the irreducibles
are $(2,2n-2)$ and $(2n-2,(n-1)(n-2))$, leading to $n = 2$ where
$SO(3,2)/U(1,1)$ has metrics of special
signature $(4,2)$.   If $1 < k < n-1$ there is no invariant metric of
special signature.  Summarizing,
$$
\begin{aligned}
& SO(2n+1)/U(n),\,\, SO(2n,1)/U(n):\,\, (n^2+n,0) \\
&SO(5)/U(2),\,\, SO(4,1)/U(2),\,\, SO(3,2)/U(1,1):\,\, (4,2)
\end{aligned}
$$

{\bf Case (12):} Here $n \geqq 3$.  That excludes the first and fourth
cases of Case (12).  In the second and third cases we exclude the range
$0 < k < n-1$ where $4k, 4n-4k-4 \geqq 4$.  There, for $k = 0$ and
$k = n-1$ we note that 
$Sp(n)/[Sp(n-1)\times U(1)]\text{ and } Sp(n-1,1)/[Sp(n-1)\times U(1)]$
have metrics of special signatures $(4n-2,0)$ and $(4n-4,2)$.
\smallskip

\centerline{\bf From Table 4.12.}

{\bf Case (13):} Here $n\geqq 2$.  That excludes the first and second cases 
of Case (13).  It also excludes the possibility $k\ell \ne 0$ in the third 
and fourth cases.  That leaves 
$G/H = \tfrac{SU(n)\times SU(n+1)}{SU(n)\times U(1)}$ and
$G/H = \tfrac{SU(n)\times SU(n,1)}{SU(n)\times U(1)}$, which have
invariant metrics of special signature $(n^2+2n-1,0)$.

{\bf Case (14):} Here $n\geqq 1$.  That excludes the first two 
cases of Case (14).  In the third case, $8b - 4b^2$ implies $b \leqq 2$,
and $b = 1$ is excluded because of a metric irreducible $(4,6)$.
For $b=0$ and $b = 2$, the signatures of the metric irreducibles are 
$(0,10)$ and $(8n-8a,8a)$, so $a=0$ or $a=n$, leading to 
$$
[Sp(n,2)\times Sp(2)]/[Sp(n)\times Sp(2)]\text{ and } 
	[Sp(n+2)\times Sp(2)]/[Sp(n)\times Sp(2)]
$$ 
with invariant metric of special signature $(8n+10,0)$. 

{\bf Case (15):}  Here $n\geqq 3$ since $G$ is semisimple.  That excludes
the first case of Case (15).  In the second and third cases $a = 0$ and
$a = n$ lead to $[SO(n)\times SO(n,1)]/SO(n) \text{ and } 
[SO(n)\times SO(n+1)]/SO(n)$ with invariant metric of special signature
$(\frac{n(n+1)}{2},0)$, and the cases $a=1$ and $a=n-1$ lead only to
$[SO(2,1)\times SO(2,2)]/SO(2,1)\text{ and } 
[SO(2,1)\times SO(3,1)]/SO(2,1)$ with invariant metric of special signature
$(4,2)$.  The cases $1 < a < n-1$ do not lead to special signature.

{\bf Case (16):} Here $n+m\geqq 1$.  The first, second, seventh, eighth,
ninth and tenth cases of Case (16) are excluded at a glance, reducing the 
discussion to the third, fourth, fifth and sixth cases.  The third and
fourth require $n = 0$ and then further require $a = 0$ or $a = m$, leading to
$[Sp(1)\times Sp(m,1)]/[Sp(m)\times Sp(m)]$ and 
$[Sp(1)\times Sp(m+1)]/[Sp(m)\times Sp(m)]$
with invariant metrics of special signature $(3+4m,0)$.  These are in
fact included in the fifth and sixth cases.  For the fifth and sixth cases,
we must have $a_1=0$ or $a_1=n$, and $a_2=0$ or $a_2=m$. Then we
arrive at the spaces 
$$
\tfrac{SU(n+2)\times Sp(m+1)}{U(n)\times SU(2)\times Sp(m)},\,\,
\tfrac{SU(n,2)\times Sp(m+1)}{U(n)\times SU(2)\times Sp(m)},\,\,
\tfrac{SU(n+2)\times Sp(m,1)}{U(n)\times SU(2)\times Sp(m)},\,\,
\tfrac{SU(n,2)\times Sp(m,1)}{U(n)\times SU(2)\times Sp(m)},
$$ 
which have invariant metrics of special signature $(4n+4m+3,0)$. 

{\bf Case (17):} This essentially is a simplification of Case (16).  By
the considerations there, we have that the spaces 
$$
\tfrac{SU(n+2)\times Sp(m+1)}{SU(n)\times SU(2)\times Sp(m)},\,\,
\tfrac{SU(n,2)\times Sp(m+1)}{SU(n)\times SU(2)\times Sp(m)},\,\, 
\tfrac{SU(n+2)\times Sp(m,1)}{SU(n)\times SU(2)\times Sp(m)},\,\, 
\tfrac{SU(n,2)\times Sp(m,1)}{SU(n)\times SU(2)\times Sp(m)}
$$ 
have invariant metrics of special signatures $(4n+4m+4,0)$ and $(4n+4m+3,1)$.

{\bf Case (18):}  Here $n+m\geq 1$. The first, second, fourth and fifth 
cases of Case (18) are excluded at a glance, so we only need to consider 
the third case. There, the signatures of the irreducibles are $(0,3)$, 
$(4(n-a_1)-4b_1(n-2a_1),4a_1+4b_1(n-2a_1))$ and 
$(4(m-a_2)-4b_2(m-2a_2),4a_2+4b_2(m-2a_2))$, 
where $b_1,b_2=\{0,1\}$.  Thus we must have $a_1=0$ or $a_1=n$, and 
$a_2=0$ or $a_2=m$. That brings us to the spaces 
$$
\tfrac{Sp(n+1)\times Sp(m+1)}{Sp(n)\times Sp(1)\times Sp(m)},\,\,
\tfrac{Sp(n,1)\times Sp(m+1)}{Sp(n)\times Sp(1)\times Sp(m)},\,\,
\tfrac{Sp(n,1)\times Sp(m,1)}{Sp(n)\times Sp(1)\times Sp(m)},\,\,
\tfrac{Sp(n+1)\times Sp(m,1)}{Sp(n)\times Sp(1)\times Sp(m)},
$$ 
which have invariant metrics of special signature $(4n+4m+3,0)$.

{\bf Case (19):} The first case of Case (19) is excluded at a 
glance.  Visibly, the second and fourth cases require $\ell=m=n=0$, where
$G/H$ is $[Sp(1;\R)\times Sp(1;\R)\times Sp(1;\R)]/Sp(1;\R)$  or
$[Sp(1;\C)\times Sp(1;\R)]/Sp(1;\R)$; they have
invariant metrics of special signature $(4,2)$.   The fifth and sixth
cases require $n = \ell = 0$ and $a = 0$ or $a = m$, leading to
$$
[Sp(1;\C) \times Sp(m+1)]/[Sp(1)\times Sp(m)] \text{ and } 
[Sp(1;\C) \times Sp(m,1)]/[Sp(1)\times Sp(m)]
$$
which have invariant metrics of special signature $(4m+6,0)$.

For the third case, we must have $a_1=0$ or $a_1=n$, $a_2=0$ or $a_2=\ell$, 
and $a_3=0$ or $a_3=m$. In other words, $G/H$ must be one of 
$$\tfrac{Sp(n+1)\times Sp(\ell+1)\times Sp(m+1)}{Sp(n)\times Sp(\ell) \times Sp(m)\times Sp(1)},
\tfrac{Sp(n+1)\times Sp(\ell+1)\times Sp(m,1)}{Sp(n)\times Sp(\ell) \times Sp(m)\times Sp(1)},
\tfrac{Sp(n+1)\times Sp(\ell,1)\times Sp(m+1)}{Sp(n)\times Sp(\ell) \times Sp(m)\times Sp(1)},
\tfrac{Sp(n+1)\times Sp(\ell,1)\times Sp(m,1)}{Sp(n)\times Sp(\ell) \times Sp(m)\times Sp(1)},$$
$$\tfrac{Sp(n,1)\times Sp(\ell+1)\times Sp(m+1)}{Sp(n)\times Sp(\ell) \times Sp(m)\times Sp(1)},
\tfrac{Sp(n,1)\times Sp(\ell+1)\times Sp(m,1)}{Sp(n)\times Sp(\ell) \times Sp(m)\times Sp(1)},
\tfrac{Sp(n,1)\times Sp(\ell,1)\times Sp(m+1)}{Sp(n)\times Sp(\ell) \times Sp(m)\times Sp(1)},
\tfrac{Sp(n,1)\times Sp(\ell,1)\times Sp(m,1)}{Sp(n)\times Sp(\ell) \times Sp(m)\times Sp(1)},$$
These all have invariant metrics of special signature $(4n+4\ell+4m+6,0)$. 

{\bf Case (20):} The first, second, fifth and sixth cases are excluded
at a glance.  For the third and fourth cases, we must have $a_1=0$ or $a_1=n$, 
and $a_2=0$ or $a_2=m$. Then the spaces 
$$
\tfrac{Sp(n+1)\times Sp(1,1)\times Sp(m+1)}	
	{Sp(n)\times Sp(1)\times Sp(1)\times Sp(m)},\,\,
\tfrac{Sp(n,1)\times Sp(1,1)\times Sp(m+1)}
	{Sp(n)\times Sp(1)\times Sp(1)\times Sp(m)},\,\,
\tfrac{Sp(n,1)\times Sp(1,1)\times Sp(m,1)}
	{Sp(n)\times Sp(1)\times Sp(1)\times Sp(m)},\,\, 
\tfrac{Sp(n+1)\times Sp(1,1)\times Sp(m,1)}
	{Sp(n)\times Sp(1)\times Sp(1)\times Sp(m)}
$$ 
$$
\tfrac{Sp(n+1)\times Sp(2)\times Sp(m+1)}
	{Sp(n)\times Sp(1)\times Sp(1)\times Sp(m)},\,\,
\tfrac{Sp(n,1)\times Sp(2)\times Sp(m+1)}
	{Sp(n)\times Sp(1)\times Sp(1)\times Sp(m)},\,\,
\tfrac{Sp(n,1)\times Sp(2)\times Sp(m,1)}
	{Sp(n)\times Sp(1)\times Sp(1)\times Sp(m)},\,\,
\tfrac{Sp(n+1)\times Sp(2)\times Sp(m,1)}
	{Sp(n)\times Sp(1)\times Sp(1)\times Sp(m)}
$$ 
have invariant metrics of special signature $(4n+4m+10,0)$.

\smallskip

\centerline{\bf From Berger's \cite[Table II]{B1957}.}

The irreducible pseudo-riemannian symmetric spaces $G/H$ of \cite{B1957} fall
into two classes: the real form families for which $G_u$ is simple and those 
for the compact group manifolds $G_u = L_u \times L_u$ where $H_u$ is the 
diagonal $\delta L_u = \{(x,x)\mid x \in L_u\}$.
First consider the group manifolds. There the real tangent space of $G/H$ is
$\gm = \{(\xi,-\xi) \mid \xi \in \gl\}$ and the invariant pseudo-riemannian
metrics come from multiples of the Killing form of $\gl$.  
Thus $G/H =  (L\times L)/diag(L)$ has an invariant pseudo-riemannian metric
of special signature if and only if (i) the Killing form of $\gl$ 
is definite, or (ii) the Killing form
of $\gl$ has signature $\pm(\dim\gl -1,1)$, or (iii) the Killing form
of $\gl$ has signature $\pm(\dim\gl -2,2)$.  

The case (i) is the case where $G/H$ is a compact simple group manifold
with bi-invariant metric.
The cases (ii) and (iii) occurs only for the group manifold $SL(2;\R)$ (up 
to covering); that group manifold has bi-invariant metrics of signatures
$(2,1)$ and $(1,2)$.

For the moment we put the group manifold cases
aside and consider the cases where $G_u$ is simple.  Start with the compact
simple classical groups: $SU(n)$ for $n\geqq 2$, $Sp(n)$ for $n\geqq 2$, and 
$SO(n)$ for $n\geqq 7$.

For $G_u = SU(n)$, $n\geqq 2$, we have the following cases:
\begin{enumerate}
   \item $SL(n;\R)/SO(n)$ and $SU(n)/SO(n)$ with signature $(\frac{n^2+n}{2}-1,0)$.
   \item $SL(2;\C)/SO(2;\C)$ with signature $(2,2)$.
   \item $SL(2;\R)/\R$ with signature $(1,1)$.
   \item $SL(2;\C)/SL(2;\R)(\text{ or } SL(2;\C)/SU(1,1))$ with signature $(2,1)$.
   \item $SU^*(2n)/Sp(n)$ and $SU(2n)/Sp(n)$ with signature $(2n^2-n-1,0)$.
   \item $SL(2;\C)/SU^*(2)$ with signature $(3,0)$.
   \item $SU(m,n)/S(U(m)\times U(n))$ and $SU(m+n)/S(U(m)\times U(n))$ 
	with signature $(2mn,0)$.
   \item $SL(n;\C)/SU(n)$ with signature $(n^2-1,0)$.
   \item $SL(3;\R)/[SL(2;\R)\times R]$ with signature $(2,2)$.
   \item $SL(4;\R)/Sp(2;\R)$ with signature $(3,2)$.
   \item $SL(4;\R)/GL'(2;\C)$ with signature $(6,2)$.
   \item $SU^*(4)/Sp(1,1)$ with signature $(4,1)$.
   \item $SU^*(4)/GL'(2;\C)$ with signature $(6,2)$.
   \item $SU(2,1)/SO(2,1)$ with signature $(3,2)$.
   \item $SU(2,2)/Sp(2;\R)$ with signature $(3,2)$.
   \item $SU(2,2)/Sp(1,1)$ with signature $(4,1)$.
   \item We discuss the case 
	$SU(m-a+b,n-b+a)/S(U(m-a,a)\times U(n-b,b))$ 
	with signature $(2mn-2(m-a)b-2(n-b)a,2(m-a)b+2(n-b)a)=
	2(m-a)(n-b)+2ab,2(m-a)b+2(n-b)a)$. If $(m-a)(n-b)+ab=0$, or 
	$(m-a)b+(n-b)a=0$, we are in the riemannian  case (7) just above. If 
	$(m-a)(n-b)+ab=1$, or $(m-a)b+(n-b)a=1$, we have 
	$SU(n-1,2)/U(n-1,1)$ and 
	$SU(n,1)/U(n-1,1)$ with signature $(2n-2,2)$.
\end{enumerate}

For $G_u = SO(n)$, $n\geqq 7$, we have the following cases:
\begin{enumerate}
  \item $SO^*(2n)/U(n)$ and $SO(2n)/U(n)$ with 
signature $(n^2-n,0)$.
  \item $SO(m,n)/[SO(m)\times SO(n)]$ and $SO(m+n)/[SO(m)\times SO(n)]$ with 
signature $(mn,0)$.
   \item $SO(n;\C)/SO(n)$ with signature $(\frac{n^2-n}{2},0)$.
   \item  $SO(n-3,3)/[SO(n-3,1)\times SO(2)]$ and 
	$SO(n-1,1)/[SO(n-3,1)\times SO(2)]$ with signature $(2n-6,2)$.
   \item Finally (for $SO(n)$) we discuss the case 
$SO(m-a+b,n-b+a)/[SO(m-a,a)\times SO(n-b,b)]$ with signature
$(mn-(m-a)b-(n-b)a=(m-a)(n-b)+ab,(m-a)b+(n-b)a)$.  We need to see when one 
of the above two numbers in the signature is 0, 1, or 2.

If $(m-a)(n-b)+ab=0$, or $(m-a)b+(n-b)a=0$, we are in case (2) of $SO(n)$ 
just above.
If $(m-a)(n-b)+ab=1$ or $(m-a)b+(n-b)a=1$, we have $SO(n-1,2)/SO(n-1,1)$ and
$SO(n,1)/SO(n-1,1)$ with invariant metric of signature $(n-1,1)$.
The discussion for these cases is similar to case (17) for $SU(n)$
because the equations are the same.  

Now we consider the cases where $(m-a)(n-b)+ab=2$ or $(m-a)b+(n-b)a=2$.

First let $(m-a)b+(n-b)a=2$. If $(m-a)b = 1$ then $a=b=1$ and $m=n=2$,
contradicting our assumption $n \geqq 7$.  Thus either $(m-a)b=0$ and 
$(n-b)a=2$, or $(m-a)b=2$ and $(n-b)a=0$.  Then we have the following solutions:
\begin{alignat*}{2}
&(1)\,\, m=a=1,n=b+2\qquad &&(2)\,\, m=a=2,n=b+1 \\
&(3)\,\, b=0,n=1,a=2\qquad &&(4)\,\, b=0,n=2,a=1 \\
&(5)\,\, a=0,m=1,b=2\qquad &&(6)\,\, a=0,m=2,b=1 \\
&(7)\,\, n=b=1,m-a=2\qquad &&(8)\,\, n=b=2,m-a=1
\end{alignat*}
\end{enumerate}
We may assume $m \leqq n$.  As $m+n \geqq 7$ the solutions are (1), (2),
(5) and $6$.  That leads us to
$$
SO(n-1,3)/[SO(n-1,1)\times SO(2)]:\, (2n-2,2) \text{ and }
	SO(n-2,3)/SO(n-2,2):\, (n-2,2).
$$
A similar discussion of the case $(m-a)(n-b)+ab=2$ leads to
$$
SO(n+1,1)/[SO(n-1,1)\times SO(2)]:\, (2n-2,2) \text{ and }
	SO(n-1,2)/SO(n-2,2):\,\, (n-2,2).
$$

For $G_u = Sp(n)$, $n\geqq 2$, we have the following cases:
\begin{enumerate}
 \item $Sp(n;\R)/U(n)$ and $Sp(n)/U(n)$ with signature 
$(n^2+n,0)$.
 \item $Sp(m,n)/[Sp(m)\times Sp(n)]$ and $Sp(m+n)/[Sp(m)\times Sp(n)]$ with 
signature $(4mn,0)$.
 \item $Sp(n;\C)/Sp(n)$ with signature $(2n^2+n,0)$.
 \item $Sp(2;\R)/[Sp(1;\R)\times Sp(1;\R)]$ with signature $(2,2)$.
 \item $Sp(2;\R)/U(1,1)$ and $Sp(1,1)/U(1,1)$ with 
signature $(4,2)$.
 \item $Sp(2;\R)/Sp(1;\C)$ and $Sp(1,1)/Sp(1;\C)$ with signature $(3,1)$.
 \item We discuss the case $Sp(m-a+b,n-b+a)/[Sp(m-a,a)\times Sp(n-b,b)]$ with 
signature $(4mn-4(m-a)b-4(n-b)a,4(m-a)b+4(n-b)a)$. It is enough to discuss 
$4mn-4(m-a)b-4(n-b)a=4(m-a)(n-b)+4ab=0$ or $4(m-a)b+4(n-b)a=0$. It gives 
case (2) just above.
\end{enumerate}

Now we look for special signature in real form families where $G_u$ is a 
compact simple exceptional group.
\begin{enumerate}
\item $G_2^*/[SU(2)\times SU(2)]$ and $G_2/[SU(2)\times SU(2)]$ with signature $(8,0)$.


\item $F_{4,C_3A_1}/[Sp(3)\times SU(2)]$ and $F_4/[Sp(3)\times SU(2)]$ with signature $(28,0)$.

\item $F_{4,B_4}/SO(9)$ and $F_4/SO(9)$ with signature $(16,0)$.


\item $E_{6,C_4}/Sp(4)$ and $E_6/Sp(4)$ with signature $(42,0)$.

\item $E_{6,A_5A_1}/[SU(6)\times SU(2)]$ and $E_6/[SU(6)\times SU(2)]$ with signature $(40,0)$.

\item $E_{6,D_5T_1}/[SO(10)\times T_1]$ and $E_6/[SO(10)\times T_1]$ with signature $(32,0)$.

\item $E_{6,F_4}/F_4$ and $E_6/F_4$ with signature $(26,0)$.


\item $E_{7,A_7}/SU(8)$ and $E_7/SU(8)$ with signature $(70,0)$.

\item $E_{7,D_6A_1}/[SO(12)\times SU(2)]$ and $E_7/[SO(12)\times SU(2)]$ with signature $(64,0)$.

\item $E_{7,E_6T_1}/[E_6\times T_1]$ and $E_7/[E_6\times T_1]$ with signature $(54,0)$.


\item $E_{8,D_8}/SO(16)$ and $E_8/SO(16)$ with signature $(128,0)$.

\item $E_{8,E_7A_1}/[E_7\times SU(2)]$ and $E_8/[E_7\times SU(2)]$ with signature $(112,0)$.
\end{enumerate}

Finally, we tabulate the results according to special signature.  As indicated 
earlier, the semisimple riemannian symmetric spaces are (up to local isometry)
the products of spaces from Table 5.1, the semisimple lorentzian
spaces are (up to local isometry) the products of spaces from Table 5.1 and
one space from Table 5.2, and the semisimple trans--lorentzian  spaces are
(up to local isometry) the products of spaces from Table 5.1 and either one 
space from Table 5.3 or two spaces from Table 5.2.

\addtocounter{equation}{1}
{\tiny
\begin{longtable}{|c|l|l|}
\caption*{\centerline{\large \bf Table \thetable \,
Weakly Symmetric Pseudo--Riemannian $G/H$,}}\\ 
\caption*{\centerline{\large \bf $G$ Semisimple and $H$ Reductive, of Riemannian Signature}}\\
\hline
{\normalsize Type of $\gg$} 
        &{ \normalsize $G/H$: irreducible cases of riemannian signature} 
	& {\normalsize metric signature} \\
\hline
\hline
\endfirsthead
\multicolumn{3}{l}{\textit{table continued from previous page $\dots$}} \\
\hline
{\normalsize Type of $\gg$} 
	&{ \normalsize $G/H$: irreducible cases with riemannian signature} 
	& {\normalsize metric signature} \\
\hline
\hline
\endhead
\hline \multicolumn{2}{r}{\textit{$\dots$ table continued on next page}} \\
\endfoot
\hline
\endlastfoot
\hline
A & $SU(m+n)/[SU(m)\times SU(n)] \text{ and } SU(m,n)/[SU(m)\times SU(n)]$ 
	& $(2mn+1,0)$ \\
\hline
D & $SO(2n)/SU(n)\text{ and } SO^*(2n)/SU(n)$ & $(n(n - 1) + 1, 0)$\\
\hline
E & $E_6/Spin(10) \text{ and } E_{6,D_5T_1}/Spin(10)$ & $(33,0)$\\
\hline
A & $SU(2n+1)/Sp(n) \text{ and } SU(2n,1)/Sp(n)$ & $(2n^2+3n,0)$\\
\hline
A & $SU(2n+1)/[Sp(n)\times U(1)] \text{ and } SU(2n,1)/[Sp(n)\times U(1)]$ & 
        $(2n^2+3n-1,0)$ \\
\hline
B & $Spin(7)/G_2$ &  $(7,0)$ \\
\hline 
G & $G_2/SU(3)$ &  $(6,0)$ \\
\hline
D & $SO(10)/[Spin(7)\times SO(2)] \text{ and } SO(8,2)[Spin(7)\times SO(2)]$ & 
	$(23,0)$ \\
\hline 
B & $SO(9)/Spin(7) \text{ and } SO(8,1)/Spin(7)$ & $(15,0)$ \\
\hline
D & $Spin(8)/G_2 \text{ and } Spin(7,1)/G_2$ & $(14,0)$ \\
\hline
B & $SO(2n+1)/U(n) \text{ and } SO(2n,1)/U(n)$ & $(n^2 + n,0)$ \\
\hline
C & $Sp(n)/[Sp(n-1)\times U(n)] \text{ and } Sp(n-1,1)/[Sp(n-1)\times U(n)]$ & 
	$(4n-2,0)$ \\
\hline
A+A & $[SU(n)\times SU(n+1)]/[SU(n)\times U(1)] \text{ and }
[SU(n)\times SU(n,1)]/[SU(n)\times U(1)]$ & $(n^2+2n-1,0)$ \\
\hline
C+C & $[Sp(n,2)\times Sp(2)]/[Sp(n)\times Sp(2)]\text{ and } 
        [Sp(n+2)\times Sp(2)]/[Sp(n)\times Sp(2)]$ &
	$(8n+10,0)$ \\
\hline
B+D & $[SO(n)\times SO(n,1)]/SO(n) \text{ and } [SO(n)\times SO(n+1)]/SO(n)$ &
	$(\frac{n(n+1)}{2},0)$ \\
\hline
A+C &
$\tfrac{SU(n+2)\times Sp(m+1)}{U(n)\times SU(2)\times Sp(m)},\,
\tfrac{SU(n,2)\times Sp(m+1)}{U(n)\times SU(2)\times Sp(m)},\,
\tfrac{SU(n+2)\times Sp(m,1)}{U(n)\times SU(2)\times Sp(m)},\,
\tfrac{SU(n,2)\times Sp(m,1)}{U(n)\times SU(2)\times Sp(m)}$ &
$(4n+4m+3,0)$ \\
\hline
A+C &
$\tfrac{SU(n+2)\times Sp(m+1)}{SU(n)\times SU(2)\times Sp(m)},\,
\tfrac{SU(n,2)\times Sp(m+1)}{SU(n)\times SU(2)\times Sp(m)},\, 
\tfrac{SU(n+2)\times Sp(m,1)}{SU(n)\times SU(2)\times Sp(m)},\, 
\tfrac{SU(n,2)\times Sp(m,1)}{SU(n)\times SU(2)\times Sp(m)}$ &
$(4n+4m+4,0)$ \\
\hline
C+C &
$\tfrac{Sp(n+1)\times Sp(m+1)}{Sp(n)\times Sp(1)\times Sp(m)},\,
\tfrac{Sp(n,1)\times Sp(m+1)}{Sp(n)\times Sp(1)\times Sp(m)},\,
\tfrac{Sp(n,1)\times Sp(m,1)}{Sp(n)\times Sp(1)\times Sp(m)},\,
\tfrac{Sp(n+1)\times Sp(m,1)}{Sp(n)\times Sp(1)\times Sp(m)}$ &
$(4n+4m+3,0)$ \\
\hline
C+C &
$[Sp(1;\C) \times Sp(m+1)]/[Sp(1)\times Sp(m)] \text{ and } 
[Sp(1;\C) \times Sp(m,1)]/[Sp(1)\times Sp(m)]$ &
$(4m+6,0)$ \\
\hline
C+C+C &
$\begin{smallmatrix}
\tfrac{Sp(n+1)\times Sp(\ell+1)\times Sp(m+1)}{Sp(n)\times Sp(\ell) 
	\times Sp(m)\times Sp(1)},\,
\tfrac{Sp(n+1)\times Sp(\ell+1)\times Sp(m,1)}{Sp(n)\times Sp(\ell) 
	\times Sp(m)\times Sp(1)},\,
\tfrac{Sp(n+1)\times Sp(\ell,1)\times Sp(m+1)}{Sp(n)\times Sp(\ell) 
	\times Sp(m)\times Sp(1)} \\
\tfrac{Sp(n+1)\times Sp(\ell,1)\times Sp(m,1)}{Sp(n)\times Sp(\ell) 
	\times Sp(m)\times Sp(1)},\,
\tfrac{Sp(n,1)\times Sp(\ell+1)\times Sp(m+1)}{Sp(n)\times Sp(\ell) 
	\times Sp(m)\times Sp(1)},\,
\tfrac{Sp(n,1)\times Sp(\ell+1)\times Sp(m,1)}{Sp(n)\times Sp(\ell) 
	\times Sp(m)\times Sp(1)} \\
\tfrac{Sp(n,1)\times Sp(\ell,1)\times Sp(m+1)}{Sp(n)\times Sp(\ell) 
	\times Sp(m)\times Sp(1)},\,
\tfrac{Sp(n,1)\times Sp(\ell,1)\times Sp(m,1)}{Sp(n)\times Sp(\ell) 
	\times Sp(m)\times Sp(1)} \end{smallmatrix}$ & 
$(4n+4\ell+4m+6,0)$ \\
\hline
C+C+C &
$\begin{smallmatrix}
\tfrac{Sp(n+1)\times Sp(1,1)\times Sp(m+1)}     
        {Sp(n)\times Sp(1)\times Sp(1)\times Sp(m)},\,
\tfrac{Sp(n,1)\times Sp(1,1)\times Sp(m+1)}
        {Sp(n)\times Sp(1)\times Sp(1)\times Sp(m)},\,
\tfrac{Sp(n,1)\times Sp(1,1)\times Sp(m,1)}
        {Sp(n)\times Sp(1)\times Sp(1)\times Sp(m)}\\ 
\tfrac{Sp(n+1)\times Sp(1,1)\times Sp(m,1)}
        {Sp(n)\times Sp(1)\times Sp(1)\times Sp(m)},\,
\tfrac{Sp(n+1)\times Sp(2)\times Sp(m+1)}
        {Sp(n)\times Sp(1)\times Sp(1)\times Sp(m)},\,
\tfrac{Sp(n,1)\times Sp(2)\times Sp(m+1)}
        {Sp(n)\times Sp(1)\times Sp(1)\times Sp(m)}\\
\tfrac{Sp(n,1)\times Sp(2)\times Sp(m,1)}
        {Sp(n)\times Sp(1)\times Sp(1)\times Sp(m)},\,
\tfrac{Sp(n+1)\times Sp(2)\times Sp(m,1)}
        {Sp(n)\times Sp(1)\times Sp(1)\times Sp(m)}\end{smallmatrix}$ &
$(4n+4m+10,0)$ \\
\hline
various & see discussion of Case (21) & see Case (21) \\
\hline
A+A & $[SU(n)\times SU(n)]/diag(SU(n))$ and $SL(n;\C)/SU(n)$ & $(n^2-1,0)$ \\
\hline
BD+BD & $[SO(n)\times SO(n)]/diag(SO(n))$ and $SO(n;\C)/SO(n)$ 
	& $(n(n-1)/2,0)$ \\
\hline
C+C & $[Sp(n)\times Sp(n)]/diag(Sp(n))$ and $Sp(n;\C)/Sp(n)$ & $(2n^2+n,0)$ \\
\hline
G+G & $[G_2 \times G_2]/diag(G_2)$ and $G_{2,\C}/G_2$ & $(14,0)$ \\
\hline
F+F & $[F_4 \times F_4]/diag(F_4)$ and $F_{4,\C}/F_4$ & $(52,0)$ \\
\hline
E+E & $[E_6 \times E_6]/diag(E_6)$ and $E_{6,\C}/E_6$ & $(78,0)$ \\
\hline
E+E & $[E_7 \times E_7]/diag(E_7)$ and $E_{7,\C}/E_7$ & $(133,0)$ \\
\hline
E+E & $[E_8 \times E_8]/diag(E_8)$ and $E_{8,\C}/E_8$ & $(248,0)$ \\
\hline
A & $SL(n;\R)/SO(n)$ and $SU(n)/SO(n)$ & $(\frac{n^2+n}{2}-1,0)$ \\
\hline
A & $SU(2n)/Sp(n)$ and $SU(2n)^*/Sp(n)$ & $(2n^2-n-1,0)$ \\
\hline
A & $SU(m+n)/S(U(m)\times U(n))$ and $SU(m,n)/S(U(m)\times U(n))$ & $(2mn,0)$ \\
\hline
D & $SO(2n)/U(n)$ and $SO^*(2n)/U(n)$ & $(n^2-n,0)$ \\
\hline
BD & $SO(m+n)/[SO(m)\times SO(n)]$ and $SO(m,n)/[SO(m)\times SO(n)]$ & 
	$(mn,0)$ \\
\hline
C & $Sp(n)/U(n)$ and $Sp(n;\R)/U(n)$ & $(n^2+n,0)$\\
\hline
C & $Sp(m+n)/[Sp(m)\times Sp(n)]$ and $Sp(m,n)/[Sp(m)\times Sp(n)]$ 
	& $(4mn,0)$ \\
\hline
G & $G_2/[SU(2)\times SU(2)]$ and $G_{A_1A_1}/[SU(2)\times SU(2)]$ & $(8,0)$ \\
\hline
F & $F_4/Spin(9)$ and $F_{4,B_4}/Spin(9)$ & $(16,0)$ \\
\hline
F & $F_4/[Sp(3)\times SU(2)]$ and $F_{4,C_3A_1}/[Sp(3)\times SU(2)]$ & $(28,0)$ \\
\hline
E & $E_6/Sp(4)$ and $E_{6,C_4}/Sp(4)$ & $(42,0)$ \\
\hline
E & $E_6/[SU(6)\times SU(2)]$ and $E_{6,A_5A_1}/[SU(6)\times SU(2)]$ & $(40,0)$ \\
\hline
E & $E_6/[SO(10)\times SO(2)]$ and $E_{6,D_5T_1}/[SO(10)\times SO(2)]$ &
	$(32,0)$ \\
\hline
E & $E_6/F_4$ and $E_{6,F_4}/F_4$ & $(26,0)$ \\
\hline
E & $E_7/SU(8)$ and $E_{7,A_7}/SU(8)$ & $(70,0)$ \\
\hline
E & $E_7/[SO(12)\times SU(2)]$ and $E_{7,D_6A_1}/[SO(12)\times SU(2)]$ & 
	$(64,0)$ \\
\hline
E & $E_7/[E_6\times T_1]$ and $E_{7,E_6T_1}/[E_6\times T_1]$ & $(54,0)$ \\
\hline
E & $E_8/SO(16)$ and $E_{8,D_8}/SO(16)$ & $(128,0)$ \\
\hline
E & $E_8/[E_7\times SU(2)]$ and $E_{8,E_7A_1}/[E_7\times SU(2)]$ & $(112,0)$
\end{longtable}
}

\addtocounter{equation}{1}
{\tiny
\begin{longtable}{|c|l|l|}
\caption*{\centerline{\large \bf Table \thetable \,
Weakly Symmetric Pseudo--Riemannian $G/H$,}}\\
\caption*{\centerline{\large \bf $G$ Semisimple and $H$ Reductive, of Lorentz Signature}}\\
\hline
{\normalsize Type of $\gg$}
        &{ \normalsize $G/H$: irreducible cases of Lorentz signature}
        & {\normalsize metric signature} \\
\hline
\hline
\endfirsthead
\multicolumn{3}{l}{\textit{table continued from previous page $\dots$}} \\
\hline
{\normalsize Type of $\gg$}
        &{ \normalsize $G/H$: irreducible cases with Lorentz signature}
        & {\normalsize metric signature} \\
\hline
\hline
\endhead
\hline \multicolumn{2}{r}{\textit{$\dots$ table continued on next page}} \\
\endfoot
\hline
\endlastfoot
\hline
A & $SU(m+n)/[SU(m)\times SU(n)]$ and $SU(m,n)/[SU(m)\times SU(n)]$ &
	$(2mn,1)$ \\
\hline
D & $SO(2n)/SU(n)$ and $SO^*(2n)/SU(n)$ & $(n(n-1),1)$ \\
\hline
E & $E_6/Spin(10)$ and $E_{6,D_5T_1}/Spin(10)$ & $(32,1)$ \\
\hline
A & $SU(2n+1)/Sp(n)$ and $SU(2n,1)/Sp(n)$ & $(2n^2+3n-1,1)$ \\
\hline
A+C & $\tfrac{SU(n+2)\times Sp(m+1)}{SU(n)\times SU(2)\times Sp(m)},\,
\tfrac{SU(n,2)\times Sp(m+1)}{SU(n)\times SU(2)\times Sp(m)},\,
\tfrac{SU(n+2)\times Sp(m,1)}{SU(n)\times SU(2)\times Sp(m)},\,
\tfrac{SU(n,2)\times Sp(m,1)}{SU(n)\times SU(2)\times Sp(m)}$ &
$(4n+4m+3,1)$ \\
\hline
A+A & group manifold $SL(2;\R) = [SL(2;\R)\times SL(2;\R)]/diag(SL(2;\R))$ 
	viewed as Lorentz manifold & $(2,1)$ \\
\hline
A+A & $SL(2;\C)/SL(2;\R)$ viewed as Lorentz manifold & $(2,1)$ \\
\hline
A & $SL(2;\R)/\R$ & $(1,1)$ \\
\hline
A & $SU^*(4)/Sp(1,1)$ & $(4,1)$ \\
\hline
A & $SU(2,2)/Sp(1,1)$ & $(4,1)$ \\
\hline
BD & $SO(n,1)/SO(n-1,1) \text{ and } SO(n-1,2)/SO(n-1,1)$ & $(n-1,1)$ \\
\hline
C & $Sp(2;\R)/Sp(1;\C) \text{ and } Sp(1,1)/Sp(1;\C)$ &  $(3,1)$ \\
\hline
\end{longtable}
}

\addtocounter{equation}{1}
{\tiny
\begin{longtable}{|c|l|l|}
\caption*{\centerline{\large \bf Table \thetable \,
Weakly Symmetric Pseudo--Riemannian $G/H$,}}\\
\caption*{\centerline{\large \bf $G$ Semisimple and $H$ Reductive, of 
Trans--Lorentz Signature}}\\
\hline
{\normalsize Type of $\gg$}
        &{ \normalsize $G/H$: irreducible cases of Trans--Lorentz signature}
        & {\normalsize metric signature} \\
\hline
\hline
\endfirsthead
\multicolumn{3}{l}{\textit{table continued from previous page $\dots$}} \\
\hline
{\normalsize Type of $\gg$}
        &{ \normalsize $G/H$: irreducible cases with Trans--Lorentz signature}
        & {\normalsize metric signature} \\
\hline
\hline
\endhead
\hline \multicolumn{2}{r}{\textit{$\dots$ table continued on next page}} \\
\endfoot
\hline
\endlastfoot
\hline
A & $SL(3;\R)/SL(2;\R)$ & $(3,2)$ \\
\hline
A & $SU(n-1,2)/SU(n-1,1)$ and $SU(n,1)/SU(n-1,1)$ & $(2n-1,2)$ \\
\hline
G & $G_{2,A_1A_1}/SU(1,2)$ & $(4,2)$\\
\hline
B & $SO(4,1)/U(2)$, $SO(5)/U(2)$ and $SO(3,2)/U(1,1)$ & $(4,2)$ \\
\hline
C & $Sp(n)/[Sp(n-1)\times U(1)]\text{ and } Sp(n-1,1)/[Sp(n-1)\times U(1)]$ &
	$(4n-4,2)$ \\
\hline
BD & $[SO(2,1)\times SO(2,2)]/SO(2,1)\text{ and } 
	[SO(2,1)\times SO(3,1)]/SO(2,1)$ & $(4,2)$ \\
\hline
A+A & group manifold $SL(2;\R) = [SL(2;\R)\times SL(2;\R)]/diag(SL(2;\R))$
        viewed as Trans--Lorentz manifold & $(1,2)$ \\
\hline
A & $SL(2;\C)/SO(2;\C)$ & $(2,2)$ \\
\hline
A & $SL(2;\C)/SL(2;\R) = SL(2;\C)/SU(1,1))$ & $(1,2)$ \\
\hline
A & $SL(3;\R)/[SL(2;\R)\times R]$ & $(2,2)$ \\
\hline
A & $SL(4;\R)/Sp(2;\R)$ & $(3,2)$ \\
\hline
A & $SL(4;\R)/GL'(2;\C)$ and $SU^*(4)/GL'(2;\C)$ 
        & $(6,2)$ \\
\hline
A & $SU(2,1)/SO(2,1)$ & $(3,2)$ \\
\hline
A & $SU(2,2)/Sp(2;\R)$ & $(3,2)$ \\
\hline
A & $SU(n-1,2)/U(n-1,1)$ and $SU(n,1)/U(n-1,1)$ & $(2n-2,2)$ \\
\hline
BD & $SO(n-3,3)/[SO(n-3,1)\times SO(2)]$ and $SO(n-1,1)/[SO(n-3,1)\times SO(2)]$
        & $(2n-6,2)$ \\
\hline
BD & $SO(n-2,3)/SO(n-2,2)$ and $SO(n-1,2)/SO(n-2,2)$ & $(n-2,2)$ \\
\hline
C & $Sp(2;\R)/[Sp(1;\R)\times Sp(1;\R)]$ & $(2,2)$ \\
\hline
C & $Sp(2;\R)/U(1,1)$ and $Sp(1,1)/U(1,1)$  
        &  $(4,2)$ \\
\hline
\end{longtable}
}

\end{document}